\documentclass[11pt,a4paper]{amsart}
\usepackage[english]{babel}
\usepackage[T1]{fontenc} 
\usepackage{mathpazo,amssymb}
\usepackage{upref}
\usepackage{enumerate}
\usepackage{graphicx}
\usepackage{tikz}
\usepackage{cite}
\usepackage{comment}
\usepackage{color}
\usepackage[normalem]{ulem}
\usepackage{mathtools}   
\usepackage{setspace}
\usepackage{dsfont}       

\usepackage[a4paper,includeheadfoot,twoside, margin=2.6cm]{geometry}

\usetikzlibrary{calc,fadings,decorations.pathreplacing}

\DeclareMathOperator*{\osc}{osc}

\tikzset{%
  >=latex, 
  inner sep=0pt,%
  outer sep=2pt,%
  mark coordinate/.style={inner sep=0pt,outer sep=0pt,minimum size=3pt,
    fill=black,circle}%
}

\mathtoolsset{showonlyrefs=true}   

\usepackage[colorlinks=true, pdftitle={Log-concavity},
pdfauthor={Ben Andrews, Julie Clutterbuck
and Daniel Hauer} ]{hyperref}

\title{Non-concavity of the Robin ground state}
  
\author{Ben Andrews} \address[Ben Andrews]{Mathematical Sciences
  Institute, Australian National University, ACT 2601 Australia; and
  Yau Mathematical Sciences Center, Tsinghua University, Beijing
  100084, China.}
\email{\href{mailto:Ben.Andrews@anu.edu.au}{\nolinkurl{Ben.Andrews@anu.edu.au}}}
\thanks{The research of the first author was supported by grants
  DP120102462 and FL150100126 of the Australian Research Council.
  The research of the second author was supported by grant FT1301013 of the Australian Research Council.
  }

\author{Julie Clutterbuck} 
\address[Julie Clutterbuck]{School of Mathematical Sciences, 
  Monash University,  VIC 3800  Australia}
\email{\href{mailto:Julie.Clutterbuck@monash.edu}{\nolinkurl{Julie.Clutterbuck@monash.edu}}}

\author{Daniel Hauer} \address[Daniel Hauer]{School of Mathematics and
  Statistics, The University of Sydney, NSW 2006, Australia}
\email{\href{mailto:daniel.hauer@sydney.edu.au}{\nolinkurl{daniel.hauer@sydney.edu.au}}}



\numberwithin{equation}{section}

\newtheorem{theorem}{Theorem}[section]
\newtheorem{proposition}[theorem]{Proposition}
\newtheorem{lemma}[theorem]{Lemma}
\newtheorem{corollary}[theorem]{Corollary}

\theoremstyle{definition}
\newtheorem{definition}[theorem]{Definition}

\theoremstyle{remark}
\newtheorem{remark}[theorem]{Remark}

\newcommand\R{{\mathbb{R}}}
\renewcommand\S{{\mathbb{S}}}

\newcommand\dx{\,\mathrm{d}x }

\newcommand\dr{\mathrm{d}r }

\newcommand\tr{\mathrm{tr}\,}
\newcommand\dH{\mathrm{d}\mathcal{H}}

\newcommand\abs[1]{\lvert#1\rvert}
\newcommand\labs[1]{\left\lvert#1\right\rvert} 
\newcommand\norm[1]{\lVert#1\rVert}
\newcommand\lnorm[1]{\left\lVert#1\right\rVert}

\definecolor{darkred}{rgb}{0.7,0.1,0.1}

\definecolor{grey}{gray}{0.6}

\usepackage{wrapfig}     

\begin{document}


\maketitle


\begin{abstract} 
  On a convex bounded Euclidean domain, the ground state
  for the Laplacian with Neumann boundary conditions is a constant,
  while the Dirichlet ground state is log-concave.  The Robin
  eigenvalue problem can be considered as interpolating between the
  Dirichlet and Neumann cases, so it seems natural that the Robin
  ground state should have similar concavity
  properties.  In this paper we show that this is false, by analysing
  the perturbation problem from the Neumann case.  In particular we
  prove that on polyhedral convex domains, except in very special
  cases (which we completely classify) the variation of the
  ground state with respect to the Robin parameter is not a concave
  function.  We conclude from this that the Robin ground state
   is not log-concave (and indeed even has some superlevel sets which are non-convex) for small Robin parameter on
  polyhedral convex domains outside a special class, and hence also on
  arbitrary convex domains which approximate these in Hausdorff
  distance.
\end{abstract}

\spacing{1.2}

\tableofcontents

 {
   \hypersetup{hidelinks=true}
      \tableofcontents
 }

\section{Introduction and Main Results}

The Laplacian eigenvalue problem on a bounded convex domain
$\Omega\subset\R^n$ is to find a function $u$ and a constant $\lambda$
satisfying
\begin{equation}
  \label{eigenvalue}
    -\Delta u= \lambda u \text{ in $\Omega$,}
\end{equation}
subject to one of the following boundary conditions:
\begin{align}\notag
 &\text{ Dirichlet:  }& u=0 \text{ on }\partial\Omega,\\ \notag
 & \text{ Neumann: } &D_\nu u=0 \text{ on }\partial\Omega, \\ \label{eq:Robin}
& \text{ or Robin: }&D_\nu u + \alpha u=0  \text{ on }\partial\Omega.
\end{align}
Here $\nu$ is the outward pointing unit normal to $\Omega$, and
$\alpha$ is a real constant.  In this paper we are exclusively
concerned with the case $\alpha\ge0$.  For each of these problems,
there exists an non-decreasing sequence of eigenvalues
\begin{displaymath}
 0\le \lambda_0< \lambda_1 \le \dots \rightarrow \infty.
\end{displaymath}
Our main interest in this paper is in the first Robin eigenvalue
$\lambda_0^{R}(\alpha)$ for $\alpha>0$, and the corresponding
\emph{ground state} $u_{\alpha}$ which is (up to
scaling) the unique eigenfunction with definite sign (which we take to
be positive). The Robin problem~{\eqref{eigenvalue}-\eqref{eq:Robin}}
with $\alpha>0$ is often regarded as interpolating between the
Dirichlet and Neumann cases: if we consider $\alpha$ as a parameter,
the Neumann case corresponds to $\alpha=0$ and the Dirichlet case to
the limit as $\alpha\rightarrow \infty$. In particular, if we write
the eigenvalues for each boundary condition as $\lambda_j^D$,
$\lambda_j^N$, $\lambda_j^R(\alpha)$, then the $j$th Robin eigenvalue
$\lambda^{R}_{j}(\alpha)$ is monotone in $\alpha$, so in particular we
have following monotonicity property:
  \begin{displaymath}
\lambda_j^N\le \lambda_j^R{(\alpha)}\le \lambda_j^D\quad\text{ for
  all $\alpha \ge 0$.}
\end{displaymath}

We are particularly concerned with the shape of the first
eigenfunction $u_\alpha$. In the Neumann case, the first eigenfunction is constant. In the
Dirichlet case, the first eigenfunction $u_{\infty}$ is
\emph{log-concave} (that is, $\log u_{\infty}$ is concave)
\cite{MR0450480}.  Explicit eigenfunctions on rectangular domains show
that this cannot be improved to concavity of the eigenfunction itself.

In the Dirichlet case, the log-concavity of the first eigenfunction is
a key step in proving the lower bound on the gap between $\lambda_0^D$
and $\lambda_1^D$ \cite{SWYY,YuZhong,fundamental}.  Our investigation
of the concavity properties of the ground state was motivated by
possible applications to such a lower bound for the Robin case:
indeed, in those cases where the first Robin eigenfunction is
log-concave, the same proof as in the Dirichlet case applies, implying
the (non-sharp) inequality
\begin{displaymath}
\lambda_1^R(\alpha)-\lambda_0^R(\alpha)\geq
\frac{\pi^2}{D^2},
\end{displaymath}
where $D$ is the diameter of $\Omega$ and $\alpha>0$. We
describe this result in Section \ref{Section gap}.

For some domains, the Robin eigenfunction $u_{\alpha}$ can be found
explicitly and is log-concave. For example, on a ball $\Omega=B_{R}$
of radius $R>0$, 
the first eigenfunction $u_{\alpha}$ is a rotationally symmetric
function $u_{{\alpha}}(r)$ satisfying
\begin{displaymath}
u''_{\alpha}+\frac{d-1}{r}u'_{\alpha}
  +\lambda_1^R(\alpha)\,u_{\alpha}=0
\quad\text{on $[0,R)$,}\quad\text{with }\quad u'_{\alpha}(0)=0.
\end{displaymath}
\newcommand{\va}{v} 
Defining $\va = (\log u_{\alpha})'$, we have  
\begin{equation}
  \label{eq:1}
  \va' = \frac{u_\alpha''}{u_\alpha}-\left(\frac{u_\alpha'}{u_\alpha}\right)^2 
         = -\frac{d-1}{r}\va-\lambda_1^R(\alpha)-\va^2 < -\frac{d-1}{r}v\quad\text{on $[0,R)$} 
\end{equation}
and $\va(0)=0$. Thus, $\va< 0$ on $(0,R)$. Letting $w=v'$ (so that
$w<0$ for small $r$ by \eqref{eq:1}) we find that
\begin{displaymath}
  w'=-\left(\frac{d}{r}+2v\right)w-
  \frac{\lambda_1^R(\alpha)+\va^2}{r}< -\left(\frac{d}{r}+2v\right)w
  \quad\text{on $[0,R)$.} 
\end{displaymath}
It follows that $w<0$ on $[0,R)$.  The eigenvalues of the Hessian of
$\log u_\alpha$ are $(\log u_\alpha)''=w<0$ and
$\frac{(\log u_\alpha)'}{r} = \frac{v}{r}<0$, so $u_\alpha$ is
log-concave.

Another easily computed example is that of
rectangular domains given by products of intervals, where separation
of variables produces the first eigenfunction as a product of concave
trigonometric functions, which is therefore log-concave.

One might expect then that in general, the first Robin eigenfunction
$u_{\alpha}$ with $\alpha>0$ on a convex domain is log-concave, a
question raised by Smits \cite{smits1996spectral}. In this paper we
show that this is not the case: there exist convex domains, and small
values of $\alpha>0$, for which the first Robin eigenfunction
$u_{\alpha}$ fails to be log-concave and has some non-convex
superlevel sets.

Our main result is concerned with convex \emph{polyhedral} domains
$\Omega$ in $\R^{d}$, $d\ge 1$, by which we mean open bounded domains
given by the intersection of finitely many open half-spaces:
\begin{displaymath}
  \Omega = \bigcap_{i=1}^m\Big\{x\in \R^{d}\,\Big\vert\; x\cdot\nu_i< b_i\Big\},
\end{displaymath}
where $\nu_1,\dots,\nu_m$ are unit vectors and $b_1,\dots,b_m$ are
constants. The corresponding \emph{faces} $\Sigma_{i}$ of
$\Omega$ are given by
\begin{displaymath}
  \Sigma_i = \Big\{x\in\overline\Omega\,\Big\vert\; x\cdot\nu_i=b_i\Big\}
\end{displaymath}
and $\nu_{i}$ denotes the outer unit normal to $\Omega$ on the face
$\Sigma_{i}$. The \emph{tangent cone}
$\Gamma_x$ to $\Omega$ at $x\in \overline{\Omega}$ is 
\begin{displaymath}
  \Gamma_x :=  \bigcap_{i\in \mathcal{I}(x)}
  \Big\{y\in\R^{d}\,\Big\vert\; y\cdot\nu_i<0\Big\}\quad
  \text{with index set}\quad
  \mathcal{I}(x):=\Big\{i\in\{1,\dots,m\}\,\Big\vert\, x\cdot\nu_i=b_i\Big\}.
\end{displaymath}
We introduce a special subclass of polyhedral domains, with terminology borrowed from \cite{AM}:

\begin{definition}
  \label{def:exscribed}
  A convex polyhedral domain $\Omega$ in
  $\mathbb{R}^d$ is a \emph{circumsolid} if there exists a ball
  $B_R(x_0)\subset\overline\Omega$ touching every face of $\Omega$
  (that is,
  $\partial{B}_R(x_0)\cap \Sigma_i$ contains exactly one point for every $i\in\{1,\cdots,m\}$).
  Equivalently, $\Omega$ has the form
  \begin{displaymath}
    \Omega = \bigcap_{i=1}^m\Big\{x\in \R^{d}\;\Big\vert\; (x-x_0)\cdot\nu_i<R\Big\}.
  \end{displaymath}
  We say that a convex polyhedron $\Omega$ is a \emph{product of circumsolids} if there is
  a decomposition of $\R^d$ into orthogonal subspaces $E_1,\cdots,E_k$, and circumsolids
  $\Omega_i\subset E_i$ for $i=1,\cdots,k$ such that
  \begin{displaymath}
  \Omega = \Big\{x\in\R^d\;\Big\vert\; \pi_i(x)\in\Omega_i\text{\ for\ }i=1,\dots,k\Big\}
\end{displaymath}
where $\pi_i$ is the orthogonal projection from $\R^d$ onto $E_i$ for
each $i$. Here, circumsolids are trivially products of circumsolids.

We say that a point $x\in \overline{\Omega}$ has \emph{consistent
  normals} if the outward unit normals
$\{\nu_i\,\vert\, i\in{\mathcal I}(x)\}$ to the tangent cone $\Gamma_x$ are
such that there exists a solution $\gamma\in\R^d$ of the system of
equations
\begin{displaymath}
\gamma\cdot\nu_i=-1,\quad i\in{\mathcal I}(x).
\end{displaymath}
Otherwise we say that $x$ has \emph{inconsistent normals}.
Consistency of the normals at $x$ is equivalent to the statement that
the points $\{\nu_i\,\vert\ i\in{\mathcal I}(x)\}$ lie in a hyperplane
disjoint from the origin, or to the statement that the tangent cone
$\Gamma_x$ is an (unbounded) circumsolid (see Proposition
\ref{propo:irregular-bdry}).
 \end{definition}

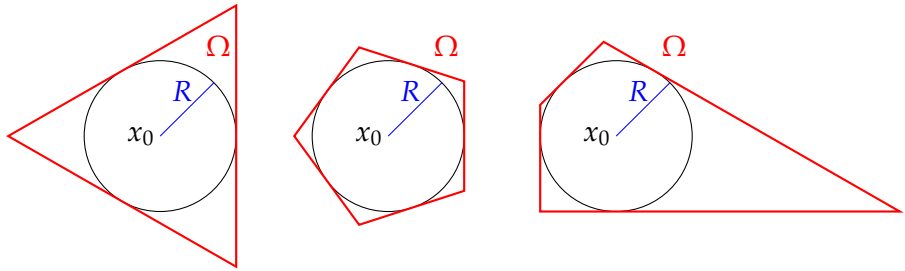
\begin{figure}
\begin{tikzpicture}
\draw (0,0) circle (1);
\node [left] at (0,0) {$x_0$};
\node [left,red] at (1,1.2) {$\Omega$};
\node [left,blue] at (0.5,0.6) {$R$};
\draw [blue] (0,0) -- (0.71,0.71);
\draw [thick, red] (1,1.732) -- 
	(-2,0) -- (1,-1.732) -- cycle;
\draw (3,0) circle (1);
\node [left] at (3,0) {$x_0$};
\node [left,red] at (4,1.2) {$\Omega$};
\node [left,blue] at (3.5,0.6) {$R$};
\draw [blue] (3,0) -- (3.71,0.71);
\draw [thick, red, 
	declare function={
		xvalue(\a,\b) = (sin(\b)-sin(\a))/sin(\b-\a);
		yvalue(\a,\b) = -(cos(\b)-cos(\a))/sin(\b-\a);}
] (4,0.726) -- 
	({3-0.382},1.176) --
	({3-1.236},0) -- ({3-0.382},-1.176) -- (4,-0.726) --
	cycle;
\draw (6,0) circle (1);
\node [left] at (6,0) {$x_0$};
\node [left,red] at (7,1.2) {$\Omega$};
\node [left,blue] at (6.5,0.6) {$R$};
\draw [blue] (6,0) -- (6.71,0.71);
\draw [thick, red, 
	declare function={
		xvalue(\a,\b) = (sin(\b)-sin(\a))/sin(\b-\a);
		yvalue(\a,\b) = -(cos(\b)-cos(\a))/sin(\b-\a);}
] ({6-1},0.414) -- 
	({6-0.1645},1.250) --
	({6+3.732},-1) -- ({6-1},-1) --
	cycle;
\end{tikzpicture}
\caption{Planar circumsolid examples: Regular triangle, regular pentagon, skew quadrilateral}
\label{fig:planar}
\end{figure}

We mention some examples: In one dimension any interval is a
circumsolid. Planar examples include all regular polygons,
such as the triangle and pentagon in Figure \ref{fig:planar}.  However
circumsolids can be non-symmetric, such as the skew quadrilateral in
Figure \ref{fig:planar}.  Every triangle 
is a circumsolid (Figure \ref{fig:triangle}). The same is not true for quadrilaterals:
For the trapezium shown in Figure \ref{fig:trapez} only a specific
spacing between the ends (marked with a dashed line) results in
a circumsolid; a very long trapezium is not a circumsolid.

\begin{figure}[!htb]
 \minipage{7cm}
\begin{tikzpicture}
\draw (0,0) circle (1);\node [left] at (0,0) {$x_0$};
\node [left,red] at (2,0.7) {$\Omega$};
\node [left,blue] at (0.5,0.6) {$R$};
\draw [blue] (0,0) -- (0.71,0.71);
\draw [thick, red, 
	rotate=225
] (1,0.577) -- 
	(-6.078,4.664) --
	(1,-2.414) --
	cycle;
\end{tikzpicture}
\caption{Skew triangle}
\label{fig:triangle}
\endminipage\hspace{2cm}
\minipage{3.9cm}
\mbox{}\quad
\begin{tikzpicture}
  \node [left,red] at (3.3,1.4) {$\Omega$};
  \draw [thick,red] (0,0) -- (1,{sqrt(3)}) -- (3.5,{sqrt(3)}) -- (4,0) -- (0,0);
  \draw [dashed, red] ({3/2+sqrt(3)/2+1/4+1/(4*sqrt(13))},0) -- ({3/2+sqrt(3)/2-1/4+1/(4*sqrt(13))},{sqrt(3)});
  \draw (1.5,{sqrt(3)/2}) circle ({sqrt(3)/2});
\end{tikzpicture}
\caption{Trapezium}
  \label{fig:trapez}
\endminipage
\end{figure}

In higher dimensions any affine simplex is a circumsolid: For any
$d+1$ points $x_0,\dots,x_d$ in $\R^d$ which do not lie in a
$(d-1)$-dimensional subspace, the tetrahedron
$\{\sum_{i=0}^d \lambda_ix_i\,\vert\, \lambda_i\geq 0,\ \sum_i\lambda_i=1\}$
is a circumsolid (Figure \ref{fig:regtet}). 

\begin{figure}[!htb]
\minipage{6.5cm}
\hskip 1 cm\includegraphics[scale=0.7]{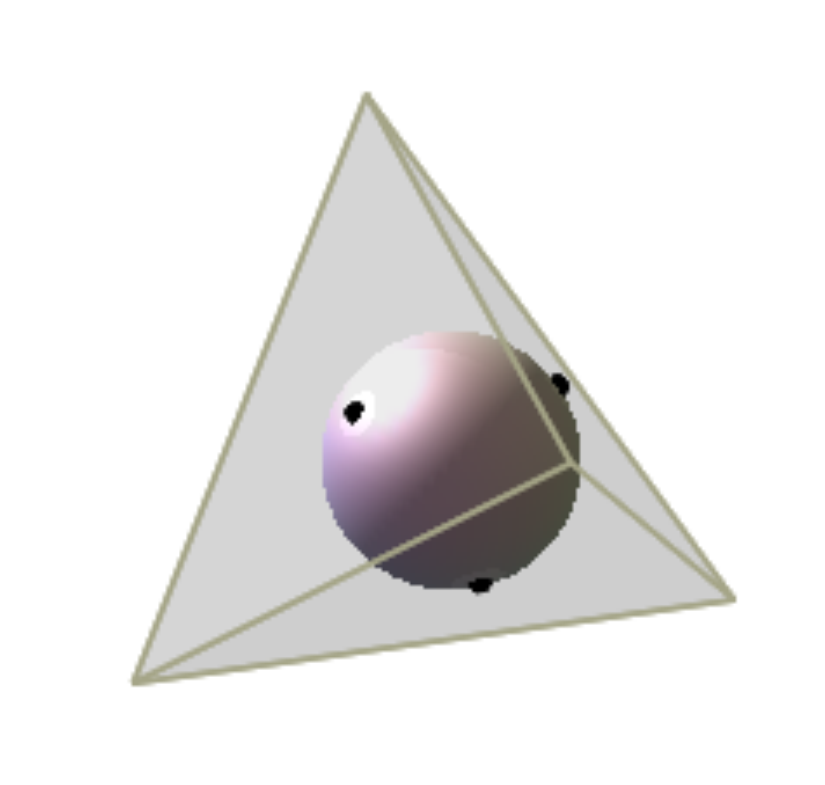}
\caption{Tetrahedron}
\label{fig:regtet}
\endminipage
\minipage{6.5cm}
\hskip 1 cm\includegraphics[scale=0.5]{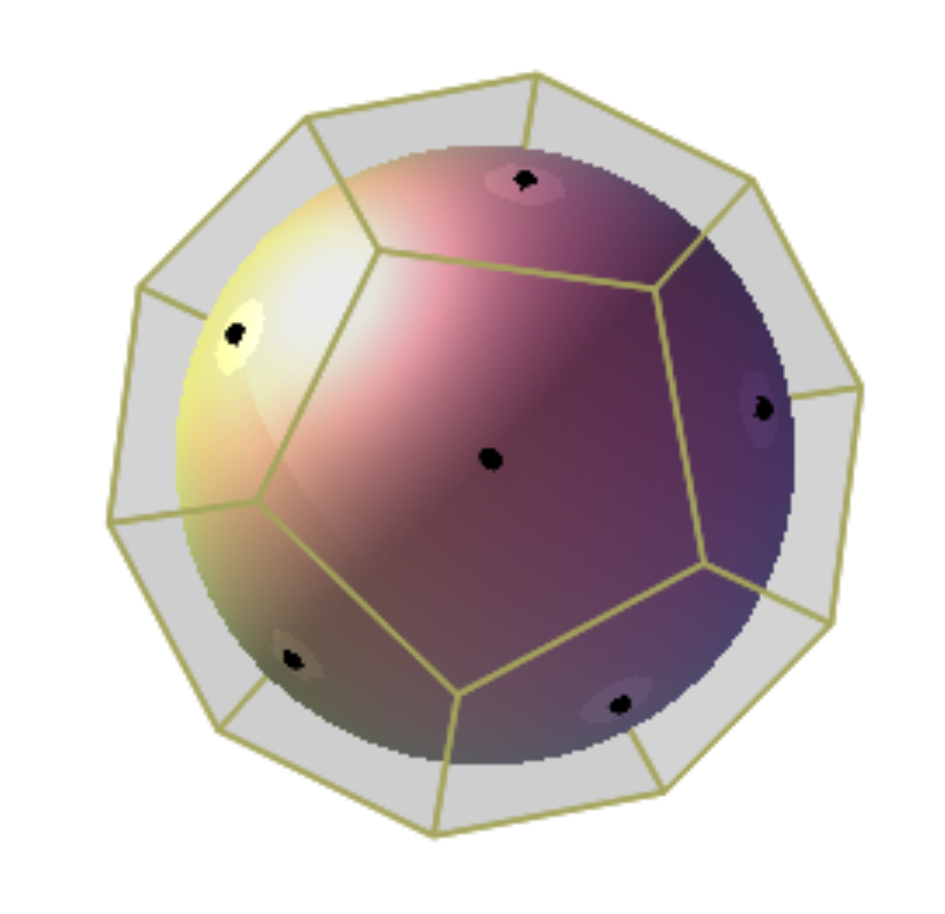}
\caption{Regular dodecahedron}
\label{fig:regdod}
\endminipage
\end{figure}

However, truncating one
of the vertices as in Figure \ref{fig:twft} does not produce a
circumsolid unless the plane of truncation is chosen to match the
inscribed sphere.  Other examples of three-dimensional circumsolids
include the platonic solids and other Archimedean solids (see for
example Figure \ref{fig:regdod}).

\begin{figure}[!htb]
\minipage{6.8cm}
\mbox{}\hspace{1cm}
\begin{tikzpicture}
  \node [left,red] at (0.7,1.4) {$\Omega$};
  \draw [thick,red] (0,2) -- (3,0) -- (4,3);
  \draw [thick,red,dashed] (0,2) -- (4,3);
  \draw [thick,red] (2.5,4) -- (3,5) -- (1.5,4.5) -- (2.5,4);
  \draw [thick,red] (2.5,4) -- (3,0);
  \draw [thick,red] (3,5) -- (4,3);
  \draw [thick,red] (1.5,4.5) -- (0,2);
\end{tikzpicture}
\caption{Tetrahedron with a flat top}
\label{fig:twft}
\endminipage
\minipage{6.4cm}
\mbox{}\quad
\includegraphics[scale=0.5]{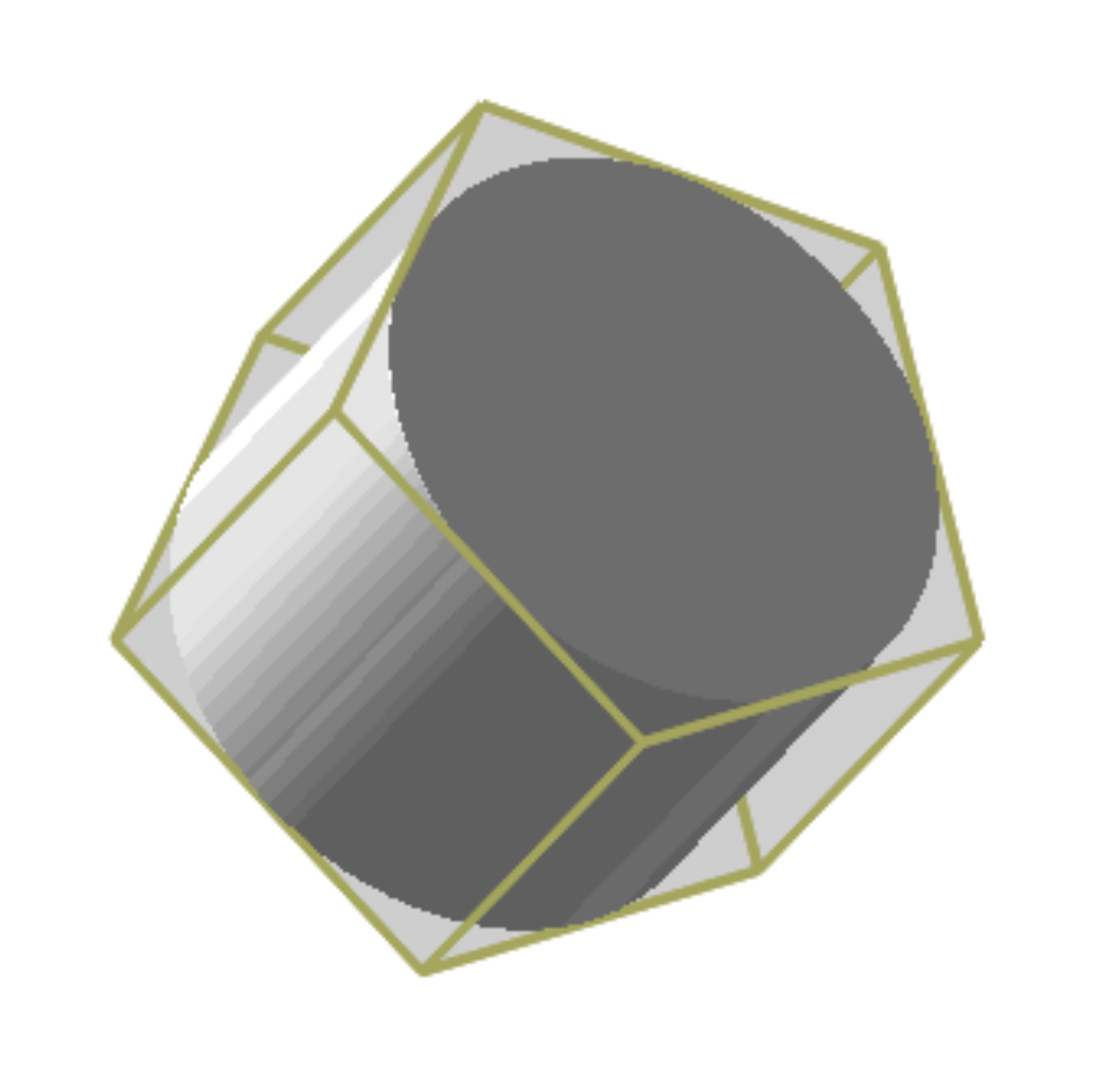}
\caption{Prism over a regular pentagon}
\label{fig:prodcirc}
\endminipage
\end{figure}

In the plane, the only domains which are nontrivial products of
circumsolids are rectangles (products of intervals in orthogonal
one-dimensional subspaces).  In three dimensions, rectangular prisms
(products of three intervals) are productes of circumsolids, as are 
prisms over planar circumsolids, such
as the example in Figure \ref{fig:prodcirc}.

\begin{figure}[!htb]
\minipage{9cm}
\mbox{}\hspace{1.7cm}
\begin{tikzpicture}
  \node [left,red] at (0.7,1.4) {$\Omega$};
   \draw [thick,red] (0,2) -- (3,0) -- (4,3);
   \draw [thick,red,dashed] (0,2) -- (4,3);
   \draw [thick,red] (2.5,4) -- (3,5); 
   \draw [thick,red] (2.5,4) -- (3,0);
   \draw [thick,red] (3,5) -- (4,3);
   \draw [thick,red] (0,2) -- (2.5,4);
   \draw [thick,red] (0,2) -- (3,5); 
   \node [below, black] at (0,1.9) {$x_{0}$};
   \draw [fill] (0,2) circle [radius=0.03];
\end{tikzpicture}
\caption{Tetrahedron with non-horizontal sliced tip}
\label{fig:twftsliced}
\endminipage
\end{figure}
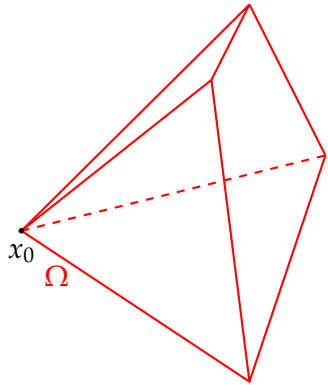

We note that if $\Omega$ is a product of circumsolids then every
boundary point has consistent normals, since we can define $\gamma$ by
$\pi_i(\gamma) = -\frac{1}{R_i}\pi_i(x_0^i-x)$ for $i=1,\cdots,k$,
where $x_0^i$ and $R_i$ are the centre and radius of the circumsolid
$\Omega_i\subset E_i$ for each $i$.   In the plane, every boundary point
of a convex polygon has consistent normals. Figure~\ref{fig:twftsliced} is an
example of a convex polyhedron in $\R^{3}$ with vertex $x_{0}$ having
inconsistent normals.

The following Theorem is the main result of this paper.

\begin{theorem}
  \label{thm:main1}
  Let $\Omega$ be a convex polyhedral domain in $\R^d$, $d\ge 2$,
  which is not a product of circumsolids. Then for sufficiently small
  $\alpha>0$, the first Robin eigenfunction $u_{\alpha}$ is not
  log-concave. 
\end{theorem}

To prove that the first Robin eigenfunction $u_{\alpha}$ admits
non-convex superlevel sets in dimension $d\ge3$, we make the following
stronger assumption:

\begin{theorem}
  \label{thm:main1bis}
  Let $\Omega$ be a convex polyhedral domain in $\R^d$.  If $d=2$ and
  $\Omega$ is not a product of circumsolids, then the first Robin
  eigenfunction $u_{\alpha}$ admits non-convex superlevel sets for
  sufficiently small $\alpha>0$. The same conclusion holds if $d\ge 3$
  and $\Omega$ has boundary points with inconsistent normals.
\end{theorem}

We stress that although Theorem~\ref{thm:main1} is stated for
polyhedral domains, one cannot hope to avoid
such non-concavity results by imposing more regularity on the boundary.

\begin{corollary}
  \label{cor:approx-omega}
  Let $\Omega_{0}$ be a convex polyhedral domain in $\R^d$, $d\ge 2$,
  which is not a product of circumsolids.  Then for any sufficiently
  small $\alpha>0$, for any convex domain $\Omega$ which is
  sufficiently close to $\Omega_0$ in Hausdorff distance, the first
  Robin eigenfunction $u_{\alpha}$ on $\Omega$ is not log-concave.
\end{corollary}

 For $\alpha<0$, the first Robin eigenvalue $\lambda_{\alpha}$ is negative, and
 the methods used to prove Theorem~\ref{thm:main1} and
 Corollary~\ref{cor:approx-omega} also lead to the following result.

\begin{theorem}
  \label{thm:main2}
  Let $\Omega$ be a convex polyhedral domain in $\R^d$, $d\ge 2$,
  which is not a product of circumsolids. Then for sufficiently small
  $\alpha<0$, the first Robin eigenfunction $u_{\alpha}$ is not
  log-convex. Moreover, for any convex domain $\hat{\Omega}$ which
  is sufficiently close to $\Omega$ in Hausdorff distance, the first
  Robin eigenfunction $\hat{u}_{\alpha}$ on $\hat{\Omega}$ is not log-convex.
\end{theorem}

Our approach to Theorem~\ref{thm:main1} is to treat the Robin
problem~\eqref{eigenvalue}-\eqref{eq:Robin} for small positive
$\alpha$ as a perturbation from the Neumann case $\alpha=0$. To be
more precise, let $v = \tfrac{d u_\alpha}{d\alpha}\big|_{{\alpha=0}}$.
Then we show in Section~\ref{section perturbation} that the function
$v$ satisfies
  \begin{equation}
  \label{eq:3}
  \begin{cases}
   \Delta v + \mu = 0 & \quad\text{in $\Omega$,}\\
    D_{\nu} v = -1 &  \quad\text{on $\partial\Omega$,}
  \end{cases}
\end{equation}
for some constant $\mu$.  The concavity properties of $u_\alpha$ for
small $\alpha$ relate directly to the concavity properties of $v$, so
we proceed to investigate the latter, in the particular case of
polyhedral domains.  We deduce Theorem~\ref{thm:main1} from 
the statement that the solution $v$ of \eqref{eq:3} on a convex polyhedral domain
$\Omega$ is concave precisely when
$\Omega$ is a product of circumsolids.

Our argument proceeds as follows: After some
preliminary material on the perturbation problem in Section
\ref{section perturbation}, we prove in Section \ref{toy problem
  section} the remarkable result that every $C^2$ solution of
\eqref{eq:3} on a polyhedral domain is a quadratic function.  In
section \ref{sec: quadratic} we relate this to concave solutions, by showing that any
concave solution of \eqref{eq:3} is $C^2$ up to the boundary.  
This involves expanding the solution in terms of homogeneous harmonic functions 
about any boundary point, and requires in particular the interesting
observation that any degree two homogeneous harmonic function
with bounded second derivatives and with Neumann boundary condition 
on a polyhedral cone in $\R^d$ is a quadratic function.

In Section \ref{sec:domain} we prove that those polyhedral domains on
which a quadratic function solves the equation \eqref{eq:3} are products of
circumsolids. This completes the preliminaries needed to prove our main
Theorem~\ref{thm:main1} in Section~\ref{sec:proof-main-results}.  In the last section, we
discuss 
some interesting observations and open problems.

\section{Motivation:  Log-concavity and the fundamental gap}
\label{Section gap} 

In the case of Dirichlet boundary data, the log-concavity of the first
eigenfunction is a key step in proving the lower bound of the gap
between the two smallest eigenvalues \cite{fundamental}.  In the case
that the first Robin eigenfunction is log-concave, then a similar
bound holds.  {Here we note that we can include a potential, and since we impose the strong  hypothesis that the first eigenfunction is log-concave, we do not need to assume that the potential is convex.}


\begin{theorem} \label{Robin gap} 
  Let $\lambda_0$ and $\lambda_1$ be the two smallest eigenvalues for the
  eigenvalue problem 
 \begin{equation}
  \label{eigenvalue}
    -\Delta u+Vu= \lambda u \text{ in $\Omega$,}
\end{equation} 
  with Robin boundary conditions~\eqref{eq:Robin}
  on a bounded convex domain $\Omega$ with diameter $D$, {and  $V\in L^1_{loc}(\Omega)$.} If the ground state $u_0$
  associated to $\lambda_0$ is log-concave, then
  \begin{equation}
    \label{eq:1gap}
    \lambda_1-\lambda_0\ge \frac{\pi^2}{D^2}.
  \end{equation}
\end{theorem}

\begin{proof}
  Let $u_0$ and $u_1$ be the eigenfunctions associated to $\lambda_0$ and
  $\lambda_1$ respectively.   Since $u_0$ is positive on $\Omega$, we can set
  \begin{displaymath}
    v(x,t):= \frac{e^{-\lambda_1 t} u_1(x)}{e^{-\lambda_0 t} u_0(x)}
  \end{displaymath}
  which solves the parabolic equation
  \begin{equation} 
    \label{parabolic equation} 
    \dfrac{\partial v}{\partial t}= \Delta v + 2 D\log u_0\cdot Dv\quad
    \text{on $\Omega\times (0,+\infty)$.}
  \end{equation}
  On the lateral boundary $\partial\Omega\times (0,+\infty)$, the normal derivative
  of $v$ disappears:
  \begin{displaymath}
    D_\nu v= \frac{e^{-\lambda_1 t} }{e^{-\lambda_0 t} }  \left( \frac{D_\nu u_1}{u_0}
      - \frac{u_1 D_\nu u_0}{u_0^2}\right) = v\left(-\alpha+\alpha\right)=0.
  \end{displaymath}
  By hypothesis, $u_{0}$ is log-concave, so
 the drift term in~\eqref{parabolic equation} given by  
 $X:= 2\, D\log u_0$
  satisfies the \emph{modulus of contraction} inequality
  \begin{displaymath}
    \Big( X(y,t)-X(x,t)\Big)\cdot \frac{y-x}{\abs{y-x}}\le 
    0
  \end{displaymath}
  corresponding to the modulus of contraction $\omega\equiv 0$. 
  Therefore by \cite[Theorem 2.1]{fundamental}, 
  for some large constant $C>0$, the function
  \begin{displaymath}
    \varphi(s,t):=Ce^{- \frac{\pi^2}{D^2} t}\sin\left(\frac{\pi
        s}{D}\right)
    \quad\text{for every $s\in [0,D/2]$, $t\ge 0$,}
  \end{displaymath}
  is a \emph{modulus of continuity} for $v$, that is, 
  \begin{displaymath}
    v(y,t)-v(x,t)\le 2\, \varphi\left(\tfrac{y-x}{\abs{y-x}},t\right) \quad
    \text{for every $x$, $y\in \overline{\Omega}$, $t\ge 0$,}
  \end{displaymath}
  where $\frac{\pi^2}{D^2}$ is the second (or the difference of the second
  and first) Neumann eigenvalue on the interval. From this, we can deduce that 
  \begin{displaymath}
    e^{- (\lambda_1-\lambda_0)
      t}\osc_{\overline{\Omega}}\left(\frac{u_{1}}{u_{0}}\right)
    \le C\,e^{- \frac{\pi^2}{D^2}t}\quad\text{for all $t\ge 0$,}
  \end{displaymath}
  which can only hold if inequality~\eqref{eq:1gap} holds.  This completes the
  proof of Theorem~\ref{Robin gap}.
\end{proof}

The argument given follows the approach used in  the Dirichlet case
\cite{fundamental}.  A similar result would follow using the gradient estimate approach of
\cite{SWYY,YuZhong}.  

The resulting estimate is sharp in the case $\alpha=0$, where it is the Payne-Weinberger
inequality for the first nontrivial Neumann eigenvalue \cite{MR0117419,ZhongYang}. Otherwise, it 
is not sharp, as can be seen from the one dimensional case, where the
eigenvalues can be computed. It is appealing to conjecture that the
sharp lower bound for given $\alpha$ and $D$ should correspond to the
gap for the corresponding one-dimensional problem, which would result
in an estimate which depends on $\alpha$ and increases from
$\frac{\pi^2}{D^2}$ to $\frac{3\pi^2}{D^2}$ as $\alpha$ increases from
$0$ towards infinity. However, our main theorem
(Theorem~\ref{thm:main1}, that the ground state is in general not
log-concave) means that a sharp result must necessarily be proved by
rather different means.

%
%
%
%

\section{The Robin eigenvalue problem and perturbations}

\label{section perturbation}
We recall some properties of the first Robin eigenvalue
$\lambda_{\alpha}$ 
and the corresponding eigenfunction $u_{\alpha}$.  These results are
quite well established \cite{MR1335452}, see also \cite[Theorem
1.3.1]{KennedyPhD2010} or~\cite{MR3238844}, however we include a proof
for the convenience of the reader.

\begin{proposition}
  \label{propo:properties-of-Robineigenvalues}
  Let $\Omega$ be a connected bounded Lipschitz domain in $\R^{d}$. Then
  \begin{enumerate}
  \item For every $\alpha\in \R$, there is a first Robin eigenvalue
    $\lambda_{\alpha}\in \R$ with a positive eigenfunction
    $u_{\alpha}\in H^{1}(\Omega)$.     \label{item i}
  \item For every $\alpha\in \R$, the first Robin eigenvalue $\lambda_{\alpha}$ is
    simple.
     \label{item ii}
  \item The function $\alpha\mapsto \lambda_{\alpha}$ is   \label{item iii}
    differentiable, with derivative given by
    \begin{equation}
      \label{eq:31}
      \dot{\lambda}_{\alpha}= \frac{\int_{\partial
          \Omega}u_{\alpha}^{2}\,\dH}{\int_{\Omega}u_{\alpha}^{2}\dx}\ge 0.
    \end{equation}
  \item\label{propo:properties-of-Robineigenvalues-claim4} The
    positive Robin eigenfunction $u_{\alpha}$ (normalised to have
    $\tfrac{1}{\abs{\Omega}}\int_\Omega u_\alpha^2\ \dx=1$) is
    $C^{1}$-dependent on $\alpha$ in $H^{1}(\Omega)$ and in
    $C^{0,\beta}(\Omega)$ for some $\beta\in (0,1)$. More precisely,
    $u_{\alpha}$ is continuously dependent on $\alpha$ in
    $H^{1}(\Omega)$ and in $C^{0,\beta}(\Omega)$, and if for
    $\alpha_{0}\in \R$, $v$ is the unique solution, orthogonal to
    $u_{\alpha_{0}}$ in $L^2(\Omega)$, of
     \begin{equation}
       \label{eq:50}
        \begin{cases}
         \Delta v + \lambda_{\alpha_{0}} v=-\dot{\lambda}_{\alpha_{0}} u_{\alpha_{0}}
         & \text{in $\Omega$,}\\
         \;D_{\nu}v +\alpha_{0} v=- u_{\alpha_{0}}& \text{on $\partial\Omega$,}\\
       \end{cases}
    \end{equation}
    then $u_{\alpha}= u_{\alpha_{0}} +
    v\,(\alpha-\alpha_{0})+o(\alpha-\alpha_{0})$ for every
    $\alpha$ in a neighbourhood of $\alpha_{0}$, where 
    $o(\alpha-\alpha_{0})/(\alpha-\alpha_{0}) \to 0$ in $H^{1}(\Omega)\cap
          C^{0,\beta}(\Omega)$ as $\alpha\to \alpha_{0}$.
\end{enumerate}
\end{proposition}

\begin{proof}
  We begin by showing that for every $\alpha\in \R$, there is a first
  Robin eigenvalue $\lambda_{\alpha}\in \R$. For every $M>0$, let
  $[\cdot,\cdot]_{M}$ be given by
  \begin{displaymath}
    [u,v]_{M}:=\int_{\Omega}Du\, Dv\dx+M \int_{\Omega}u\,v\dx
  \end{displaymath}
  for every $u$, $v\in H^{1}(\Omega)$. Then $[\cdot,\cdot]_{M}$ is an
  inner product on $H^{1}(\Omega)$, which by the theorem of the
  bounded inverse \cite[Corollary~2.7]{MR2759829} is equivalent to
  the usual inner product on $H^{1}(\Omega)$. We denote by
  $\norm{\cdot}_{M}$ the norm on $H^{1}(\Omega)$ induced by
  $[\cdot,\cdot]_{M}$. For the rest of this proof, we denote by
  $H^{1}_{M}(\Omega)$ the Hilbert space $H^{1}(\Omega)$ equipped with
  the inner product $[\cdot,\cdot]_{M}$, and set
  \begin{displaymath}
    \tau(u,v)=\int_{\partial\Omega}u\,v\,\dH\quad\text{and}\\quadquad
    b(u,v)=\int_{\Omega}u\,v\,\dH
  \end{displaymath}
  for every $u$, $v\in H^{1}_{M}(\Omega)$. The bilinear forms $\tau$
  and $b$ on $H^{1}_{M}(\Omega)$ are bounded. Hence, by the Riesz-Fr\'echet
  representation theorem \cite[Theorem~5.5]{MR2759829}, for every
  $u\in H^{1}_{M}(\Omega)$, there are unique $Tu\in H^{1}_{M}(\Omega)$ and
  $Bu\in H^{1}_{M}(\Omega)$ satisfying
  \begin{equation}
    \label{eq:33}
    [Tu,v]_{M}=\tau(u,v)\quad\text{and}\quad
    [Bu,v]_{M}=b(u,v)
  \end{equation}
  $v\in H^{1}_{M}(\Omega)$. This defines bounded linear mappings $T$
  and $B$ on $H^{1}_{M}(\Omega)$.  Since $H^1(\Omega)$ is compactly
  embedded in $L^2(\Omega)$, $B$ is also a compact linear operator on
  $H^{1}_{M}(\Omega)$. We employ the two operators $\alpha T$ and $B$
  to characterise Robin eigenfunctions. First, recall that for every
  $\alpha\in \R$, $u\in H^{1}(\Omega)\setminus\{0\}$ is a Robin
  eigenfunction to eigenvalue $\lambda_{\alpha}$ if and only if $u$
  satisfies
  \begin{displaymath}
    \int_{\Omega}Du\,Dv\dx+\alpha \int_{\partial\Omega}u\,v\,\dH 
    = \lambda_{\alpha}\,\int_{\Omega}u\,v\dx
  \end{displaymath}
  for every $v\in H^{1}(\Omega)$, or equivalently for every $M>0$,
  \begin{displaymath}
    \int_{\Omega}Du\, Dv\dx+ M \int_{\Omega} u\,v\dx 
    +\int_{\partial\Omega}u\,v\,\dH 
    = (\lambda_{\alpha}+M)\,\int_{\Omega}u\,v\dx
  \end{displaymath}
  for every $v\in H^{1}(\Omega)$. Thus, if $I$ is the identity
  operator then the above is equivalent to
  \begin{displaymath}
    (I+\alpha T)u=(\lambda_{\alpha}+M)Bu\quad\text{in $H^{1}(\Omega)$.}
  \end{displaymath}
  By the continuity of the trace operator on $W^{1,1}(\Omega)$
  (cf~\cite[Theorem~15.8]{MR2527916}) and Young's inequality, we
  find that for all $\varepsilon>0$, there is 
  $C_{\varepsilon}>0$ such that
  \begin{displaymath}
    \norm{v}_{L^{2}(\partial\Omega)}^{2}\le \varepsilon\,\norm{D
      v}_{L^{2}(\Omega)}^{2}+C_{\varepsilon}\norm{v}_{L^{2}(\Omega)}^{2}
  \end{displaymath}
  for every $v\in H^{1}(\Omega)$ and so by choosing
  $M_{\varepsilon}=C_{\varepsilon}/\varepsilon$, we obtain that
  \begin{equation}
    \label{eq:36}
    \norm{v}_{L^{2}(\partial\Omega)}\le \sqrt{\varepsilon}\,\left(\norm{D
      v}_{L^{2}(\Omega)}^{2}+\tfrac{C_{\varepsilon}}{\varepsilon}\norm{v}_{L^{2}(\Omega)}^{2}\right)^{1/2}
  =\sqrt{\varepsilon}\,\norm{v}_{M_{\varepsilon}}
  \end{equation}
  for every $v\in H^{1}(\Omega)$. Applying Cauchy-Schwarz's
  inequality and~\eqref{eq:36}, we see
  \begin{displaymath}
    \norm{Tu}_{M_{\varepsilon}}^{2}  = [Tu,Tu]_{M_{\varepsilon}}
     =\tau(u,u)
    \le
      \norm{u}_{L^{2}(\partial\Omega)}\,\norm{Tu}_{L^{2}(\partial\Omega)}
    \le \varepsilon\,\norm{u}_{M_{\varepsilon}}\,\norm{Tu}_{M_{\varepsilon}},
  \end{displaymath}
  proving that for every $\varepsilon>0$, the operator $T$ on
  $H^{1}(\Omega), [\cdot,\cdot]_{M_{\varepsilon}}$ has operator norm
  $\norm{T}_{\mathcal{L}(H^{1}(\Omega))}\le \varepsilon$. Now, for
  given $\alpha\in \R$, we fix $\varepsilon>0$ such that
  $\abs{\alpha}<\tfrac{1}{\varepsilon}$. It follows that the operator
  $\alpha T$ on $(H^{1}(\Omega), [\cdot,\cdot]_{M_{\varepsilon}})$ has
  operator norm $\norm{-\alpha
    T}_{\mathcal{L}(H^{1}(\Omega))}<1$. Hence the operator
  $I+\alpha T$ is invertible on $H^{1}(\Omega)$, and so
  $u_{\alpha}\in H^{1}(\Omega)$ is a Robin eigenfunction with
  eigenvalue $\lambda_{\alpha}$ if and only if $u_{\alpha}$ is an
  eigenfunction of the operator $T_{\alpha}:= (I+\alpha T)^{-1}B$ for
  the eigenvalue $\lambda_{\alpha}+M_{\varepsilon}$. Note, for every
  $\alpha\in \R$, $T_{\alpha}$ is compact on $H^{1}(\Omega)$ since
  $(I+\alpha T)^{-1}$ is bounded and $B$ is compact on
  $H^{1}(\Omega)$. Therefore \cite[Theorems 6.6 \& 6.8]{MR2759829}, for
  every $\alpha\in \R$, the point spectrum $\sigma_{p}(T_{\alpha})$ of
  $T_{\alpha}$ consists of a sequence
  $(\Lambda_{\alpha}^{(j)})_{j\ge 1}$ of eigenvalues
  $\Lambda_{\alpha}^{(j)}\in \R\setminus\{0\}$ of finite algebraic and
  geometric mulitiplicity. In particular, this proves the existence of
  the first Robin eigenvalue
  $\lambda_{\alpha}:=\tfrac{1}{\Lambda_{\alpha}(1)}-M_{\varepsilon}\in
  \R$ for every $\alpha\in \R$ (for $\alpha=0$, $\lambda_{\alpha}=0$
  is the first Neumann eigenvalue). The eigenspace of
  $\lambda_{\alpha}$ is one-dimensional
  (see~\cite[Theorem~1.3.1]{KennedyPhD2010}) and admits a positive
  eigenfunction $u_{\alpha}\in H^{1}(\Omega)$ satisfying the
  normalisation $\int_{\Omega}u_{\alpha}^{2} \dx=1$. Now, the
  family $(T_{\alpha})_{\alpha\in \R}$ of compact operators
  $T_{\alpha}$ satisfies the hypotheses of~\cite[Theorem~2.6 of
  Chapter 8.2]{MR1335452}. Thus, 
  statement~\eqref{propo:properties-of-Robineigenvalues-claim4} with
  respect to the topology given by $H^{1}(\Omega)$ holds. Furthermore,
  if we apply~\cite[Theorem~3.14]{MR2812574} to the function
  $w:=u_{\hat{\alpha}}-u_{\alpha}-v(\hat{\alpha}-\alpha)$, then we see
  that statement~\eqref{propo:properties-of-Robineigenvalues-claim4}
  holds with respect to the topology given by
  $C^{0,\beta}(\Omega)$ for some $\beta\in (0,1)$.
\end{proof}

Next, we state a
convergence result on Robin problems on varying domains, which
is a slight improvement of~\cite[Corollary~3.4]{MR3566212}. For this,
we recall the definition of the \emph{Hausdorff complementary topology
on open sets} (cf~\cite[Section~2]{MR3566212}). For closed subsets $F_{1}$,
$F_{2}$ in $\R^{d}$, the \emph{Hausdorff metric} $d_{\mathcal{H}}$ is defined by
\begin{displaymath}
  d_{\mathcal{H}}(F_{1},F_{2})=\max\Big\{\sup_{x\in
    F_{1}}\textrm{dist}(x,F_{2}), \sup_{x\in F_{2}}\inf_{y\in
    F_{1}}\abs{x-y} \Big\},
\end{displaymath}
where $\textrm{dist}(x,F_{i}):=\inf_{y\in F_{i}}\abs{x-y}$ with the
standard conventions $\textrm{dist}(x,\emptyset)=+\infty$ so that
$d_{\mathcal{H}}(x,F)=0$ if $F=\emptyset$ and $d_{H}(\emptyset,F):=+\infty$ if
$F\not=\emptyset$. 
Let $\Omega^{c}=\R^d\setminus\Omega$ be the complement of $\Omega$.
 Now, a sequence $(\Omega_{n})_{n\ge 1}$ of open
sets $\Omega_{n}$ in $\R^{d}$  {converges to the open set $\Omega$
in $\R^{d}$ in the Hausdorff complementary topology},  which we write as
$\Omega_{n}\to \Omega$ in $\mathcal{H}^{c}$, if for every closed ball
$B$ in $\R^{d}$, one has that $d_{H}(B\cap \Omega_{n}^{c},B\cap
\Omega^{c})\to 0$ as $n\to \infty$.

\begin{proposition}
  \label{propo:domain-perturbation}
  For $d\ge 1$, let $D\subseteq \R^{d}$ be an open and bounded set,
  and let $\Omega$ and $\Omega_{n}$ be open domains with a Lipschitz
  continuous boundary satisfying $\Omega$,
  $\Omega_{n}\subset\subset D$. Let
  \begin{displaymath}
    \Omega_{n}\to \Omega\text{ in $\mathcal{H}^{c}$,}\quad
    \abs{\Omega_{n}}\to \abs{\Omega},\quad
    \mathcal{H}^{d-1}(\partial\Omega_{n})\to \mathcal{H}^{d-1}(\partial\Omega)
  \end{displaymath}
  as $n\to+\infty$. 
  Furthermore, for $\alpha>0$, let $\lambda_{\alpha,n}$ and
  $\lambda_{\alpha}$ be the first Robin eigenvalue on $\Omega_{n}$ and
  $\Omega$, and let $u_{\alpha,n}$ and $u_{\alpha}$ be the first
  positive Robin eigenfunctions with unit $L^{2}(\Omega)$-norm.  Then
  \begin{align}
    \label{eg:conv-eigenvalues}
    &\lambda_{\alpha,n}\to \lambda_{\alpha} \quad\text{as
                                               $n\to+\infty$,}\\ 
    \label{eq:conv-eigenfunctionH1}
    &u_{\alpha,n}\,\mathds{1}_{\Omega_{n}}\to
    u_{\alpha}\,\mathds{1}_{\Omega}\quad\text{in $H^{1}(D)$ as
                                     $n\to+\infty$.}
  \end{align}
   Furthermore, there are $\gamma\in (0,1)$ and $C>0$ such that
  \begin{equation}
    \label{eq:14}
    \norm{u_{\alpha,n}}_{C^{0,\gamma}(\overline{\Omega}_{n})}\le C\quad\text{for
    all $n\ge 1$,}
  \end{equation}
  and for every non-empty set
  $B\subseteq \bigcap_{n\ge n_{0}}\overline{\Omega}_{n}$,
  $n_{0}\ge 1$, and $0\le \hat{\gamma}<\gamma$, there is a subsequence
  $(u_{\alpha,\hat{n}})_{\hat{n}\ge 1}$ of
  $(u_{\alpha,\hat{n}})_{\hat{n}\ge 1}$ such that
  \begin{equation}
    \label{eq:16}
    u_{\alpha,\hat{n}}\to u_{\alpha}\quad\text{in
      $C^{0,\hat{\gamma}}(B)$ as $\hat{n}\to+\infty$.}
  \end{equation}
\end{proposition}

\begin{proof}
  Under the hypotheses of this Proposition,
  \cite[Corollary~3.4]{MR3566212} 
  implies~\eqref{eq:31}.
  Thus,
  \begin{equation}
    \label{eq:15lim}
    \begin{split}
      \lim_{n\to +\infty}\norm{D
        u_{\alpha,n}}_{L^{2}(\Omega_{n};\R^{d})}^{2}+\alpha
      \norm{u_{\alpha,n}}_{L^{2}(\partial\Omega_{n})}^{2}
      & =\lim_{n\to +\infty}\lambda_{\alpha,n}\\
      & = \lambda_{\alpha}\\
      & =\norm{D
        u_{\alpha}}_{L^{2}(\Omega;\R^{d})}^{2}+\alpha\norm{u_{\alpha}}_{L^{2}(\partial\Omega)}^{2}.
    \end{split}
  \end{equation}
  Since $\Omega_{n}\subseteq D$ for all $n\ge 1$, we can conclude from
  the limit~\eqref{eq:15lim} and by \cite[Lemma~4.2 and
  Lemma~4.7]{MR3566212} that limit~\eqref{eq:conv-eigenfunctionH1}
  holds strongly in $L^{2}(D)$ and weakly in $H^{1}(D)$. Moreover, by
  limit~\eqref{eq:15lim} and since
  $D u_{\alpha,n}\,\mathds{1}_{\Omega_{n}}\to D
  u_{\alpha}\,\mathds{1}_{\Omega}$ weakly in $L^{2}(D;R^{d})$ as
  $n\to+\infty$ and by~\cite[Lemma~4.7]{MR3566212}, it follows that
  limit~\eqref{eq:conv-eigenfunctionH1} holds in $H^{1}(D)$. Finally,
  bound~\eqref{eq:14} and limit~\eqref{eq:16} in
  $C^{0,\hat{\gamma}}(B)$ for every non-empty set
  $B\subseteq \bigcap_{n\ge 1}\overline{\Omega}_{n}$ and
  $0\le \hat{\gamma}<\gamma$ are consequences
  from~\cite[Proposition~3.6]{MR2812574}.
\end{proof}

%
%
%
%

\section{Regular solutions are quadratic}

\label{toy problem section}

When $\alpha=0$, the perturbation problem \eqref{eq:50} reduces to
equation \eqref{eq:3}, 
with the constant $\mu$ 
computed by integrating the first equation over $\Omega$ and
applying the boundary condition, yielding
$\mu= {\mathcal H}^{d-1}(\partial\Omega)/\mathcal{H}^d(\Omega)$.

In this and the next several sections we consider a class of problems
generalising \eqref{eq:3}, under the assumption that $\Omega$ is a
convex polyhedral domain in $\R^{d}$ for $d\ge 2$. More precisely,
this means that $\Omega$ is
the intersection of finitely many open
half-spaces:
\begin{displaymath}
\Omega=\bigcap_{i=1}^m\Big\{x\in \R^{d}\,\Big\vert\, x\cdot\nu_i< b_i\Big\},
\end{displaymath}
and we can assume without loss of generality that the description is
minimal, meaning that omitting any one of the half-spaces from the
intersection results in a strictly larger set. 
In this case $\Omega$
has $m$ faces
\begin{displaymath}
  \Sigma_i=\Big\{x\in\overline\Omega\,\Big\vert\ \nu_i\cdot x=b_i\Big\}
\end{displaymath}
for $i=1,\dots,m$, each of which is itself a convex polyhedral subset
of the affine subspace $\{x\in\R^d\,\vert\ \nu_i\cdot x=b_i\}$.  
The outer unit normal to $\Omega$ on the face $\Sigma_i$ is $\nu_i$.

For an open convex set $\Omega$ in $\R^d$, the \emph{tangent cone}
$\Gamma_x$ to $\Omega$ at a point $x\in\overline{\Omega}$ is defined
by
\begin{displaymath}
  \Gamma_x = \Big\{r(y-x)\;\Big\vert\; y\in\Omega,\ r>0\Big\} 
  = \bigcup_{r>0}r(\Omega-x).
\end{displaymath}
If $x$ is in $\Omega$, the tangent cone $\Gamma_x$ is simply
$\mathbb{R}^d$. In the case of polyhedral domains, the tangent cone
can be described as follows: For each point $x\in\overline{\Omega}$,
let
\begin{equation}
  \label{eq:9}
  \mathcal{I}(x) := \Big\{i\in\{1,\dots,m\}\;\Big\vert\; x\cdot
  \nu_i=b_i\Big\}
\end{equation}
index the faces touching $x$, then the tagent cone
  \begin{displaymath}
\Gamma_{x}=
\bigcap_{i\in{\mathcal I}(x)}\Big\{y\;\Big\vert\; y\cdot\nu_i< 0\Big\}.
\end{displaymath}
This is a cone over the subset $A_{x} = \Gamma_{x}\cap \S^{d-1}$ of
the unit sphere.  In particular, $\Gamma_x$ is 
 the intersection of finitely many half-spaces with the origin in their
common boundary. We call such a set a \emph{polyhedral cone}.
\begin{remark}
  \label{rem:1}
  A special feature of polyhedral domains is that for every
  $x\in\overline{\Omega}$ there exists $r>0$ such that
  $B_r(x)\cap\Omega = x+\left(B_r(0)\cap\Gamma_{x}\right)$, so that
  $\Omega$ is a cone near $x$.
\end{remark}

We now establish a version of the strong maximum principle on general open
cones $\Gamma$ with a Lipschitz boundary.   In this paper, our application of
Proposition~\ref{prop:cone-MP} remains on cones with a polyhedral structure.

\begin{proposition}\label{prop:cone-MP}
  Let $\Gamma$ be an open cone with Lipschitz boundary and vertex at
  the origin in $\R^d$, and $r>0$.  Let $w\in H^{1}(B_r(0)\cap\Gamma)$ be a
  weak solution of
  \begin{equation}\label{eq:weak-harmonic-Neumann}
    \begin{cases}
      \Delta w = 0&\text{on}\ B_r(0)\cap \Gamma,\\
      D_\nu w = 0&\text{on}\ B_{r}(0)\cap \partial\Gamma.
    \end{cases}
  \end{equation}
  If $w(0)=0$ and $w\leq 0$ on $B_r(0)\cap \Gamma$, then $w\equiv 0$
  on $B_r(0)\cap \Gamma$.
\end{proposition}

By scaling it suffices to consider the case $r=1$.
We begin by setting $A=S^{d-1}\cap \Gamma$. Then the set
$B_1(0)\cap\Gamma$ can be described by the \emph{polar coordinate} map
\begin{displaymath}
  (r,z)\in(0,1)\times A\mapsto rz\in B_1(0)\cap\Gamma.
\end{displaymath}
Since the set $A$ is a Lipschitz domain in $\S^{d-1}$, there is a
complete $L^2(A)$-ortho\-normal set of eigenfunctions
$\{\varphi_i\}_{i=0}^\infty$ for the Neumann Laplacian on $A$, with
associated eigenvalues $\lambda_i$ which we arrange in non-decreasing
order with $\lambda_0=0$. Let $w\in H^{1}(B_1(0)\cap\Gamma)$. Then for
every $r\in (0,1)$, $w\in H^{1}(\ B_r(0)\cap \Gamma)$, the trace
$w(r,\cdot)$ of $w$ exists in $L^{2}(A)$. Using this, we see that $w$
can be rewritten in polar coordinates as
  \begin{equation}
    \label{eq:7series}
     w(r,z) = \sum_{i=0}^{\infty} w_i(r)\varphi_i(z)\quad
     \text{for every $(r,z)\in[0,1)\times A$,}
  \end{equation}
  where for every $r\in (0,1)$ and $i\ge 1$,
  \begin{equation}
    \label{eq:7coef}
     w_{i}(r):=(w(r,\cdot), \varphi_{i})_{L^{2}(A)}
  \end{equation}
is the $i$th Fourier coefficient of the trace of $w(r,\cdot)$
  in $L^{2}(A)$. In order to continuous the proof of
  Proposition~\ref{prop:cone-MP}, we need to establish first some more
  properties of the series decomposition~\eqref{eq:7series} of the
  weak solution $w$ of~\eqref{eq:weak-harmonic-Neumann}. This is done
  in the next two statements.

 \begin{lemma}
   \label{lem:conv-fi-series}
   Let $\Gamma$ be an open cone with Lipschitz boundary and vertex at
  the origin in $\R^d$, and let $w\in H^{1}(B_1(0)\cap\Gamma)$ be a
  weak solution of Neumann problem~\eqref{eq:weak-harmonic-Neumann}. 
    Then for all $i\ge 1$,
   \begin{equation}
     \label{eq:7coef-of-f}
     f_i:=\lim_{r\to 1}w_i(r)
   \end{equation}
   exists, and furthermore the series
   $\sum_{i=0}^{\infty} \sqrt{1+\lambda_i}f_i^2$ converges with
   \begin{equation}
     \label{eq:7-series-fi}
     \sum_{i=0}^{\infty} \sqrt{1+\lambda_i}f_i^2 \leq C\,\norm{w}_{H^1(B_1(0)\cap\Gamma)}^2.
   \end{equation}
 \end{lemma}

 \begin{proof}
   The $H^1(B_1(0)\cap\Gamma)$-norm of
   $w$ can be written as
  \begin{equation}\label{eq:H1norm}
     \begin{split}
       \norm{w}_{H^{1}(B_1(0)\cap\Gamma)}^2 &= \int_{B_1(0)\cap\Gamma}w^2+|Dw|^2\,d{\mathcal H}^d\\
       &=\sum_{i,j=0}^{\infty}\left\{\int_0^1
       \left(w_iw_j+w_i'w_j'\right)r^{d-1}\,dr 
       \int_A\varphi_i\varphi_j d{\mathcal H}^{d-1}(z)\right\}\\
       &\qquad\null+ \sum_{i,j=1}^{\infty}\left\{
         \int_0^1 w_iw_j r^{d-3} dr \int_A D\varphi_i\cdot D\varphi_j\,d{\mathcal H}^{d-1}(z)\right\}\\
       &=\sum_{i=0}^{\infty} \int_0^1
       \left(\Big(1+\frac{\lambda_i}{r^2}\Big)\, w_i^2 + (w_i')^2
         \right)r^{d-1}\,dr,
     \end{split}
 \end{equation}
 where $w_i'(r)=\frac{dw_i}{dr}(r)=\int_A\nabla w(r,z)\cdot
 z\,\varphi_{i}(z)\,d\mathcal{H}^{d-1}(z)$ 
and
   \begin{equation*}
 \int_A D\varphi_i\cdot D\varphi_j \,d{\mathcal H}^{d-1}
   =-\int_A \Delta\varphi_i\varphi_j\,d{\mathcal H}^{d-1}
   =\lambda_i\int_A\varphi_i\varphi_j\,d{\mathcal H}^{d-1}=\lambda_i\delta_{ij}.
 \end{equation*}

Let $\delta\in (0,1)$ and consider the mapping $g : [\delta,1) \to L^2(A)$ defined by
\begin{displaymath}
  g(r)=\sum_{i=0}^{\infty}\sqrt{1+\lambda_i}\,w_{i}^{2}(r)\quad\text{for every $r\in [\delta,1)$.}
\end{displaymath}
Then
\begin{displaymath}
\labs{\frac{d}{dr}g(r)} = \labs{2 \sum_{i=0}^{\infty}\sqrt{1+\lambda_i}\,w_iw_i'}
\leq \sum_{i=0}^{\infty}w_i^2 \Big(1+\frac{\lambda_i}{r^{2}}\Big) + \sum_{i=0}^{\infty}\left(w_i'\right)^2
\end{displaymath}
and so
\begin{equation}
  \label{eq:7g}
  \begin{split}
    \abs{g(r_2)-g(r_1)}
    &\le \int_{r_1}^{r_2}\labs{\frac{d}{dr}g(r)}\,\dr\\
    & \le \int_{r_1}^{r_2}\sum_{i=0}^{\infty} 
    (1+\frac{\lambda_i}{r^{2}})w_i^2 + \sum_{i=0}^{\infty}\left(w_i'\right)^2\,dr\\
    &\le
    C_{\delta}\sum_{i=0}^{\infty}\int_{r_1}^{r_2}\left((1+\frac{\lambda_i}{r^2})w_i^2
      + (w_i')^2\right)r^{d-1}\,dr.
  \end{split}
\end{equation}
for every $0<\delta<r_1<r_2<1$. By~\eqref{eq:H1norm}, the right hand
side in the last estimate of~\eqref{eq:7g} tends to zero as $r_{1}$,
$r_{2}\to 1^-$. Hence, the Cauchy criterion implies that
\begin{displaymath}
  \lim_{r\to 1^-}g(r)=
  \sum_{i=0}^{\infty}\sqrt{1+\lambda_i}\,f_{i}^{2}\quad\text{exists,}
\end{displaymath}
where $f_{i}$ is defined by~\eqref{eq:7coef-of-f}. This 
shows that the function $g$ is absolutely continuous on $[\delta,1]$
for every $\delta\in (0,1)$. By the mean value theorem for
integrals, there is an $r_{\delta}\in (\delta,1)$ satisfying
 \begin{displaymath}
    g(r_{\delta}) = \tfrac{1}{1-\delta}\int_\delta^1 g(r)dr 
    = \tfrac{1}{1-\delta}\sum_{i=0}^{\infty}\int_\delta^1 \sqrt{1+\lambda_i}w_i^2\,dr
  \le \tfrac{C_\delta}{1-\delta}\norm{w}_{H^1(B_1(0)\cap\Gamma)}^2,
 \end{displaymath} 
where we also used~\eqref{eq:7g} and \eqref{eq:H1norm}. Using this
together with~\eqref{eq:7g}, one finds
\begin{displaymath}
  g(r)= g(r)-g(r_{\delta})+ g(r_{\delta}) \le C\,\norm{w}_{H^1(B_1(0)\cap\Gamma)}^2
\end{displaymath}
for some $C>0$ independent of $r\in (\delta,1)$. Sending $r\to 1$, we
find~\eqref{eq:7-series-fi}.
\end{proof}
 
Due to Lemma~\ref{lem:conv-fi-series}, every weak
solution $w$ of~\eqref{eq:weak-harmonic-Neumann} has the following series
expansion.

\begin{proposition}
  \label{lem:series-expansion}
  Let $\Gamma$ be an open cone with Lipschitz boundary and vertex at
  the origin in $\R^d$, and 
  follow the notation of Lemma~\ref{lem:conv-fi-series}. Then every
  weak solution $w\in H^{1}(B_1(0)\cap\Gamma)$
  of~\eqref{eq:weak-harmonic-Neumann} 
 satisfies
  \begin{equation}\label{eq:hhdecomp}
  w(rz) = \sum_{i=0}^\infty f_i\, r^{\beta_i}\varphi_i(z)\quad\text{for
  every $z\in A$ and $r\in(0,1)$.}
\end{equation}
The convergence of the series holds in 
$H^1(B_1(0)\cap\Gamma)\cap
C^{\gamma_{r}}(\overline{B_r(0)\cap\Gamma})$ for every $0<r<1$, where
$\gamma_{r}\in (0,1)$, and for every integer $i\ge 0$, $\beta_i\ge 0$ solves
\begin{equation}
  \label{eq:beta-i}
  \beta_i^2 +(d-2)\beta_i-\lambda_i=0.
\end{equation}
\end{proposition}

We will often use  \eqref{eq:beta-i} in the form $\beta_i=\frac12\left( d-2 +\sqrt{(d-2)^2+4\lambda_i}\right)$.

\begin{proof}
We define 
\begin{displaymath}
  \psi_i(rz) := r^{\beta_i}\varphi_i(z)\quad\text{for every $rz\in B_1(0)\cap\Gamma$.}
\end{displaymath}
Then  $\psi_i$ is harmonic on $B_1(0)\cap\Gamma$ since
  \begin{align*}
  \Delta (r^\beta_{i}\varphi_{i}(z))
  & = r^{\beta_{i}-2}\Delta_{\S^{d-1}} \varphi_{i}
    +(d-1)\frac{\partial r^{\beta_{i}}}{\partial r}\varphi_{i} 
    + \frac{\partial^2 r^{\beta_{i}}}{\partial r^2}\varphi_{i} \\
  &= r^{\beta_{i}-2}\left(-\lambda_{i} + ({d-2}) \beta_{i}+
    \beta_{i}^2\right)\varphi_{i}\\
  &=0
\end{align*}
by~\eqref{eq:beta-i} 
and the fact that $\varphi_{i}$ satisfies
\begin{equation}
  \label{eq:35}
  \Delta^{\S^{d-1}}\varphi_i+\lambda_i\varphi_i=0\quad\text{ on $A$.}
\end{equation}
Furthermore, $\psi_i$ satisfies Neumann boundary conditions on
$B_1(0)\cap\partial\Gamma$, since $\varphi_i$ satisfies Neumann
conditions on $\partial A$. Thus, each $\psi_{i}$ is a weak
solution of~\eqref{eq:weak-harmonic-Neumann}.

Now, let $\tilde{w} : B_1(0)\cap\partial\Gamma\to \R$ be given by 
\begin{displaymath}
  \tilde{w}(r,z) : = \sum_{i=0}^{\infty}
  f_i\,\psi_i(rz)=\sum_{i=0}^{\infty}f_i\, r^{\beta_i}\,\varphi_i(z)\quad
  \text{for every $rz\in B_{1}(0)\cap \Gamma$,}
\end{displaymath}
where $f_{i}$ is given by~\eqref{eq:7coef-of-f}. Next, we show that
the infinite series of $\tilde{w}$ converges in
$H^1(B^1(0)\cap\Gamma)$. For this, let $\tilde w^N$ be the 
 partial
sum of $\tilde{w}$ given by
\begin{displaymath}
  \tilde w^N(r,z) = \sum_{i=0}^N f_i\,r^{\beta_i}\varphi_i(z) \quad
  \text{for every $rz\in B_{1}(0)\cap \Gamma$.}
\end{displaymath}
For integers $1\le M< N$, applying \eqref{eq:H1norm} to
$\tilde w^N-\tilde w^M = \sum_{i=M+1}^Nf_ir^{\beta_i}\varphi_i$, we
find
\allowdisplaybreaks
\begin{align*}
\norm{\tilde w^N-\tilde w^M}_{H^1(B_{1}(0)\cap \Gamma)}^2 
  &= \sum_{i=M+1}^N \int_0^1\left(f_i^2 r^{2\beta_i+d-1}+\beta_i^2f_i^2r^{2\beta_i+d-3}\right)\,dr\\
&=\sum_{i=M+1}^N \left(\frac{1}{2\beta_i+d}+\frac{\beta_i^2}{2\beta_i+d-2}\right)f_i^2\\
&\leq C\sum_{i=M+1}^N \left(\beta_i+1\right)f_i^2\\
&\leq C\sum_{i=M+1}^N\sqrt{1+\lambda_i}f_i^2.
\end{align*}
 Lemma~\ref{lem:conv-fi-series} implies that the infinite
series $\sum_{i=0}^{\infty}\sqrt{1+\lambda_i}f_i^2$ is convergent, and so 
there is $\tilde{w}\in H^1(B_1(0)\cap\Gamma)$ such that
$\tilde w^N$ converges to $\tilde{w}$ in
$H^1(B_1(0)\cap\Gamma)$. Since every partial sum
$\tilde{w}^{N}$ is a weak solution
of~\eqref{eq:weak-harmonic-Neumann}, the limit function $\tilde w$ is
also a weak solution of~\eqref{eq:weak-harmonic-Neumann} and has $L^{2}$-trace 
\begin{displaymath}
\sum_{i=0}^{\infty} f_i\,\varphi_i\quad\text{on
$A$.}
\end{displaymath}
Since the same is true for $w$, we have $w=\tilde w$, proving
that~\eqref{eq:hhdecomp} holds in $H^1(B_1(0)\cap\Gamma)$. To obtain
convergence of the series~\eqref{eq:7series} in
$C^{\gamma_{r}}(\overline{B_r(0)\cap\Gamma})$ for every $0<r<1$ with
some $\gamma_{r}\in (0,1)$, we employ a reflection argument in a small
neighbourhood $U$ of each boundary point of
$B_{r}(0)\cap \partial\Gamma$ as in \cite{MR2812574} and use the
interior H\"older-regularity
result~\cite[Theorem~8.24]{MR1814364}. Further, we can cover
$\overline{B_r(0)\cap\Gamma}\setminus \partial\Gamma$ by finitely many
balls and apply again the interior H\"older-regularity to
$w$. Summarising, we see that for every $0<r<1$, there is a
$\gamma_{r}\in (0,1)$ such that the series~\eqref{eq:7series}
converges in $C^\gamma_{r}(\overline{B_r(0)\cap\Gamma})$.
\end{proof}

With the above preliminary results established, we can prove
Proposition~\ref{prop:cone-MP}:

\begin{proof}[Proof of Proposition~\ref{prop:cone-MP}]
Since $w(0)=0$ and $\beta_i>0$ for $i>0$, we have
$f_0=0$. Note that $\beta_i$ is non-decreasing in $\lambda_i$ and
hence in $i$.

Now assume $w$ is not identically zero.   Let
$f_{i_0}$ be the first non-zero coefficient, so that we have
\begin{displaymath}
  w(r,z) = r^{\beta_{i_0}}\left(\sum_{{\beta_i=\beta_{i_0}}}f_i \varphi_i(z) 
    + \sum_{\beta_i>\beta_{i_0}}f_i r^{\beta_i-\beta_{i_0}}\varphi_i(z)\right).
\end{displaymath}
The bracket on the right is non-positive since $w\leq
0$, and the second term converges uniformly to zero in $z$ as
$r$ approaches zero, while the first term is constant in
$r$. Hence we have that the term
\begin{displaymath}
   h(z):=\sum_{\beta=\beta_{i_0}}f_i\varphi_i(z)\leq 0\quad\text{for
     all $z\in A$.}
\end{displaymath}
But $h$ is a non-constant Neumann eigenfunction on the connected
domain $A$, and hence changes sign. This is a contradiction, and so
$w$ must be identically zero as claimed in Proposition~\ref{prop:cone-MP}.
\end{proof}

Although we are mostly interested in the perturbation problem
\eqref{eq:3}, the results of this section and the next also apply for
a somewhat larger class: We consider (weak) solutions $v$ of the
problem  
\begin{equation}\label{eq:3-generalised}
\begin{cases}
\Delta v + \mu = 0&\quad\text{in\ }\Omega,\\
D_{\nu_i} v + \gamma_i =0&\quad\text{on\ }\Sigma_i.
\end{cases}
\end{equation}
where $\mu$ and $\gamma_1,\cdots,\gamma_m$ are constants.  We observe
(by integration of the first equation over $\Omega$ and application of
the boundary condition on each face $\Sigma_i$) that these constants
necessarily satisfy the relation
\begin{displaymath}
\sum_{i=1}^m\gamma_i{\mathcal H}^{d-1}(\Sigma_i) = \mu\mathcal H^d(\Omega).
\end{displaymath}

The main result of this section is the following:

\begin{theorem} 
  \label{lemma 1.2} 
  Let $\Omega$ be a polyhedral domain in $\R^d$ with faces
  $\Sigma_1,\dots,\Sigma_m$ and $v$ be a solution
  of~\eqref{eq:3-generalised}. If $v\in C^{2}(\overline{\Omega})$ then
  $v$ is quadratic; that is, there are constants $a_{ij}$, $b_{i}$, $c\in \R$
  such that
 \begin{displaymath}
  v(x)=\sum_{i,j=1}^{d}a_{ij}x_{i}x_{j}+\sum_{i=1}^{d}b_{i}x_{i}+c
\end{displaymath}
for every $x\in \overline{\Omega}$.
\end{theorem}

  Our strategy to prove Theorem~\ref{lemma 1.2} is to show that there
  exists a subspace $E$ in $\R^{d}$ on which the Hessian function
  $(x,e)\mapsto D^2v|_x(e,e)$ is constant for all unit vectors $e\in E$ 
  and $x\in\Omega$.
  It will follow from this that $v(x)$ is a multiple of the
  squared length of the $E$ component of $x$, plus another function
  depending only on the $E^\perp$ component, 
  where $E^\perp$ denotes the orthogonal complement of $E$ in $\R^{d}$.
   This reduces the
  original problem to a similar problem on the lower-dimensional space
  $E^\perp$, enabling an induction on dimension to establish the
  result.

  Accordingly, we proceed by induction: For $d=1$, a polyhedral domain
  is simply an interval, and every solution to
  \eqref{eq:3-generalised} is a quadratic function, so the statement
  of Theorem~\ref{lemma 1.2} holds in this case.  Now, assume that the statement of
  Theorem~\ref{lemma 1.2} holds for every polyhedral domain in
  $\R^j$ for $j<d$, and let $\Omega$ be a polyhedral domain in $\R^d$
  and $v\in C^2(\overline{\Omega})$ be a solution of \eqref{eq:3-generalised} on $\Omega$.
  Since $v\in C^2(\overline\Omega)$, there exists
  $(x_{0},e_{1})\in \overline{\Omega}\times \S^{d-1}$ such that
  \begin{equation}
    \label{eq:7'}
    \Lambda:=\max_{x\in \overline{\Omega}, \,e\in \S^{d-1}}
    D^2v|_x(e,e) =D^2v|_{x_0}(e_1,e_1).
  \end{equation}

%


  \begin{lemma}
    \label{lem:constNeumann}
    Suppose that $v$ is a $C^2$ function on an open subset $B$
    of $\overline{\Omega}$, where $\Omega$ is a polyhedral domain in
    $\R^d$. For $j\in \{1,\dots,m\}$, let $\nu_{j}$ be the outward
    pointing unit normal vector on face $\Sigma_{j}$ and suppose
    \begin{equation}
      \label{eq:8}
     D_{\nu_{j}}v+\gamma_{j}=0\quad\text{on $\overline{\Sigma}_{j}\cap B$.}
  \end{equation}
  Then for every tangent vector $e$ parallel to $\Sigma_j$ one has
  \begin{equation} 
        \label{DnuDe} 
        D^2v|_x(e,\nu_j)=0 \quad\text{for every
          $x\in\overline{\Sigma}_j\cap B$.} 
      \end{equation}
   In particular, $\nu_{j}$ is an eigenvector for the Hessian $D^2v|_x$ for each
  $x\in\overline{\Sigma}_{j}\cap B$.
  \end{lemma}
  
  \begin{proof}
    On polyhedra, the normal vector $\nu_{j}$ is constant on face $\Sigma_{j}$. 
      Differentiating the boundary condition~\eqref{eq:8} in
      the direction of any tangent vector $e\in T\Sigma_j$
      yields~\eqref{DnuDe}. Since $\R^{d}$ can be decomposed
      as a direct sum of the tangent space 
      $T\Sigma_{j}$ 
      and
      the 
       normal vector $\nu_{j}$, 
       \eqref{DnuDe} implies
      that $\nu_{j}$ is an eigenvector for the Hessian $D^2v|_x$ for $x\in
      \overline{\Sigma}_{j}\cap B$.
  \end{proof}

  Our second lemma captures in slightly greater generality the
  dimension-reduction argument outlined above:

\begin{lemma}
  \label{lem:Hessian-SMP}
  Suppose that $v$ is a $C^2$ solution of \eqref{eq:3-generalised} on
  a convex open subset $B$ of $\overline{\Omega}$, where $\Omega$ is a
  polyhedral domain in $\R^d$.  If there exists $(x_0,e_1)$ in
  $B\times \S^{d-1}$ such that
  \begin{equation}
    \label{eq:7bis}
    D^2v|_{x_0}(e_1,e_1)
    = \Lambda:=\sup_{(x,e)\in B\times
      \S^{d-1}}D^2v|_x(e,e),
  \end{equation}
  then there exists a subspace $E$ of positive dimension in $\R^d$ such that
  \begin{displaymath}
    B\cap\Omega = \Big\{x\in B\;\Big\vert\; \pi_E(x)\in\Omega^E,\; 
    \pi_{E^\perp}(x)\in\Omega^\perp\Big\},
  \end{displaymath}
  where $E^\perp$ is the orthogonal complement of $E$, $\pi_E$ and
  $\pi_{E^\perp}$ are the orthogonal projections onto $E$ and
  $E^\perp$, and $\Omega^E=\pi_{E}(\Omega)$ and
    $\Omega^\perp=\pi_{E^{\perp}}(\Omega)$ are polyhedral
  domains in $E$ and $E^{\perp}$ respectively. Furthermore,
  \begin{equation}
    \label{eq:10}
    v(x) = \frac{\Lambda}{2}|\pi_E(x-x_0)|^2 + Dv|_{x_0}\left(\pi_E(x-x_0)\right)+ g(\pi_{E^{\perp}}(x))
  \end{equation}
  for all $x\in B$, where $g$ is a $C^2$ solution
  of an equation of the form \eqref{eq:3-generalised} on
  $\pi_{E^{\perp}}(B)\subseteq\overline{\Omega^\perp}\subseteq E^\perp$.
\end{lemma}

\begin{proof}
  Without loss of generality, we can assume that we have chosen
  $x_0\in B$ so that the dimension of the eigenspace of $H_v(x_0)$
  with eigenvalue $\Lambda$ is maximized. We begin by defining $u$ to
  be the part of $v$ without its \emph{quadratic approximation about
    $x_0$}:
\begin{equation}
  \label{eq:26}
  u(x):={v}(x)- v(x_0)-Dv|_{x_0}(x-x_0) - \frac12 D^2v|_{x_0}(x-x_0,x-x_0)
\end{equation}
for every $x\in B$. Then $u$ has the following properties:
\begin{align}
 u(x_{0})&=0, \quad  D u(x_{0})=0, \quad
   D^2u|_{x_0}=0;\notag \\
 \label{eq:grad-u}
   D u|_x &= Dv|_x-Dv|_{x_0}-D^2v|_{x_0}(x-x_0,.)\quad \text{
    for every $x\in B$;}\\
  \label{eq:laplace-u-equation} 
  \Delta u(x)
    &=  \Delta {v}(x) - \Delta v(x_0)  = -\mu +\mu=0\quad \text{
    for every $x\in B\cap \Omega$;}\\
  \label{number 8}  
   D_{\nu_j} u (x)&=0 \quad\text{ for all $x\in\Sigma_j\cap
    B$}\quad\text{if $j\in \mathcal{I}(x_0)$, } 
\end{align}
where the index set $\mathcal{I}(x_0)$ is given by~\eqref{eq:9}. To
see that~\eqref{number 8} holds, first note that this is trivially
satisfied if $x_0\not\in\partial\Omega$, since then $\mathcal{I}(x_0)$
is empty. If $x_0\in\partial\Omega$, then by
Lemma~\ref{lem:constNeumann}, for every $j\in \mathcal{I}(x_0)$, $v$
satisfies~\eqref{DnuDe}. If $x\in \Sigma_j\cap B$, then
both $x$ and $x_0$ lie in the same face $\Sigma_{j}$ and so
$(x-x_0)\in T\Sigma_j$. %
 By taking $e=x-x_0$ and using~\eqref{eq:grad-u} and \eqref{DnuDe}, 
 one has
\begin{displaymath}
 D_{\nu_j} u(x)= D_{\nu_j}v (x) -D_{\nu_j}v(x_0)-
 \left.D^2v\right|_{x_0}(e,\nu_j)
 = \gamma_j-\gamma_j=0 \text{ for all }
x\in\Sigma_j\cap B.
\end{displaymath}

Now, let $E$ be the eigenspace of $D^2v|_{x_0}$ corresponding to its
largest eigenvalue $\Lambda$. Then, $e_1\in E\cap \S^{d-1}$. We choose
an orthonormal basis $\{e_{1}, e_{2}, \dots, e_{k}\}$ of $E$, $1\le k\le d$, and set
\begin{equation*}
    \label{eq:25}
  f(x)=\textrm{tr}_{E}\left(D^2u|_x\right):=\sum_{i=1}^{k}\left.D^2u\right|_x(e_i,e_i)\quad
  \text{for every $x\in B$.}
\end{equation*}  
Then $f$ has the following properties:
\allowdisplaybreaks
\begin{align}
  f(x)&=\sum_{i=1}^{k}\left(D^2v|_x-D^2v|_{x_0}\right)(e_i,e_i)
    \quad\text{for all $x\in B$ by~\eqref{eq:26};} \notag \\  
  \label{eq:fx0}  
  f(x_{0}) 
    &=0
\quad\text{by the above;}\\
     \label{f harmonic} 
  \Delta f(x)&=0\quad\text{for every $x\in \Omega\cap B$
    by~\eqref{eq:laplace-u-equation};}\\
  \label{eq:sign-of-f} 
  f(x)&\le 0\quad\text{for every $x\in 
        B$;}\\  
   \label{Neumann f}   
  D_{\nu_{j}}f(x)&=0\quad\text{for every $x\in B\cap\Sigma_{j}$, if $j\in \mathcal{I}(x_0)$.}
\end{align}
To see that~\eqref{eq:sign-of-f} holds, note that by~\eqref{eq:7bis}, 
\begin{equation}  \label{eq:8prime}  
  D^2u|_x(\xi,\xi)=D^2v|_x(\xi,\xi)-D^2v|_{x_0}(\xi,\xi) = D^2v|_x(\xi,\xi)-\Lambda\leq 0
  \end{equation}
for all $\xi \in E \cap \S^{d-1}$ and $x\in B$. 

To show~\eqref{Neumann
  f}, fix $j\in \mathcal I(x_0)$. Then by Lemma~\ref{lem:constNeumann}
applied to $v$, the normal $\nu_j$ is an eigenvector of $D^2v|_x$ for
$x\in\overline{\Sigma}_j$.  On the interior of the face $\Sigma_j$,
$v\in C^3(\Sigma_j)$ (since $u$ extends by even reflection in
$\Sigma_j$ as a harmonic function) and so we can differentiate
\eqref{DnuDe} again to find
 \begin{equation}
   \label{DnuDeDe} 
   \left.D^3v\right|_x(e,e,\nu_j)
   =0\quad \text{ for every $e\in T\Sigma_j$ and
     $x\in{\Sigma_j}.$} 
\end{equation}
Since the normal
 $\nu_j$
 is an eigenvector of $D^2v|_{x_0}$, and all eigenspaces of the matrix
 $D^2v|_{x_0}$ are orthogonal, the eigenvector $\nu_j$ is either in $E$
 or belongs to the orthogonal space $E^\perp$. If
 $\nu_{j}\in E^{\perp}$, then $e_i$ is orthogonal to $\nu_j$ and so is
 in $T\Sigma_j$ for each $i\in\{1,\cdots,k\}$.  Then~\eqref{DnuDeDe}
 implies
 \begin{displaymath}
   D_{\nu_j} f (x) =D_{\nu_j}\left(\sum_{i=1}^{k} \left.D^2v\right|_x(e_i,e_i)
   \right)=0
\end{displaymath}
for every $x\in B\cap\Sigma_{j}$. On the other hand, if $\nu_j\in E$, then  
\begin{align*}
  D_{\nu_j} f(x)& = D_{\nu_j} \left(\sum_{i=1}^{k} \left.D^2u\right|_x(e_i,e_i)
                  \right)= D_{\nu_j} \left(\Delta u -
                  \sum_{i=k+1}^{d} \left.D^2u\right|_x(e_i,e_i)
                  \right)\\
                & = D_{\nu_j} \left(
                  0 - \sum_{i=k+1}^{d} \left.D^2u\right|_x(e_i,e_i)
                   \right)=0
\end{align*} 
for every $x\in B\cap\Sigma_{j}$, where
$\lbrace e_{k+1},\dots,e_d\rbrace$ is a basis for
$E^\perp\subseteq \nu_j^\perp=T\Sigma_j$, and we again use
\eqref{DnuDeDe}. 

By Remark~\ref{rem:1}, the set $\Omega\cap B\cap B_r(x_0)$ coincides
with $x_0+(\Gamma_{x_0}\cap B_r(0))$ for sufficiently small $r>0$.
Equations \eqref{eq:fx0}-\eqref{Neumann f} (and that fact that $f$ is
continuous on $B$ since $v\in C^2(B)$) allow us to apply
Proposition~\ref{prop:cone-MP} to the function $\tilde f(z)=f(x_0+rz)$ on
$B_1(0)\cap\Gamma_{x_0}$ to infer that $f$ is identically zero on a
neighbourhood of $x_0$.  We conclude that the set where $f$ vanishes
is a non-empty, open, and closed subset of $B$, hence equal to $B$.  It
follows from \eqref{eq:26} that $\tr_E D^2v\equiv k\Lambda$ on $B$.
Since $D^2v\leq \Lambda I$ on $B$, this implies that
$D^2v(e_i,e_i) = \Lambda$ on $B$ for all $i=1,\dots, k$ and so,
\begin{equation}
  \label{eq:11}
D^2v|_x(e,e) = \Lambda\quad\text{for all $x\in B$ and $e\in \S^{d-1}\cap E$.}
\end{equation}
In particular $E$ is contained in the $\Lambda$-eigenspace of $D^2v|_x$ for every $x\in B$.
Since we chose $x_0\in B$ such that $k$ is the maximal dimension of
the $\Lambda$-eigenspace of $D^2v|_x$ over all $x\in B$, we can
conclude that $E$ is the $\Lambda$-eigenspace of $D^2v|_x$ for every
$x\in B$.  It then also follows that 
\begin{equation}
  \label{eq:12}
  D^2v|_x(e,\hat e)
  = 0\quad\text{for all $x\in
    B$, $e\in E$, and $\hat e\in E^\perp$.}
\end{equation}
Now, writing $x = \pi_E(x)+\pi_{E^{\perp}}(x)$, integrating \eqref{eq:11} along directions in $E$ yields
\begin{displaymath}
v(x) =
v(\pi_E(x_{0})+\pi_{E^{\perp}}(x))+Dv(\pi_E(x_{0})+\pi_{E^{\perp}}(x))\, \pi_E(x-x_0)+\frac{\Lambda}{2}\abs{\pi_E(x-x_{0})}^2.
\end{displaymath}
By~\eqref{eq:12}, differentiating
$Dv(\pi_E(x_0)+\pi_{E^{\perp}}(x))$ in a direction tangent to
$E^\perp$ gives zero, so
$Dv(\pi_E(x_0)+\pi_{E^{\perp}}(x))$ is independent of
$\pi_{E^{\perp}}(x)$ and in particular is equal to $Dv(x_0)$.
Defining 
$g(\pi_{E^{\perp}}(x))=v(\pi_E(x_{0})+\pi_{E^{\perp}}(x))$ shows that $v$ is
of the form~\eqref{eq:10}.  

If $k=\dim(E)=d$ then $E^\perp$ is trivial and there is nothing further to prove.  Otherwise it
follows that $g$ is a $C^2$ function on $\pi_{E^{\perp}}(B)\subset\overline{\Omega^\perp}$, and we have
\begin{displaymath}
0 = \Delta v+\mu = \Delta g + k\Lambda +\mu
\end{displaymath}
and for $\nu_i\in E^\perp$ we have
\begin{displaymath}
0= D_{\nu_i}v+\gamma_i = D_{\nu_i}g+\gamma_i.
\end{displaymath}
That is, $g$ is a $C^2$ solution of an equation of the form
\eqref{eq:3-generalised} on the open subset $\pi_{E^{\perp}}(B)$ of $\overline{\Omega^\perp}\subseteq E^\perp$. 
By Lemma~\ref{lem:constNeumann},
 $\nu_j$ is an eigenvector of $H_{v}(x)$ at every point
$x\in\Sigma_j\cap B$, and hence the normals $\nu_j$ are either in $E$
or $E^\perp$.
Then we can write
\begin{align*} 
  \Omega\cap B &= \bigcap_{i=1}^{m}\Big\{ x\in B\;\Big\vert\; x\cdot\nu_i<b_i\Big\}\\
     &= \bigcap_{i:\ \nu_i\in E}\Big\{ x\in B\;\Big\vert\;
           x\cdot\nu_i<b_i\Big\} 
          \bigcap\,\bigcap_{i:\ \nu_i\in E^\perp}\Big\{ x\in B\;\Big\vert\; x\cdot\nu_i<b_i\Big\}\\
    &=  \bigcap_{i:\ \nu_i\in E}\Big\{ x\in B\;\vert\;
            \pi_E(x)\cdot\nu_i<b_i\Big\}
           \bigcap\,\bigcap_{i:\ \nu_i\in E^\perp}\Big\{ x\in B\;\Big\vert\; \pi_{E^{\perp}}(x)\cdot\nu_i<b_i\Big\}\\
   &=\Big\{ x\in B\;\Big\vert\; \pi_E(x)\in\Omega^E,\ \pi_{E^{\perp}}(x)\in\Omega^\perp\Big\},
\end{align*}
where
\begin{displaymath}
  \Omega^E = 
  \bigcap_{i:\ \nu_i\in E}\Big\{ x\in E\;\big\vert\;
   x\cdot\nu_i<b_i\Big\}\text{ and }
  \Omega^\perp = \bigcap_{i:\ \nu_i\in E^\perp}\Big\{ x\in E^\perp\Big\vert\; x\cdot\nu_i<b_i\Big\}.
\end{displaymath}
This completes the proof of Lemma \ref{lem:Hessian-SMP}.
\end{proof}

Now, we can give the proof of Theorem~\ref{lemma 1.2}:

\begin{proof}[Proof of Theorem~\ref{lemma 1.2}]
  By Lemma \ref{lem:Hessian-SMP} (applied with
  $B=\overline{\Omega}$), we have that $v$ is of the
  form~\eqref{eq:10} for some solution $g$ of~\eqref{eq:3-generalised}
  on $\Omega^\perp$.  If $k=\dim(E)=d$ then $v$ is quadratic and there
  is nothing further to prove.  Otherwise the function $g$ is a $C^2$
  solution of an equation of the form \eqref{eq:3-generalised} on
  $\Omega^\perp$ in $\R^{d-k}$.  By the inductive hypothesis, $g$ is
  a quadratic function, and therefore $v$ is also quadratic.  This
  completes the induction and the proof of Theorem~\ref{lemma 1.2}.
\end{proof}

%
%
%
%


\section{Tame domains}
\label{sec: quadratic}

Our aim over the next several sections is to prove that concave
solutions of \eqref{eq:3-generalised} are twice continuously
differentiable.  The result of the previous section then implies that
such solutions are quadratic functions.

Recall that a function $f$ is \emph{semi-concave} if
there exists $C\in\R$ such that the function $x\mapsto f(x)-C|x|^2$ is concave.

Over the course of the next three sections we will prove the following:

\begin{theorem} \label{semiconcave implies C2} Let $\Omega$ be a
  polyhedral domain in $\R^{d}$ with faces
  $\Sigma_{1}, \dots, \Sigma_{m}$, and $v$ be a weak solution of
  problem \eqref{eq:3-generalised} for some $\mu$,
  $\gamma_{1},\dots, \gamma_{m}\in \R$. If $v$ is semi-concave in
  $\Omega$, then $v\in C^{2}(\overline{\Omega})$.
\end{theorem}

The main difficulty in proving that $v\in C^{2}(\overline{\Omega})$ is
to understand the behaviour of $v$ at points on the boundary
$\partial\Omega$, particularly where two or more of the faces
$\Sigma_{i}$ intersect. We begin by using the series
expansion~\eqref{eq:hhdecomp} to understand the behaviour of $v$ near
a boundary point $x_0$ in terms of homogeneous Neumann harmonic
functions on the tangent cone $\Gamma_{x_0}$.  A crucial step in our
argument will be to prove the result that homogeneous degree two
Neumann harmonic functions must be quadratic if they have bounded
second derivatives.  We will accomplish this in the next section.  In
the rest of this section we will establish that this result is
sufficient to prove regularity.

\begin{definition}\label{def:tame}
  For given vectors $\nu_{1}, \dots, \nu_{m}\in \R^{d}$, a polyhedral cone
  \begin{equation}
    \label{eq:poly-cone}
    \Gamma = \bigcap_{i=1}^m\Big\{x\in \R^{d}\,\Big\vert\; x\cdot\nu_i< 0\Big\}
\end{equation}
is called \emph{tame} if every degree two homogeneous harmonic
function $v\in C^{1,1}(\overline{\Gamma})$ with homogeneous Neumann boundary
condition on $\partial\Gamma$ is quadratic. If $\Omega$ is a
polyhedral domain in $\R^d$ and $B$ is a relatively open subset of
$\overline{\Omega}$, then $B$ is called \emph{tame} if the tangent
cone $\Gamma_x$ is tame for every $x\in B$.
\end{definition}

The significance of tameness for our argument is captured by the
following preliminary theorem which is the main result of this
section.

\begin{theorem}\label{prop:nobadreg}
  Let $\Omega$ be a polyhedral domain in $\R^{d}$ and $B$ a relatively
  open tame subset of $\overline{\Omega}$. Then every weak solution $w\in
  C^{1,1}(B)\cap H^{1}(B)$ of problem
    \begin{equation}\label{eq:weak-Neumann-B}
    \begin{cases}
      \Delta w = 0&\text{on } \Omega\cap B,\\
      D_\nu w = 0&\text{on }\partial\Omega\cap B
    \end{cases}
  \end{equation}
  is in $C^2(B)$.
\end{theorem}

\begin{proof}[Proof ofTheorem~\ref{prop:nobadreg}]
  We first establish that the harmonic function $w$ is twice
  differentiable at each point $x_0\in B$, using the
  decomposition \eqref{eq:hhdecomp}.  Since the restriction of
  $B$ to a sufficiently small ball about $x_0$ agrees with
  a translate of the tangent cone to $\Omega$ at $x_0$, it is sufficient to
  consider a Neumann harmonic function defined on a ball about
  the origin in a tame cone $\Gamma$.

\begin{lemma}\label{lem:twicediff}
  Let $\Gamma$ be a tame polyhedral cone in $\R^d$ with outer unit
  face normals $\nu_{1},\dots,\nu_{m}$, and let $B=B_{1}(0)\cap \overline{\Gamma}$,
  where $B_{1}(0)$ is the open unit ball in $\R^{d}$. Then there exist
  constants $C>0$ and $\gamma\in (0,1)$ depending only on $\Gamma$ such
  that for every weak solution $w\in C^{1,1}(B)\cap H^{1}(B)$
  of~\eqref{eq:weak-Neumann-B}, there
  exists a linear functional $L : \R^{d}\to \R$ with $\{\nu_{1},\dots,\nu_{m}\}\subseteq
  \textrm{ker}(L)$ and a symmetric bilinear form $\mathfrak{a} : \R^{d}\times \R^{d}\to \R$
  with trace $\textrm{tr}(\mathfrak{a}):=\sum_{i=1}^{d}a(e_{i},e_{i})=0$ such that the following estimate holds:
  \begin{equation}\label{eq:localest}
  \left|w(x)-w(0)-L(x)-\tfrac12 \mathfrak{a}(x,x)\right|\le
  C\,\norm{w}_{L^{\infty}(B\cap\Gamma)}\,\abs{x}^{2+\gamma}\quad
  \text{for every $x\in B_{1/2}(0)\cap\Gamma$.}
  \end{equation}
 Consequently $w$ has derivatives
  up to second order at $x=0$, with $Dw|_0=L$ and $D^2w|_{0}=\mathfrak{a}$.
\end{lemma}

\begin{proof}[Proof of Lemma \ref{lem:twicediff}]
  We only need to consider the case $d\ge 2$. By
  Proposition~\ref{lem:series-expansion}, $w$ has the series
  decomposition~\eqref{eq:7series}. Since in the
  series~\eqref{eq:7series}, $\varphi_0\equiv1$ and $\beta_0=0$, we
  have $w(0)=f_0$. Thus, writing in polar coordinates $x=rz$ for $r>0$ and
  $z\in \S^{d-1}$,
 \begin{displaymath}
   w(rz) = w(0) + \sum_{i>1}f_i\,r^{\beta_i}\,\varphi_i(z) \quad\textrm{for every $rz\in
    B\cap \Gamma$,} 
\end{displaymath}
The second derivatives $D^2\psi_i$ 
of $\psi_{i}(x):=\abs{x}^{\beta_i}\varphi_i(x/\abs{x})$ are homogeneous
of degree $(\beta_i-2)$. In particular, for every $i$ with
$\beta_i<2$, $D^2\psi_i$ is unbounded as $r=\abs{x}$
approaches zero, except in the case where $\beta_i=1$ and $\psi_{i}$
is a linear function. Since $w\in C^{1,1}(B)$, the only non-zero
$\psi_{i}$ with $0<\beta_i<2$ are those with $\beta_i=1$, and these
form a linear function $L$.  Those $\psi_{i}$ satisfy homogeneous Neumann boundary
conditions on $B\cap \partial\Gamma$, implying that $L(\nu_{i})=0$ for every
$i=1,\dots, m$. Now, defining
$v(rz):=\sum_{\beta_i=2}f_i\,r^{2}\,\varphi_i(z)$ for every $rz \in B\cap \Gamma$, one has that
\begin{equation}
  \label{eq:2}
  w(rz) = w(0)+ L(rz)+v(rz)+\sum_{
{\beta_i> 2}}f_i\,r^{\beta_i}\,\varphi_i(z)
\end{equation}
for every $rz \in B\cap \Gamma$. The function $v$ is harmonic and
homogeneous of degree $2$, satisfies $D_{\nu}v=0$ on $\partial\Gamma$
and has bounded second derivatives since they are given by limits of
second derivatives of $w\in C^{1,1}(B)$ as $r\to0^+$. Thus
$v\in C^{1,1}(\overline{\Gamma})$. Since $\Gamma$ is tame, $v$ is
quadratic and since $v$ is a homogeneous quadratic function and so,
there is a symmetric bilinear form $\mathfrak{a}$ on $\R^{d}$ such
that
\begin{displaymath}
 v(x)=\tfrac12\mathfrak{a}(x,x)\quad\text{for every $x\in \overline{\Gamma}$.}
\end{displaymath}
Since $v$ is harmonic, we have that
\begin{displaymath}
  0=\Delta v(x)=\textrm{tr}(\mathfrak{a})\quad\text{for every $x\in
    B\cap \Gamma$.}
\end{displaymath}
Furthermore, since $v$ satisfies homogeneous Neumann boundary
conditions, 
one has that
\begin{displaymath}
  0 = D_{\nu_i}v(x) = Dv|_x(
  \nu_{i})=\mathfrak{a}(x,\nu_i)\quad\text{for every $x\in
    \Sigma_{i}$ and $i=1,\dots,m$.}
\end{displaymath}
Differentiating the last equality in any direction $e\in
T_x\Sigma_{i}, = (\nu_i)^\perp$, we see that
\begin{displaymath}
  0 = \mathfrak{a}(e,\nu_i)\quad\text{for every $e\perp \nu_{i}$,}
\end{displaymath}
showing that $\nu_{i}$ is an eigenvector of $\mathfrak{a}$.

Next, defining $\bar\beta=\min\{\beta_i>2:\ f_i\neq 0\}>2$, the
remaining term on the right-hand side in~\eqref{eq:2} has the form
\begin{displaymath}
r^{\bar\beta}\sum_{\beta_i>2}f_ir^{\beta_i-\bar\beta}\varphi_i(z).
\end{displaymath}
Since $f_{i}$ is defined by~\eqref{eq:7coef}-\eqref{eq:7coef-of-f} and
since $\norm{\varphi}_{L^{2}(A)}=1$, we have that
\begin{equation}
  \label{eq:4}
  \abs{f_{i}}\le \norm{w}_{L^{\infty}(B\cap \Gamma)}\quad\text{for every
    $i\ge 1$.}
\end{equation}
Further, by \cite[Corollary~1]{MR639355} and~\eqref{eq:beta-i}, 
\begin{equation}
 \label{eq:5}
  \norm{\varphi_{i}}_{L^{\infty}(A)}\le C\,
    \lambda_i^{\frac{d-1}{4}}\le C\, 2\,\beta_i^{\frac{d-1}{2}}\quad
    \text{for every $i\ge 1$,}
\end{equation}
where $C=C(d)>0$ is a constant. Combining~\eqref{eq:4} and~\eqref{eq:5}, one sees that
\begin{equation}
\label{eq:6}
\left| \sum_{\beta_i\ge 2} f_i\,r^{\beta_i-\bar\beta}\,\varphi_i(z)\right|
\le C\, \norm{w}_{L^{\infty}(B\cap \Gamma)}\,\sum_{\beta_i\ge 2} \beta_i
  ^{\frac{d-1}{2}}  r^{\beta_i-\bar\beta}
\end{equation}
Note, for every $r\in (0,1)$, there is an $N(r)>0$ such that
$f(\beta):=\beta^{(d-1)/2}\, r^{\beta}$ is decreasing on
$[N(r),+\infty)$. Thus, for every $r\in (0,1)$, let $i_{r}\ge 1$ be the first integer satisfying
$\beta_{i_{r}}\ge N(r)+2$. Then
\begin{displaymath}
  \beta^{(d-1)/2}_{i}\, r^{\beta_{i}}\le \beta^{(d-1)/2}_{i_{r}}\,
  r^{\beta_{i_{r}}}\le
  \beta^{\frac{d-1}{2}}_{i_{r}}\,r^{\beta_{i}}\quad\text{for all
    $i\ge i_{r}$.} 
\end{displaymath}
By the eigenvalue estimates due to Cheng and
Li~\cite[Theorem~1]{MR639355} and~\eqref{eq:beta-i}, there is an
integer $i_{\ast}\ge i_{r}$ and a
constant $C=C(\abs{A},d)>0$ such that
\begin{displaymath}
  \beta_{i}\ge \tfrac{1}{\sqrt{d-1}}\sqrt{\lambda_i}\ge C\; i^{\frac{1}{d-1}}\quad\text{for all $i\ge i_{\ast}$.}
\end{displaymath}
Applying this to the last estimate, we see that
\begin{displaymath}
  \beta^{(d-1)/2}_{i}\, r^{\beta_{i}}\le
  \beta^{(d-1)/2}_{i_{r}}\,r^{C\,i^{\frac{1}{d-1}}}\quad\text{for all
    $i\ge i_{\ast}$} 
\end{displaymath}
and so by~\eqref{eq:6}, 
\begin{displaymath}
  \left| \sum_{i\ge i_{\ast}} f_i\,r^{\beta_i-\bar\beta}\,\varphi_i(z)\right|
\le C\, \norm{w}_{L^{\infty}(B\cap \Gamma)}\, \sum_{i\ge i_{\ast}} r^{C\, i^{\frac{1}{d-1}}}.
\end{displaymath}
This shows that the series $S(rz):=\sum_{\beta_i\ge 2}
f_i\,r^{\beta_i-\bar\beta}\,\varphi_i(z)$ converges pointwise on
$B\cap \Gamma$, and uniformly on $B_{1/2}(0)\cap\Gamma$. In
particular, $S$ is bounded on $B_{1/2}(0)\cap\Gamma$ by $C_{1/2}\,
\norm{w}_{L^{\infty}(B\cap \Gamma)}$ for some constant
$C_{1/2}>0$. Applying this to~\eqref{eq:2} and noting that
$\bar{\beta}>2$ yields the desired estimate~\eqref{eq:localest}. The
fact that $Dw(0)=L$ and $D^{2}w(0)=\mathfrak{a}$ follow from this estimate.

\end{proof}

%
%
%
%
%

\noindent\emph{Continuation of the Proof of Theorem~\ref{prop:nobadreg}.} The remaining
difficulty in the proof of Theorem~\ref{prop:nobadreg} is to confirm
continuity of the second derivative.  As before in
Lemma~\ref{lem:twicediff}, it suffices to consider a Neumann harmonic
function $w$ on a cone, and to establish the continuity of the second
derivative at the origin.  Accordingly, we fix a point $x_0$ in
$\partial\Omega\cap B$, and $r_0>0$ sufficiently small to ensure that
\begin{displaymath}
\Omega\cap B_{r_0}(x_0) = \Big\{x_0+x\,\Big\vert\;x\in\Gamma_{x_0},\
\abs{x}<r_0\Big\},
\end{displaymath}
where $\Gamma_{x_0}$ is the tangent cone to
$\Omega$ at $x_0$. To show that the second derivatives of $w$ are continuous at $x_0$,
it is sufficient to show that the Neumann harmonic function
\begin{displaymath}
\hat{w}(x) = \frac{u(x_0+r_0x)}{\norm{w}_{L^{\infty}(B_{r_0}(x_0)\cap\overline{\Gamma}_{x_{0}})}}\quad
  \text{for every $x\in B\cap \Gamma$}
\end{displaymath}
has continuous second derivative at the origin, where $B=B_{1}(0)$ is
the open unit ball and $\Gamma$ a polyhedral cone with vertex at the
origin. 

Now  we label parts of $\Gamma$ according to the number
of faces which intersect.  Recall    the faces of 
$\Omega$ are $\Sigma_i$ with outward unit normal
vectors $\nu_{i}$ for every for $i=1, \cdots,m$. Then
\begin{displaymath}
  \Gamma^{(k)} := 
  \bigcup_{\substack{{\mathcal S}\subset\{1,\cdots,m\}\\|{\mathcal S}|=k}}
  \left(\bigcap_{i\notin{\mathcal S}}\left\{x\,\Big\vert\ 
      x\cdot\nu_i\le 0\right\}\right)\cap\left(\bigcap_{j\in{\mathcal S}}
    \left\{x\,\Big\vert\ x\cdot\nu_j=0\right\}\right)
\end{displaymath}
denotes the set of all $x\in \Gamma$ where $k$ faces intersect. Thus
$\Gamma^{(0)}=\overline{\Gamma}$,
$\Gamma^{(1)}=\partial\Gamma$, and
$0\in \Gamma^{(m)}$.\medskip 

We now proceed by (decreasing) induction on $k$, starting with $k=m$:

\begin{proposition}\label{lem:nobadreglem1}
  Let $\Gamma$ be a tame polyhedral cone in $\R^d$ and $B=B_{1}(0)\cap
  \overline{\Gamma}$. Then there exist
  constants $C>0$ and $\gamma \in (0,1)$ depending only on $\Gamma$ such
  that for every weak solution $w\in C^{1,1}(B)\cap H^{1}(B)$
  of~\eqref{eq:weak-Neumann-B},
    \begin{equation}
      \label{eq:wosc1}
        \left\vert w(y)-w(x)-Dw|_x
          (y-x)-\tfrac{1}{2}D^2w|_x(y-x,y-x)\right\vert
       \le   C\,\abs{y-x}^{2+\gamma}
  \end{equation}
  for every $x\in B_{1/2}(0)\cap\Gamma^{(m)}$ and $y\in B\cap\overline{\Gamma}$.
\end{proposition}

For the proof of Proposition~\ref{lem:nobadreglem1} we will use the
following  auxiliary result, which will be also useful several
times later.

\begin{lemma}\label{lem:coeffbound}
  Let $\mathfrak{a}$ be a symmetric bilinear form and $L$ a linear
  functional on $\R^{d}$, and let $c\in \R$. Define
  \begin{displaymath}
    q(x)=\mathfrak{a}(x,x)+L(x)+c\quad\text{for every $x\in \R^{d}$.}
 \end{displaymath}
 If for $r>0$ and $M\ge 0$, one has that $\sup_{x\in \overline{B}_r(0)}|q(x)|\le M$, then
 $|c|\le M$, $\norm{L}\le 2M/r$, and the
 eigenvalues $\lambda_{i}$ of $\mathfrak{a}$ satisfy $\abs{\lambda_{i}}\le 2M/r^2$.
\end{lemma}

\begin{proof}
  Choosing $x=0$ gives $|c|\le M$, implying that
  $|\mathfrak{a}(x,x)+L(x)|\le 2\,M$ for all $x\in \overline{B}_r(0)$. Further, for
  $x\in \overline{B}_{r}(0)$, we have (by replacing $x$ by $-x$) that
  $|\mathfrak{a}(x,x)-L(x)|\le 2\,M$, and hence (taking sums and
  differences) $|a(x,x)|\le 2\,M$ and $|L(x)|\le 2\,M$. Thus,
  $\abs{\lambda_{i}}\le 2M/r^2$ follows by choosing $x/r$ to be a normalised
  eigenvector of $\mathfrak{a}$, and $\norm{L}\le 2M/r$ follows by choosing
  $x\in \partial B_{r}(0)$ with $L(x)=\norm{L}\abs{x}$.
\end{proof}

In order to apply the lemma above, we need a suitable ball.  This is
provided by the following:

\begin{lemma}\label{lem:closeball}
  Let $\Omega$ be a bounded open convex set in $\R^d$.  Then there
  exist $\sigma>0$ and $R>0$ such that for every
  $x\in\overline{\Omega}$ and every $r\in(0,R)$, there exists
  $\hat{x}\in \Omega$ such that the open
  ball $B_{\sigma r}(\hat{x})$ is contained in $B_r(x)\cap\Omega$.
\end{lemma}

\begin{proof} 
  Let $\rho_-$ be the inradius and $x_-$ an incentre of $\Omega$,
  and let $\rho_+$ be the circumradius of $\Omega$. Then, for
  $R=2\rho_+$ (so that $\Omega$ is included in $B_R(x)$ for any
  $x\in\overline{\Omega}$) and $\sigma = \frac{\rho_-}R$, one has that
 \begin{equation}\label{eq:insideball}
 B_{\sigma R}(x_-)=B_{\rho_-}(x_-)\subseteq \Omega=\Omega\cap B_R(x)
 \end{equation} 
 for any $x\in\overline{\Omega}$. Now, for fixed
 $x\in\overline{\Omega}$ and $r\in(0,R)$, let 
 \begin{displaymath}
   T_\lambda(y) = x+\lambda(y-x)\quad\text{for $y\in \R^{d}$ and $\lambda=\tfrac{r}{R}\in (0,1)$.}
 \end{displaymath}
 Since $T_{\lambda}(y)= (1-\lambda)x+\lambda y$, convexity of
 $\Omega$ implies that $T_\lambda(\Omega)\subseteq \Omega$. Thus,
 by~\eqref{eq:insideball} and since $T_\lambda(B_R(x))=B_r(x)$, one has that
  \begin{displaymath}
 B_{\sigma r}(x_{-})=T_\lambda(B_{\sigma R}(x_-))\subseteq T_\lambda\left(\Omega\cap B_R(x)\right)
 =T_\lambda(\Omega)\cap T_\lambda(B_R(x))\subseteq \Omega\cap B_r(x),
\end{displaymath}
as claimed.
\end{proof}

With these preliminaries, we can prove the base case
of our (decreasing) induction.

\begin{proof}[Proof of Proposition~\ref{lem:nobadreglem1}]
  For $x\in B_{1/2}(0)\cap\Gamma^{(m)}$, the tangent cone $\Gamma_{x}$ to $\Omega$ at
  $x$ agrees with $\Gamma$ at the origin. 
  Thus, we can apply Lemma~\ref{lem:twicediff} to the function
  \begin{displaymath}
    w^{x}(v) = w\left(x+\frac{v}{2}\right)\quad\quad\text{for every
      $v\in B_{1}\cap \Gamma$}
  \end{displaymath}
 and obtain that
 \begin{displaymath}
\left|w^{x}(v)-w^{x}(0)-Dw^{x}|_0(v)-\tfrac12D^2w^{x}|_0(v,v)\right|
\le C\,|v|^{2+\gamma}
\end{displaymath}
for all $v\in B_{1/2}(0)\cap\Gamma_x$. Now, setting $v=2(y-x)$ for
$y\in  B_{1/4}(x)\cap\Gamma$ and using the definition of $w^{x}$ we
obtain that estimate~\eqref{eq:wosc1} holds
for all $y\in B_{1/4}(x)\cap\Gamma$. To derive the same inequality for
$y\in B_1(0)\setminus B_{1/4}(x)$, we first derive bounds on the size
of $Dw|_x$ and $D^2w|_x$, using Lemma~\ref{lem:coeffbound}: by Lemma
\ref{lem:closeball} applied to $\Omega=B\cap \Gamma$ and $r=1/4$,
there are $\sigma>0$ and $x_{\ast}\in B\cap \Gamma$ such that the open
ball $B_{\sigma/4}(x_{\ast})$ is contained in
$B_{1/4}(x)\cap\Gamma$. Due to estimate~\eqref{eq:wosc1} and since $w$
is bounded on $\overline{B}_{\sigma/4}(x_{\ast})$, there is a $C>0$
such that
\begin{displaymath}
  \sup_{y\in \overline{B}_{\sigma/4}(x_{\ast})}\left|Dw|_x(y-x)+
  \tfrac12D^2w|_x(y-x,y-x)\right|\le C.
\end{displaymath}
For $y\in \overline{B}_{\sigma/4}(x_{\ast})$, setting $v=y-x_{\ast}$,
this shows that the quadratic function
\begin{displaymath}
  q(v):=Dw|_x\,(v+x_{\ast}-x)+\tfrac12D^2w|_x(v+x-x_{\ast},v+x-x_{\ast})
\end{displaymath}
is bounded on $\overline{B}_{\sigma/4}(0)$ and hence by
Lemma~\ref{lem:coeffbound}, the coefficients 
of $q$ are boun\-ded. Moreover, the quadratic part of $q$ gives that
the eigenvalues $\lambda_{i}(x)$ of $D^2w|_x$ satisfy
$\abs{\lambda_{i}(x)}\le 32 C/\sigma^{2}$. 
Since $D^{2}w^{x}|_0=\tfrac{1}{4}D^{2}w|_x$ and $D^{2}w^x|_0$ is
symmetric by Lemma~\ref{lem:twicediff}, 
the Hessian $D^2w|_x$ is symmetric and so, the bound on
$\lambda_{i}(x)$ implies that $\norm{D^2w|_x}\le 32 C/\sigma^{2}$. 
Further, the linear part of $q$ gives that
\begin{displaymath}
\norm{Dw|_x+ D^2w|_x(x_{\ast}-x)}\le 8 C/\sigma
\end{displaymath}
and since $\norm{D^2w|_x}\le 32 C/\sigma^{2}$ and
$\abs{x-x_{\ast}}<1/4$, 
this yields that $\norm{Dw|_x} \le 16 C/\sigma$. 
Now, if $y\in B_1(0)\setminus B_{1/4}(x)$, then we have
$\frac14\leq |y-x|\leq \frac32$, and so, the bounds on $w(y)$, $w(x)$, $Dw|_x$,
$D^2w|_x$, $|y-x|$ and $|y-x|^{-1}$ show that
\begin{align*}
&\left|w(y)-w(x)-Dw|_x(y-x)-\tfrac12D^2w|_x(y-x,y-x)\right|
 \le C\le C\,|y-x|^{2+\gamma},
\end{align*}
as required.  
\end{proof}

Next, we establish the inductive step:

\begin{proposition}\label{lem:nobadreglem2}
  Let $\Gamma$ be a tame polyhedral cone in $\R^d$ and $B=B_{1}(0)\cap
  \overline{\Gamma}$. Suppose that there
  exists a $\gamma\in(0,1)$ such that if a weak solution
  $w\in C^{1,1}(B)\cap H^{1}(B)$ of~\eqref{eq:weak-Neumann-B} satisfies
\begin{equation}\label{eq:nearbyclose}
    \left|w(y)-w(x)-Dw|_x(y-x)-\tfrac12D^2w|_x(y-x,y-x)\right|\\
   \lesssim |y-x|^{2+\gamma}
\end{equation}
for every $x\in B_{1/2}(0)\cap\Gamma^{(k)}$ and
$y\in B_{1}(0)\cap\overline{\Gamma}$, then $w$ also satisfies~\eqref{eq:nearbyclose} for all
$x\in B_{1/2}(0)\cap\Gamma^{(k-1)}$ and
$y\in B_{1}(0)\cap\overline{\Gamma}$.
\end{proposition}

To prove this proposition, we intend to apply Lemma
\ref{lem:twicediff} about
$x\in (B_{1/2}(0)\cap\Gamma^{(k-1)})\setminus\Gamma^{(k)}$. In
order to do this we need to estimate the \emph{cone radius}
\begin{equation}
  \label{eq:coneradius}
  \rho(x):=\sup\Big\{r>0\,\Big\vert\, B_r(x)\cap 
  \overline{\Gamma}=x+(B_r(0)\cap\overline{\Gamma}_x)\,\Big\},
\end{equation}
where $\Gamma$ is a polyhedral cone in $\R^{d}$ with vertex at the
origin and $\Gamma_{x}$ the tangent cone to $\Gamma$ at
$x\in \partial\Gamma\setminus\{0\}$. This is supplied
by the following result.

\begin{lemma}\label{lem:cone-rad-vs-dist}
  There exists $\sigma>0$ such that
  \begin{equation}
    \label{eq:18}
    \rho(x)\geq \sigma d(x,\Gamma^{(k)})\quad\text{for all
    $x\in \Gamma^{(k-1)}\setminus\Gamma^{(k)}$.} 
  \end{equation}
\end{lemma}


We say that a convex cone $\Gamma$ in $\R^d$ \emph{admits a linear
  factor $E$} if there exist a linear subspace $E$ of $\R^{d}$ of positive dimension
with orthogonal complement $E^\perp$ in $\R^d$ and a convex cone
$\tilde\Gamma$ in $E^\perp$ such that
\begin{displaymath}
  \Gamma = \Big\{x\in\R^d\,\Big\vert\; \pi_{E^{\perp}}(x)\in\tilde\Gamma\Big\},
\end{displaymath}
where $\pi_{E^{\perp}}$ is the orthogonal projection onto $E^\perp$.  In this
situation, we write $\Gamma = \tilde\Gamma\oplus E$.\medskip  

The following  observation is used in the inductive step of our argument, and
will also be used later in the paper.

\begin{lemma}
  \label{lem:TCnotvertex}
  Let $\Gamma$ be a polyhedral cone in $\R^d$ with vertex at the
  origin and outer unit face normals $\nu_1,\cdots,\nu_m$. Let
  $x_0\in\partial\Gamma\setminus\{0\}$.  Then the tangent cone
  $\Gamma_{x_0}$ to $\Gamma$ at $x_0$ has a linear factor $\R x_{0}$,
  and so had the form
  $\Gamma_{x_0} = \tilde\Gamma\oplus \R x_0$, where $\tilde\Gamma$ is
  the polyhedral cone in the $(d-1)$-dimensional subspace
  $(\R x_0)^\perp$ of $\R^d$ defined by
  \begin{equation}
    \label{eq:15}
    \tilde\Gamma = \bigcap_{i\in\mathcal I(x_0)}\Big\{x\in (\R x_{0})^{\perp}\,\Big\vert\,
    x\cdot\nu_i < 0\Big\}\quad\text{with}\quad\mathcal{I}(x_0):=\Big\{i\in \{1,\dots,m\}\,\Big\vert\,
  x_0\cdot\nu_i=0\Big\}.
  \end{equation}
\end{lemma}

\begin{proof}
  Since
  \begin{displaymath}
    \Gamma_{x_0}=\bigcap_{i\in{\mathcal I}(x_0)}\Big\{x\in\R^{d}\,\Big\vert\, x\cdot\nu_i<0\Big\},
  \end{displaymath}
  and $i\in \mathcal{I}(x_0)$ implies $\nu_{i}\cdot x_{0}=0$, we have that $\nu_{i}\in
  (\R x_{0})^{\perp}$ for all $i\in \mathcal{I}(x_0)$. Therefore 
  $\Gamma_{x_0}=\tilde\Gamma\oplus\R x_0$, where
  $\tilde\Gamma$ is given by~\eqref{eq:15}.
\end{proof}

\begin{proof}[Proof of Lemma~\ref{lem:cone-rad-vs-dist}]
  If there is no such $\sigma>0$ such that~\eqref{eq:18} holds, then
  there exists a sequence $(x_{n})_{n\ge 1}$ of points
  $x_n\in \Gamma^{(k-1)}\setminus\Gamma^{(k)}$ such that
  \begin{equation}
    \label{eq:20}
    \frac{\rho(x_n)}{d(x_n,\Gamma^{(k)})}\to 0.
  \end{equation}
  Since both $\rho(\cdot)$ and $d(\cdot,\Gamma^{(k)})$ are homogeneous of degree one,
  we can scale $x_n$ so that $x_n\in \S^{d-1}\cap\overline{\Gamma}$.

  We first exclude the possibility that there are $\alpha>0$ and a subsequence
  $(x_{n'}) _{n'\ge 1}$ of $(x_{n}) _{n\ge 1}$ such that $d(x_{n'},\Gamma^{(k)})\ge \alpha$
  for all $n'\ge 1$. Otherwise, for such a subsequence
  $(x_{n'}) _{n'\ge 1}$ of $(x_{n}) _{n\ge 1}$, one has that $\rho(x_{n'})\to 0$. Since $x_{n'}\in \S^{d-1}\cap
  (\Gamma^{(k-1)}\setminus\Gamma^{(k)})$, we can extract another
  subsequence of $(x_{n'}) _{n'\ge 1}$ which we denote, for simplicity, again by
  $(x_{n'}) _{n'\ge 1}$ such that $x_{n'}$ converges to a point
  $\bar x\in \S^{d-1}\cap\Gamma^{(k-1)}\setminus\Gamma^{(k)}$.  Label
  the faces so that $\bar x\cdot\nu_i$ is in non-increasing order.
  Then, since $\bar x\in\Gamma^{(k-1)}\setminus\Gamma^{(k)}$, we have
  $\bar x\cdot\nu_i=0$ for $i=1,\dots,k-1$ and $\bar x\cdot\nu_k<0$.  Since
  the function $x\mapsto x\cdot\nu_i$ is continuous, any point $x$ in
  $\Gamma^{(k-1)}\setminus\Gamma^{(k)}$ sufficiently close to $\bar x$
  also satisfies $x\cdot\nu_i=0$ for $i=1,\dots,k-1$ and
  $x\cdot\nu_i<\frac12\bar x\cdot\nu_k<0$ for $i\geq k$.  It follows
  that
  \begin{displaymath}
    \Gamma_x=\Gamma_{\bar x} = \bigcap_{i=1}^{k-1}\Big\{z\,\Big\vert\;
    z\cdot\nu_i<0\Big\},
  \end{displaymath}
  so the tangent cone is constant and hence the cone radius $\rho$ is
  continuous on $\Gamma^{(k-1)}$ near $\bar x$.  In particular, 
  we have that $\rho(x_{n'})$ is bounded below, contradicting the fact
  that $\rho(x_{n'})\to 0$.

  The remaining possibility is that $d(x_n,\Gamma^{(k)})$ converges to
  zero.  Passing to a subsequence, we have convergence to a point
  $\bar x\in \S^{d-1}\cap\Gamma^{(k)}$.  In particular for $n$
  sufficiently large $x_n\in B_{\rho(\bar x)}(\bar x)\cap \bar \Gamma$.

  In Lemma~\ref{lem:TCnotvertex}, we have observed that since
  $\bar x\neq 0$, the tangent cone $\Gamma_{\bar x}$ is the
  product $\Gamma_{\bar x}=\tilde{\Gamma}\oplus \R\bar{x}$, where
  $\tilde{\Gamma}$ is a polyhedral cone in the $(d-1)$-dimensional
  subspace $(\R\bar x)^{\perp}$. Thus, it follows that both $\rho(x_n)$ and
  $d(x_n,\Gamma^{(k)})$ are invariant under translation in the
  $\bar x$-direction and homogeneous of degree one under rescaling
  about $\bar x$. Therefore, we can replace $x_n$ by
  \begin{displaymath}
    \tilde x_n=\frac{\left(x_n-\frac{x_n\cdot \bar x}{|\bar x|^2}\bar
      x\right)}{\left|\left(x_n-\frac{x_n\cdot \bar x}{|\bar x|^2}\bar
        x\right)\right|}\in \left(\tilde\Gamma\times\{0\}\right)\cap \S^{d-1}
 \end{displaymath}
 and still have a sequence $(\tilde{x}_{n}) _{n\ge 1}$ satisfying
 $\tilde{x}_{n}\in \tilde\Gamma\cap
 (\Gamma^{(k-1)}\setminus\Gamma^{(k)})$ and~\eqref{eq:20} where
 $x_{n}$ is replaced by $\tilde x_{n}$.

Now, we repeat the above argument inductively, with $\Gamma$ replaced
by $\tilde\Gamma$.  At each application, the dimension of the cone
reduces by one, which is impossible since $\Gamma$ is
finite-dimensional. This contradicts our assumption that there is no
positive $\sigma$ satisfying the statement of Lemma
\ref{lem:cone-rad-vs-dist}, so the proof of the Lemma is complete.
\end{proof}

Now, we can complete the proof of the inductive step.

\begin{proof}[Proof of Proposition~\ref{lem:nobadreglem2}]
  Fix $x\in (B_{1/2}(0)\cap\Gamma^{(k-1)})\setminus\Gamma^{(k)}$. Let
  $\tilde x\in \Gamma^{(k)}$ be the closest point to $x$ in
  $\Gamma^{(k)}$ satisfying $\abs{x-\tilde{x}}<1/2$.  We claim
  that $\tilde x\in B_{1/2}(0)$.  As $\lambda\tilde x$ is in
  $\Gamma^{(k)}$ for $\lambda>0$, $g(\lambda):=|x-\lambda \tilde x|^2$
  is minimised at $\lambda=1$, and so
  $0=g'(1)=-2(x-\tilde{x})\cdot\tilde{x}$. Since $x-\tilde{x}$ and $\tilde{x}$ are orthogonal,
  \begin{displaymath}
    |x|^2=|x-\tilde{x}+\tilde{x}|^2=|x-\tilde{x}|^2+|\tilde{x}|^2\ge |\tilde{x}|^2
  \end{displaymath}
  and since $|x|<1/2$, it follows that $|\tilde{x}|<1/2$ as claimed.
  Hence, by hypothesis, $w$ satisfies~\eqref{eq:nearbyclose}
  at $\tilde{x}$. More precisely, 
\begin{equation}
  \label{eq:19}
   \left|w(y)-w(\tilde x)-Dw|_{\tilde x}(y-\tilde x)-  
   \tfrac12D^2w|_{\tilde x}(y-\tilde x,y-\tilde x)\right|\le  C|y-\tilde x|^{2+\gamma}
\end{equation}
for all $y\in B_1(0)\cap\overline{\Gamma}$ for some constant $C>0$ and
$\gamma\in (0,1)$. To make use of this, we define
\begin{displaymath}
\tilde w(y) := w(y)-w(\tilde x)-Dw|_{\tilde x}(y-\tilde x)-\tfrac12D^2w|_{\tilde
  x}(y-\tilde x,y-\tilde x)
\end{displaymath}
for every $y\in B_{1}(0)\cap \overline{\Gamma}$. Then $\tilde{w}$ is
a weak solution of~\eqref{eq:weak-Neumann-B} on $ B_{1}(0)\cap \Gamma$ and by~\eqref{eq:19},
\begin{equation}
  \label{eq:17}
  |\tilde w(y)|\le C|y-\tilde x|^{2+\gamma}\quad\text{for $y\in B_1(0)\cap\overline{\Gamma}$.}
\end{equation}
To proceed, we will apply Lemma \ref{lem:twicediff} about $x$. But
first note that by $\tilde{x} \in\Gamma^{(k)}$, after a possible
re-ordering, we may assume without loss of generality that
$\tilde{x}\cdot\nu_{i}=0$ for all $i=1,\dots, k$ and since
$x\in \Gamma^{(k-1)}\setminus\Gamma^{(k)}$, there must be an
$1\le i_{0}\le k$ such that $x\cdot \nu_{i_{0}}<0$. Now, let $\rho(x)$
be the cone radius around $x$ given by~\eqref{eq:coneradius} and we claim
that 
\begin{equation}
  \label{eq:21}
\rho(x)\le |x-\tilde{x}|.
\end{equation}
If $\rho(x)>\abs{x-\tilde{x}}$, then there is an $\varepsilon>0$ such that
\begin{displaymath}
x+(B_{(1+\varepsilon)\abs{x-\tilde{x}}}(0)\cap
\overline{\Gamma}_{x})=B_{(1+\varepsilon)\abs{x-\tilde{x}}}(x)\cap\overline{\Gamma}
\end{displaymath}
and since
$\tilde{x}\in B_{(1+\varepsilon)\abs{x-\tilde{x}}}(x)\cap\overline{\Gamma}$,
there is a $v\in
B_{\abs{x-\tilde{x}}}(0)\cap\overline{\Gamma}_{x}$ such
that $v=\tilde{x}-x$. Then
$x+(1+\varepsilon) v\in x+(B_{(1+\varepsilon)\abs{x-\tilde{x}}}(0)\cap
\overline{\Gamma}_{x})$ and hence, $x+(1+\varepsilon) v\in\overline{\Gamma}$. However,
\begin{displaymath}
  (x+(1+\varepsilon)v)\cdot\nu_{i_{0}} = x\cdot\nu_{i_{0}} + 
(1+\varepsilon)(\tilde{x}-x)\cdot\nu_{i_{0}} = -\varepsilon x\cdot\nu_{i_{0}}>0,
\end{displaymath}
which contradicts the definition of  $\Gamma$, proving our
claim~\eqref{eq:21}. 
Since $\abs{x-\tilde{x}}<1/2$, 
\begin{displaymath}
  \hat w(y):= \tilde w(x+y\rho(x))\quad\text{for $y\in B_1(0)\cap \Gamma_x$}
\end{displaymath}
is a well-defined function. Moreover, $\hat{w}$ is a weak solution
of~\eqref{eq:weak-Neumann-B} on $B_1(0)\cap \Gamma_x$. Hence, by Lemma
\ref{lem:twicediff}, there is a $\gamma\in (0,1)$ and a $C>0$ such
that
\begin{equation}
  \left|\hat{w}(y)-\hat{w}(0)-D\hat{w}|_0\, y-\tfrac12D^{2}\hat{w}|_0(y,y)\right|\le
  C\, \norm{\hat{w}}_{L^{\infty}(B_{1}(0)\cap\Gamma_{x})}\,\abs{y}^{2+\gamma}
\end{equation}
for $y\in B_{1/2}(0)\cap\Gamma_{x}$. Note, by~\eqref{eq:17} and using~\eqref{eq:21},
\begin{equation}
  \label{eq:22}
\sup_{B_1(0)\cap\overline{\Gamma_x}}\hat w =
\sup_{B_{\rho(x)}(x)\cap\overline{\Gamma}}\tilde w\le 
\sup_{B_{2|x-\tilde x|}(\tilde x)\cap\overline{\Gamma}}\tilde
w\leq C|x-\tilde x|^{2+\gamma}.
\end{equation}
Combining the last two estimates then gives
\begin{displaymath}
\left|\hat w(y)-\hat w(0)-D\hat w|_0(y)-\tfrac12 D^2\hat
  w|_0(y,y)\right| \le C|y|^{2+\gamma}|x-\tilde x|^{2+\gamma}
\end{displaymath}
for $y\in B_{1/2}(0)\cap\Gamma_x$. By the definition of $\hat w$, this
gives
\begin{align*}
&\left|\tilde w(y)-\tilde w(x)-D\tilde w|_x(y-x)-\tfrac12D^2\tilde
  w|_x(y-x,y-x)\right|
  \le C\,\left(\frac{|y-x|}{\rho(x)}\right)^{2+\gamma}|x-\tilde
  x|^{2+\gamma}
\end{align*}
for every $|y-x|<\tfrac12\rho(x)$. Since by Lemma \ref{lem:cone-rad-vs-dist}, there
is a $\sigma>0$ such that
\begin{equation}
  \label{eq:23}
  \rho(x)\ge \sigma |x-\tilde x|,
\end{equation}
we can conclude from the last estimate that
\begin{equation}
\label{eq:wosc2}
\left|\tilde w(y)-\tilde w(x)-D\tilde w|_x(y-x)-\tfrac12D^2\tilde
  w|_x(y-x,y-x)\right|
  \le C\,|y-x|^{2+\gamma}
\end{equation}
for every $|y-x|<\tfrac12\rho(x)$. From this, we deduce bounds on
$D\tilde w|_x$ and $D^2\tilde w|_x$: By Lemma \ref{lem:closeball}
applied to $\Omega=B_{1/2}(x)\cap \overline{\Gamma}$, there are
$x_{\ast}\in B_{1/2}(x)\cap \overline{\Gamma}$ and $\sigma_{\ast}>0$
such that the open ball $B_{\sigma_{\ast}\rho(x)}(x_{\ast})$ is
contained in $B_{\rho(x)}(x)\cap\overline{\Gamma}$.  By~\eqref{eq:22},
we have
\begin{displaymath}
|\tilde w(x)|+|\tilde w(y)|\leq C|x-\tilde
x|^{2+\gamma}\quad\text{for every $y\in B_{\sigma_{\ast}\rho(x)}(x_{\ast})$}
\end{displaymath}
and so, by~\eqref{eq:wosc2},
\begin{displaymath}
  \left|D\tilde w|_x(y-x)+\tfrac12 D^2\tilde w|_x(y-x,y-x)\right|
  \le C|x-\tilde x|^{2+\gamma}
\end{displaymath}
for every $y\in B_{\sigma_{\ast}\rho(x)}(x_{\ast})$. 
Moreover, from the previous application of Lemma~\ref{lem:twicediff} to $\hat{w}$,
we know that the Hessian $D^{2}\hat{w}|_0=\rho^{-2}(x)D^{2}\tilde
w|_x$ is symmetric. Thus Lemma~\ref{lem:coeffbound} yields that
\begin{align}
  \notag
  &\norm{D\tilde w|_x(x_{\ast}-x)+\tfrac12 D^2\tilde
    w|_x(x_{\ast}-x)} \le C\,|x-\tilde x|^{2+\gamma}\\ \label{est:Dtilde w}
  & 
      \lnorm{D\tilde w|_x+ D^2\tilde
        w|_x(x_{\ast}-x)} \le
      \frac{2C |x-\tilde x|^{2+\gamma}}{\sigma_{\ast}\rho(x)}\le C
      |x-\tilde x|^{1+\gamma},\\ \label{est:D2tilde w}
  & \lnorm{D^{2}\tilde w|_x}\le 2
      \frac{4C |x-\tilde
    x|^{2+\gamma}}{\sigma_{\ast}^{2}\rho^{2}(x)}\le C |x-\tilde x|^{\gamma},
\end{align}
where we used the estimate~\eqref{eq:23} in the second inequalities of
both~\eqref{est:Dtilde w}
and~\eqref{est:D2tilde w}. Since 
$|x-x_*|\leq C|x-\tilde x|$, inequality~\eqref{est:Dtilde w} implies that
\begin{equation}
  \label{eq:24}
  \norm{D\tilde w|_x}\le C|x-\tilde x|^{1+\gamma}.
\end{equation}

Next, we establish estimate~\eqref{eq:wosc2} for
$y\in (B_1(0)\setminus B_{\rho(x)/2}(x))\cap\overline{\Gamma}$: On
this set, we have $|x-\tilde x|+|y-\tilde x|\leq C|y-x|$ due
to~\eqref{eq:23} and since $\rho(x)/2\le \abs{y-x}$. Thus,
by~\eqref{eq:17},~\eqref{eq:24}, and~\eqref{est:D2tilde w},
\begin{align*}
&\left|\tilde w(y)-\tilde w(x)-D\tilde w|_x(y-x)-\tfrac12 D^2\tilde w|_x(y-x,y-x)\right|\\
&\quad\le |\tilde w(y)|+|\tilde w(x)|+\norm{D\tilde w|_x}\,|y-x|+\tfrac12\norm{D^2\tilde w|_x}\,|y-x|^2\\
&\quad\le C|y-\tilde x|^{2+\gamma}+ C|x-\tilde x|^{2+\gamma}
  +C|x-\tilde x|^{1+\gamma}|y-x|+C|x-\tilde x|^\gamma|y-x|^2\\
&\quad\le C|y-x|^{2+\gamma},
\end{align*}
as required.  This shows that estimate \eqref{eq:wosc2} holds for
all $y\in B_1(0)\cap\overline{\Gamma}$.  Finally, we note that
$\tilde w$ and $w$ differ by a quadratic function, so
\begin{equation}
  \label{eq:27}
  \begin{split}
    &\tilde w(y)-\tilde w(x)-D\tilde w|_x(y-x)-\tfrac12
    D^2\tilde
    w|_x(y-x,y-x)\\
    &\qquad\qquad\qquad=w(y)-w(x)-Dw|_x(y-x)-\tfrac12D^2w|_x(y-x).
  \end{split}
\end{equation}
Therefore inequality \eqref{eq:nearbyclose} holds for all
$y\in B_1(0)\cap\overline{\Gamma}$ and
$x\in B_{1/2}(0)\cap\Gamma^{(k-1)}$, and the proof of
Proposition~\ref{lem:nobadreglem2} is complete.
\end{proof}

%
%
%
%
%

\noindent\emph{Completion of the Proof of
  Theorem~\ref{prop:nobadreg}.}
Now, Proposition~\ref{lem:nobadreglem1} and
Proposition~\ref{lem:nobadreglem2} allow us to establish estimate
\eqref{eq:wosc1} for all points $x\in B_{1/2}(0)\cap\overline{\Gamma}$
and all points $y\in B_1(0)\cap\overline{\Gamma}$, by
\emph{(decreasing) induction on $k$}: Due to
Proposition~\ref{lem:nobadreglem1}, estimate \eqref{eq:wosc1} holds
for $x\in\Gamma^{(m)}$, and by Proposition~\ref{lem:nobadreglem2} if
estimate~\eqref{eq:wosc1} holds for $x\in\Gamma^{(k)}$ then it also holds for
$x\in\Gamma^{(k-1)}$. Therefore, by induction,
estimate~\eqref{eq:wosc1} holds for all
$x\in B_{1/2}(0)\cap\Gamma^{(0)}=B_{1/2}(0)\cap\overline{\Gamma}$. This
allows us to complete the proof of Theorem~\ref{prop:nobadreg} by
proving that $D^2w$ is continuous at the origin. 
So we must prove that $D^2w|_x$ approaches $D^{2}w|_0$ as
$x\in B_{1/2}(0)\cap\overline{\Gamma}$ approaches zero. To do this, we
apply estimate~\eqref{eq:wosc1} about
$x\in B_{1/2}(0)\cap\overline{\Gamma}$:  Let
\begin{displaymath}
  \tilde{w}(y)=w(y)-w(0)-Dw|_0(y)-\tfrac{1}{2}D^2w|_0(y,y)
\end{displaymath}
for every $y\in B\cap \overline{\Gamma}$. By estimate~\eqref{eq:wosc1},
\begin{displaymath}
    \abs{\tilde w(y)}\le   C\,\abs{y}^{2+\gamma}\quad\text{for every
      $y\in B_{1}(0)\cap \overline{\Gamma}$.}
\end{displaymath}
By~\eqref{eq:27}, estimate~\eqref{eq:wosc1} yields
\begin{displaymath}
  \sup_{y\in \overline{B}_{|x|}(x)}\left|D\tilde{w}|_x(y-x)+\tfrac12
    D^2\tilde{w}|_x(y-x,y-x)\right|
  \le C|x|^{2+\gamma}
\end{displaymath}
for every $x\in B_{1/2}(0)$.  
By Lemma
\ref{lem:closeball} there is a ball of radius comparable to $|x|$ in
$B_{|x|}(x)\cap\overline{\Gamma}$, and applying Lemma
\ref{lem:coeffbound} on this ball gives that
\begin{displaymath}
  \abs{D\tilde{w}|_x(x)+\tfrac{1}{2}D^2\tilde{w}|_x(x,x)}\le C
  |x|^{2+\gamma},\quad
  \norm{D\tilde{w}|_x+D^2\tilde{w}|_x(x,.)}\le C
  |x|^{1+\gamma}, 
\end{displaymath}
and
\begin{equation}
\label{eq:second-derivative-wtilde}
    \norm{D^2\tilde{w}|_x}\le C |x|^{\gamma}
\end{equation}
for every $x\in B_{1/2}(0)\cap \overline{\Gamma}$. Since 
 $D^{2}\tilde{w}|_x=D^2 w|_x-D^{2}w|_0$,
 inequality~\eqref{eq:second-derivative-wtilde} can be rewritten as 
 \begin{displaymath}
   \norm{D^2 w|_x-D^{2}w|_0}\le C |x|^{\gamma}\quad\text{for every $x\in B_{1/2}(0)\cap \overline{\Gamma}$.}
 \end{displaymath}
 proving that harmonic functions on a tame cone $B_{1}\cap\Gamma$
 satisfying homogeneous Neumann boundary condition on
 $B_{1}\cap \partial\Gamma$ are $C^{2,\gamma}$. 
This completes the proof of Theorem~\ref{prop:nobadreg}. 
\end{proof}

%
%
%
%

\section{Polyhedral cones are tame}



\label{sec: homogeneous}

Next, we prove the following, making the tameness hypothesis in
Theorem~\ref{prop:nobadreg} redundant.

\begin{proposition}\label{prop:tame}
  Every polyhedral cone $\Gamma$ in $\R^d$ is tame.
\end{proposition}

\begin{proof}
  The proof uses an induction on the dimension $d\ge 1$, and uses the
  regularity results for tame domains established in the previous
  section.  Our argument here is similar to that used in the
  proof of Proposition~\ref{lemma 1.2}, in that we apply a strong
  maximum principle to the Hessian of the function.  The homogeneity
  of the function allows us to consider points $x_0\in \partial\Gamma$,
  which are not near the vertex of the cone, and this is the basis of
  the induction on dimension: We observe that by Lemma
  \ref{lem:TCnotvertex}, the tangent cone is a direct
  product of a lower-dimensional cone with a line:
  $\Gamma_{x_0} = \tilde\Gamma\oplus\R x_0$, where $\tilde\Gamma$ is a
  polyhedral cone in the subspace $(\R x_0)^\perp$.  To proceed, we need
  to understand the relationship between homogeneous harmonic
  functions on $\Gamma_{x_0}$ and those on $\tilde\Gamma$:

\begin{lemma}\label{lem:prodcone}
  Any homogeneous degree $2$ Neumann harmonic function $u$ on
  $\tilde\Gamma\oplus\R x_0$ has the form
\begin{equation}  \label{eqn of 6.2}
u(x+s x_0) = \tilde u(x) + s\tilde v(x) + C\left(s^2|x_0|^2-\frac{1}{d-1}|x|^2\right)
\end{equation}
for $x\in\tilde\Gamma$, $s\in \R$, where $\tilde u$ is a homogeneous
degree $2$ Neumann harmonic function on $\tilde\Gamma$, $\tilde v$ is
a homogeneous degree $1$ Neumann harmonic function on $\tilde\Gamma$,
and $C$ is constant.
\end{lemma}

\begin{proof}
  Without loss of generality, we may assume that $|x_0|=1$.  We choose
  an orthonormal basis for $\R^d$ so that $x_0=e_{d}$.  Denote
  $A=(\tilde\Gamma\oplus\R e_d)\cap \S^{d-1}$, and
  $\tilde A = \tilde\Gamma\cap \S^{d-2}$.  Then homogeneous degree $2$
  harmonic Neumann functions on $\tilde\Gamma\oplus\R e_d$ are
  determined by their restriction to $A$ which is a Neumann
  eigenfunction.  The corresponding eigenvalue is determined by the
  relation \eqref{eq:beta-i} which produces $\lambda_i=2d$ when
  $\beta_i=2$ (cf~\cite[Chapter~2.4]{MR768584}).

  In the case $d=2$, the cone $\tilde{\Gamma}$ cannot be $\R e_1$,
  since then $\Gamma$ would be $\R^2$, contradicting
  $x_0\in \partial\Gamma$.  Therefore $\tilde{\Gamma}$ is a ray in the
  direction of $\pm e_1$, and the cone $\tilde\Gamma\oplus\R x_0$ is
  congruent to the half-space $H=\{x>0\}$ in $\R^2$.

  Any Neumann harmonic function $u$ on $H$
  extends by even reflection to an entire harmonic function on $\R^2$,
  which is therefore $C^\infty$.  In particular a homogeneous degree $2$
  Neumann harmonic function on $H$ is $C^2$ at the origin and
  therefore agrees with the degree $2$ Taylor polynomial, since the
  second derivatives are homogeneous of degree zero, which must equal
  $C(x^2-y^2)$.  In this case, \eqref{eqn of 6.2} is satisfied with $\tilde{v}\equiv \tilde{u}\equiv 0 $.

  Now, consider the case $d\geq 3$.  We will construct eigenfunctions
  on $A$ from eigenfunctions on $\tilde A$ using separation of
  variables: We parametrise points of $A$ by the map
  \begin{displaymath}
    \Phi : \tilde A\times \left[-\frac{\pi}{2},\frac{\pi}{2}\right]\to
    \S^{d-1}\text{ given by }
    \Phi(z,\theta)=(\cos\theta)\,z+(\sin\theta)\, e_d,\quad
    z\in \tilde A,\; \theta\in \left[-\frac{\pi}{2},\frac{\pi}{2}\right].
  \end{displaymath}
  The construction which follows is quite general
  (producing a basis of eigenfunctions on warped product spaces in
  terms of eigenfunctions on the warping factors), but we describe it
  here only in our specific situation.

\begin{figure}[!htb]
 \minipage{7cm}
\begin{tikzpicture} 

\def\R{2.5} 
\def\angEl{35} 
\filldraw[ball color=white] (0,0) circle (\R);

  \draw (-1.7,2.1) node[above] {$\S^{d-1}$};

  \fill[fill=black] (0,0) circle (1pt);

  \draw[thin, red, outer color=red, inner color=white] (0,-2.5) arc (-90:90:2 and 2.5) arc (90:-90:1 and 2.5);
  \draw[dashed] (2.5,0) arc (0:180:2.5 and 0.6);
  
  \filldraw[dashed] (0,0) --  (2.07,-0.38);
  \filldraw[dashed] (0,0) --  (1.07,-0.58);
   \filldraw[dashed] (0,-2.5) -- (0,2.5); 
   \draw [->, red] (2,2) -- (1.2,1.2); 
   \draw (2,2) node[above, red] {$A$};
   \draw [->, dashed] (0,0) -- (1.4,0.4); 
   \fill[fill=black] (1.4,0.4) circle (1pt);
   \draw (1.5,0.95) node[below] {$(z,\theta)$};
   \draw[->, dashed] (1.46,-0.5) arc (-40:55:0.2 and 0.6); 
   \filldraw[dashed] (0,0) -- (1.46,-0.5);	
   \draw (1.35,-0.2) node[above] {$\theta$};
   
   \draw[thick] (-2.5,0) arc (180:360:2.5 and 0.6);
   \draw (-1,-0.5) node[above] {$\S^{d-2}$};
   \draw[thick, green] (1,-0.55) arc (-70:-40:2.3 and 0.6);
   \fill[fill=black] (1.45,-0.5) circle (1pt);
  \draw (1.6,-0.5) node[below] {$z$};
  \draw (1.2,-0.6) node[below, green] {$\tilde{A}$};		
\end{tikzpicture}
\endminipage
\end{figure}

The metric induced on $\tilde A\times[-\pi/2,\pi/2]$ by the map $\Phi$
is
\begin{displaymath}
  g = \cos^2\theta \bar g + d\theta^2,
\end{displaymath}
where $\bar g$ is the metric on $\S^{d-2}$.  The Laplacian in these coordinates is
\begin{displaymath}
  \Delta^{\S^{d-1}}=\frac{1}{\cos^2\theta}\Delta^{\S^{d-2}}-(d-2)\tan\theta\,\partial_\theta + \partial^2_\theta.
\end{displaymath}
If $\varphi$ is an eigenfunction on $\tilde A$ satisfying
$\Delta^{\S^{d-2}}\varphi+\mu\varphi=0$, then the function
$f(\theta)\varphi(z)$ satisfies the eigenfunction equation on $A$ with
eigenvalue $\lambda$ provided
\begin{equation}\label{eq:Legendre}
{\mathcal L}_\mu f:= f''-(d-2)\tan\theta\,f'-\frac{\mu}{\cos^2\theta}f=-\lambda f.
\end{equation}
Then $f\varphi$ is a Neumann eigenfunction on $A$ provided $\varphi$ satisfies Neumann conditions on $\tilde A$ and $f\varphi$ extends continuously to the poles $\theta=\pm\frac\pi2$ of $A$.  If $\varphi$ is constant on $\tilde A$ (corresponding to $\mu=0$) then this amounts simply to the requirement that $f$ extends continuously to $[-\pi/2,\pi/2]$, but if $\varphi$ is non-constant (corresponding to $\mu>0$) then continuity of $f\varphi$ at the poles amounts to the requirement that $f$ has limit zero at $\pm \frac{\pi}{2}$.  We note that the endpoints $\pm\pi/2$ are regular singular points of the ODE \eqref{eq:Legendre}, and so solutions are asymptotic to $C_1(\theta+\pi/2)^{-\frac{d-3}{2}-\sqrt{\left(\frac{d-3}{2}\right)^2+\mu}}+C_2(\theta+\pi/2)^{-\frac{d-3}{2}+\sqrt{\left(\frac{d-3}{2}\right)^2+\mu}}$ as $\theta\to{-}\pi/2$, and to
$C_3(\pi/2-\theta)^{-\frac{d-3}{2}-\sqrt{\left(\frac{d-3}{2}\right)^2+\mu}}+C_4(\pi/2-\theta)^{-\frac{d-3}{2}+\sqrt{\left(\frac{d-3}{2}\right)^2+\mu}}$ as $\theta\to\pi/2$.  The continuity requirements are therefore that $C_1=0$ and $C_3=0$.

The operator ${\mathcal L}_\mu$ is essentially self-adjoint on $L^2((\cos\theta)^{d-2}d\theta)$.  Accordingly, for any $\mu$ there is an increasing sequence of values $\lambda_{\mu,j}$ approaching infinity such that there is a solution $f_{\mu,j}$ of equation \eqref{eq:Legendre} satisfying the required endpoint conditions.   These form a complete orthonormal basis for $L^2((\cos\theta)^{d-2}d\theta)$.  We claim that if $\{\varphi_i\}_{i=0}^\infty$ is a complete orthonormal basis of Neumann eigenfunctions on $\tilde A$ with eigenvalues $\mu_i$, then the resulting collection of eigenfunctions $\{f_{\mu_i,j}(\theta)\varphi_i(z)\}$ forms a complete orthonormal basis of Neumann eigenfunctions on $A$.  To see this, suppose that $g$ is a function in $L^2(d\omega_{\bar{g}}(\cos\theta)^{d-2}d\theta)$ which is orthogonal to $f_{\mu_i,j}(\theta)\varphi_i(z)$ for all $i$ and $j$.  That is, we have
$$
\int_{-\pi/2}^{\pi/2} \int_{\tilde A} g(z,\theta)\varphi_i(z)\,d\omega_{\bar g}(z)\,f_{\mu_i,j}(\theta) (\cos\theta)^{d-2}\,d\theta = 0
$$
for all $i$ and $j$.  Fix $i$, and let $g_i(\theta) = \int_{\tilde A} g(z,\theta)\varphi_i(z)d\omega_{\bar g}(z)$.  Then $g_i$ is orthogonal to $f_{\mu_i,j}$ for every $j$ in $L^2((\cos\theta)^{d-2}d\theta)$, and so vanishes almost everywhere.  It follows that for almost all $\theta$, $g_i(\theta)=0$ for every $i$.  That is, $g(\theta,z)$ is orthogonal to $\varphi_i(z)$ for every $i$, and hence $g(\theta,z)=0$ for almost all $z$.  This proves that $g=0$ almost everywhere, proving completeness.

It follows that an eigenfunction on $A$ with eigenvalue $\lambda=2d$ is a finite linear combination of 
 terms of the form $f_{\mu_i,j}(\theta)\varphi_i(z)$ for which $\lambda_{\mu_i,j}=2d$.

\begin{lemma}
\label{ode lemma}
For $\lambda=2d$, solutions $f_\mu$ of \eqref{eq:Legendre} with the required boundary conditions  
\begin{displaymath}
  f_\mu \rightarrow C^\pm( \pi/2-|\theta|)^{-\frac{d-3}{2}+\sqrt{\left(\frac{d-3}{2}\right)^2+\mu}} \text{ as }\theta\to{\pm}\pi/2
\end{displaymath}
exist only for $\mu=0$, $\mu=d-2$ and $\mu=2(d-1)$, and these are given by 
$f_0(\theta) = \sin^2\theta-\frac{1}{d-1}\cos^2\theta$, 
$f_{d-2}(\theta) = \sin\theta\cos\theta$, and
$f_{2(d-1)}(\theta) = \cos^2\theta$.
\end{lemma}

\begin{proof}
  The particular solutions given are constructed from homogeneous
  degree two spherical harmonics (harmonic polynomials on $\R^d$):
  These arise from the above construction in the case
  $\tilde A=\S^{d-2}$, and so give rise to solutions of
  \eqref{eq:Legendre}. On $\S^{d-1}$, we have $x_d = \sin\theta$ and
  $|x|=\cos\theta$, where $x=(x_1,\cdots,x_{d-1})$.

  The harmonic function $x_d^2-\frac{1}{d-1}|x|^2$ therefore restricts
  to $f_0(\theta)=\sin^2\theta-\frac{1}{d-1}\cos^2\theta$.  The
  restriction of this to $\S^{d-2}$ is constant, hence an eigenfunction
  with eigenvalue $\mu=0$ on $\S^{d-2}$.  It follows that
  $\mathcal{L}_0f_0+2df_0=0$.

  The harmonic function $x_dx_1$ restricts on $\S^{d-1}$ to the
  function $\sin\theta\cos\theta \frac{x_1}{|x|} =
  f_{d-2}(\theta)\varphi(x/|x|)$ on $\S^{d-1}$, where $\varphi(x)=x_1$
  is a homogeneous degree one harmonic function on $\R^{d-1}$, hence
  an eigenfunction of the Laplacian on $\S^{d-2}$ with eigenvalue
  $\mu=d-2$.  It follows that $\mathcal{L}_{d-2}f_{d-2}+2df_{d-2}=0$.

  Finally, the harmonic function $x_2^2-x_1^2$ on $\R^d$ restricts to
  $ f_{2(d-1)}(\theta)\varphi(x/|x|)$, where
  \begin{displaymath}
    f_{2(d-1)}(\theta) = \cos^2\theta\quad\text{ and }\quad 
    \varphi(x) = x_2^2-x_1^2,
  \end{displaymath}
  which is the restriction to $\S^{d-2}$ of a degree 2 homogeneous
  harmonic function on $\R^{d-1}$, hence an eigenfunction of the
  Laplacian on $\S^{d-2}$ with eigenvalue $\mu=2(d-1)$.  It follows
  that $\mathcal{L}_{2(d-1)}f_{2(d-1)}+2df_{2(d-1)}=0$, as required.

These formulae can 
be checked by explicit computation.

The harder part of the proof is to show that these are the only solutions of \eqref{eq:Legendre} with the required boundary conditions.  It is convenient to perform a transformation of equation \eqref{eq:Legendre} to de-singularise the endpoints at $\pm\pi/2$.  To do this we introduce the new variable $s$ by $\tanh(s/2) = \tan(\theta/2)$, so that $s$ increases over the entire real line as $\theta$ increases from $-\pi/2$ to $\pi/2$. 
 This choice implies that $\frac{d\theta}{ds} = {\cos\theta}$, and we have the identities $\cos\theta=\frac{1}{\cosh s}$, $\sin\theta=\tanh(s)$ and $\tan\theta=\sinh s$.  The equation \eqref{eq:Legendre} transforms to 
$$
0 = f_{ss}-(d-3)\tanh s f_s + \left(\frac{2d}{\cosh^2s}-\mu\right)f.
$$
Defining $f = (\cosh s)^{\frac{d-3}{2}}g$ then produces the equation
\begin{equation}\label{eq:g-potential}
0 = g_{ss}+\left(\frac{(d+1)(d+3)}{4\cosh^2s}-\left(\frac{d-3}{2}\right)^2-\mu\right)g.
\end{equation}
The behaviour at $\theta=\pm\pi/2$ translates to the condition that $g$ is asymptotic to
 $C_2{\rm{e}}^{s\sqrt{\left(\frac{d-3}{2}\right)^2+\mu}}$ as $s\to-\infty$ and to $C_4{\rm{e}}^{-s\sqrt{\left(\frac{d-3}{2}\right)^2+\mu}}$ as $s\to\infty$.

Next, we consider the Riccati equation associated to the ODE \eqref{eq:Legendre}, which is the first order ODE satisfied by the function $q=\frac{g_s}{g}$:
\begin{align*}
\partial_sq &= \frac{g_{ss}}{g}-\left(\frac{g_s}{g}\right)^2\\
&=\mu + \left(\frac{d-3}{2}\right)^2 - \frac{(d+1)(d+3)}{4\cosh^2s}-q^2.
\end{align*}
The boundary conditions then become the requirement that $q\to \sqrt{\left(\frac{d-3}{2}\right)^2+\mu}$ as $s\to-\infty$ and $q\to -\sqrt{\left(\frac{d-3}{2}\right)^2+\mu}$ as $s\to\infty$.  

The function $q$ approaches infinity whenever the value of $g$ crosses zero. 
 We remove these singularities by defining a new variable $\sigma$ which gives (twice) the angle from the positive $x$ axis of the point $(g(s),g_s(s))$, so that $\tan{(\sigma/2)}=g_s(s)/g(s)=q$.  This is defined only modulo $2\pi$, but a continuous choice of $\sigma$ exists and is uniquely defined up to an integer multiple of $2\pi$.    
  It follows from the definition that $\tan(\sigma/2)=q$, and we deduce that
\begin{equation}\label{eq:dsigma}
\sigma_s = (1+\cos\sigma)\left(\mu+1+\left(\frac{d-3}{2}\right)^2-\frac{(d+1)(d+3)}{4\cosh^2s}\right)-2.
\end{equation}
From the asymptotic conditions on $q$, our construction requires a solution $\sigma$ such that
$$
\sigma(s)\to\sigma_-(\mu):=2\arctan\left(\sqrt{\left(\frac{d-3}{2}\right)^2+\mu}\right)
$$
as $s\to-\infty$, and $\sigma(s)\to \sigma_+(\mu)$ modulo $2\pi{\mathbb Z}$ as $s\to\infty$, where
$$
\sigma_+(\mu):=-2\arctan\left(\sqrt{\left(\frac{d-3}{2}\right)^2+\mu}\right).
$$
 For each $\mu$ there is a unique solution  $\sigma_\mu(s)$ of \eqref{eq:dsigma} with $\sigma_\mu(s)\to\sigma_-(\mu)$ as $s\to-\infty$ (arising from the solutions of \eqref{eq:Legendre} with the required asymptotics near $\theta=-\pi/2$ provided by the theory of regular singular points).   It remains to find those values of $\mu$ for which $\sigma_\mu$ has the required behaviour as $s\to\infty$.

 The crucial property we require is monotonicity of $\sigma_\mu(s)$ with respect to $\mu$ for each $s$:  

Suppose $\mu_2>\mu_1\geq 0$.  Then we observe that $\sigma_{\mu_1}(x)$ satisfies
\begin{align*}
\partial_s\sigma_{\mu_1} &=(1+\cos\sigma)\left(\mu_1+1+\left(\frac{d-3}{2}\right)^2-\frac{(d+1)(d+3)}{4\cosh^2s}\right)-2\\
&\leq (1+\cos\sigma)\left(\mu_2+1+\left(\frac{d-3}{2}\right)^2-\frac{(d+1)(d+3)}{4\cosh^2s}\right)-2,
\end{align*}
so that solutions of \eqref{eq:dsigma} for $\mu=\mu_2$ cannot cross $\sigma_{\mu_1}$ from above.  But now for $s$ sufficiently negative we have $\sigma_{\mu_1}(s)$ as close as desired to $\sigma_-(\mu_1)$, while $\sigma_{\mu_2}(s)$ is as close as desired to $\sigma_-(\mu_2)$, and we have $\sigma_-(\mu_1)<\sigma_-(\mu_2)$.  That is, we have $\sigma_{\mu_1}(s)<\sigma_{\mu_2}(s)$ for $s$ sufficiently negative, and the comparison principle implies that this remains true for all $s\in\R$.  This proves that $\sigma_{\mu}(s)$ is strictly increasing in $\mu\geq 0$ for any fixed $s$.  The limit $\overline\sigma_\mu:=\lim_{s\to\infty}\sigma(\mu,s)$ therefore also exists and is (weakly) increasing in $\mu$, although it can (and will) be discontinuous.

Our construction produces a solution $f_\mu$ with the required boundary behaviour precisely when $\overline\sigma_\mu-\sigma_+(\mu)= 2\pi k$ for some $k\in{\mathbb Z}$.   Since $\overline{\sigma}_\mu$ is increasing in $\mu$ and $\sigma_+(\mu)$ is strictly decreasing in $\mu$, we have that $\overline\sigma_\mu-\sigma_+(\mu)$ is strictly increasing in $\mu$, and hence each integer $k$ can arise for at most one value of $\mu$.   We note from \eqref{eq:dsigma} that $\sigma_\mu(s)$ is strictly decreasing at any point where it takes values which are an odd multiple of $\pi$ (corresponding to points where $g(s)=0$), and hence the value of $k$ can be computed as the number of points where the corresponding solution $g$ of \eqref{eq:g-potential} equals zero.

The three solutions constructed above allow us to compute
$\overline\sigma_\mu-\sigma_+(\mu)$ for these three specific values of
$\mu$: For $\mu=0$, the solution
$f_{0}=\sin^2\theta-\frac{1}{d-1}\cos^2\theta$ gives rise to
\begin{displaymath}
  g=(\cosh s)^{-\frac{d-3}{2}}\left(1-\frac{d}{d-1}\frac{1}{\cosh^2s}\right),
\end{displaymath}
which has two crossings of zero, so that we have
$\overline\sigma_0 -\sigma_+(0)=-4\pi$.  For $\mu=d-2$, the solution
$f_{d-2}=\sin\theta\cos\theta$ gives $g=(\cosh s)^{-\frac{d+1}{2}}\sinh s$
which has a single crossing of zero and so, we have
$\overline\sigma_{d-2}-\sigma_+(d-2)=-2\pi$. Finally, for
$\mu=2(d-1)$, the solution $f_{2(d-1)}=\cos^2\theta$ produces
$g=(\cosh(s))^{-\frac{d+1}{2}}$, which has no zero crossings, and hence
$\overline\sigma_{2(d-1)}-\sigma_+(2(d-1)) = 0$.  Since the
$\overline\sigma_\mu-\sigma_+(\mu)$ is strictly increasing, there can
be no other values of $\mu$ between $0$ and $2(d-1)$ for which
$\overline{\sigma}-\sigma_+\in2\pi\mathbb{Z}$.  For $\mu>2(d-1)$ we
have $\overline{\sigma}_\mu-\sigma_+(\mu)>0$, and we observe that the
line $\sigma=\pi$ cannot be crossed by solutions of \eqref{eq:dsigma}
from below, so that we can never have
$\overline{\sigma}_\mu-\sigma_+(\mu)=2\pi k$ for $k$ a positive
integer.  This completes the proof that only the values
$\mu=0,d-2,2(d-1)$ are possible.
\end{proof}

Finally, we complete the proof of Lemma \ref{lem:prodcone}: The
argument above shows that a Neumann eigenfunction on $A$ with
eigenvalue $2d$ has the form
\begin{displaymath}
  f_0(\theta)\varphi_0(z)+f_{d-2}(\theta)\varphi_{d-2}(z)+f_{2(d-1)}\varphi_{2(d-1)}(\theta)v_{2(d-1)}(z)
\end{displaymath}
where $f_0$, $f_{d-2}$ and $f_{2(d-1)}$ are given in Lemma \ref{ode
  lemma}, and $\varphi_0$, $\varphi_{d-2}$ and $\varphi_{2(d-1)}$ are
Neumann eigenfunctions with the corresponding eigenvalues on
$\tilde A\subset S^{d-2}$.  In particular, $\varphi_0$ is a constant,
$\varphi_{d-2}$ is the restriction to $\tilde A$ of a Neumann
homogeneous degree 1 harmonic function $\tilde v$ on
$\tilde\Gamma\subset\R^{d-1}$, and $\varphi_{2(d-1)}$ is the
restriction to $\tilde A$ of a Neumann homogeneous degree 2 harmonic
function $\tilde u$ on $\tilde\Gamma$.

The homogeneous degree 2 Neumann harmonic function $u$ is then
given by extending this eigenfunction on $A$ using the homogeneity:
\allowdisplaybreaks
\begin{align*}
u(x+sx_0) &= |x+sx_0|^2\left(\cos^2\theta \tilde u\left(\frac{x}{|x|}\right)+ \sin\theta\cos\theta\tilde v\left(\frac{x}{|x|}\right)\right.\\
&\quad\null \phantom{|x+sx_0|^2}\left.\phantom{\left(\frac{x}{|x|}\right)}
+ \varphi_0\left(\sin^2\theta-\frac{1}{d-1}\cos^2\theta\right)\right)\\
&=|x+sx_0|^2\left(\frac{|x|^2}{|x+sx_0|^2}\,\frac{1}{|x|^2}\tilde u(x) + \frac{s|x|}{|x+sx_0|^2}\,\frac{1}{|x|}\tilde v(x)\right.\\
&\quad\null \phantom{|x+sx_0|^2}\left.\phantom{\left(\frac{x}{|x|}\right)}
+ \varphi_0\left(\frac{s^2}{|x+sx_0|^2}-\frac{1}{d-1}\frac{|x|^2}{|x+sx_0|^2}\right)\right)\\
&=\tilde u(x) + s\tilde v(x) + \varphi_0\left(s^2-\frac{1}{d-1}|x|^2\right)
\end{align*}
where we used $\sin^2\theta = \frac{s^2}{s^2+|x|^2}$ and $\cos^2\theta = \frac{|x|^2}{s^2+|x|^2}$, the expressions for $f_0$, $f_{d-2}$ and $f_{2(d-1)}$ from Lemma \ref{ode lemma}, and the homogeneity of $\tilde v$ and $\tilde u$.
\end{proof}

\begin{remark}
  The proof above applies with minor modifications to prove that for
  any positive integer $k$, the values of $\mu$ which can give rise to
  an eigenfunction on $A$ with eigenvalue $\lambda = k^2+(d-2)k$
  (corresponding to the restriction of a harmonic function on
  $\tilde\Gamma\times\R$ which is homogeneous of degree $k$) are
  precisely $\mu = j^2+(d-3)j$ for $j=0,\dots,k$ (corresponding to
  eigenfunctions on $\tilde A$ given by the restriction of a harmonic
  function on $\tilde\Gamma$ which is homogeneous of an integer degree
  no greater than $k$).
\end{remark}

\begin{lemma}\label{lem:dimredtame}
  If $\tilde\Gamma$ is a tame cone in a $(d-1)$-dimensional subspace
  $E=(x_0)^\perp$ of $\R^d$, then $\tilde\Gamma\oplus\R x_0$ is a tame
  cone in $\R^d$.
\end{lemma}

\begin{proof}
  Suppose $u$ is a homogeneous degree two Neumann harmonic function on
  $\tilde\Gamma\oplus\R x_0$, with bounded second derivatives.  By
  Lemma \ref{lem:prodcone} we can write
  \begin{displaymath}
    u(x+sx_0) = \tilde u(x) + s\tilde v(x) +
    C\left(s^2|x_0|^2-\frac{1}{d-1}|x|^2\right)\quad\text{for every $x\in\tilde\Gamma$,}
  \end{displaymath}
  where $\tilde u$ is a homogeneous degree 2
  Neumann harmonic function on $\tilde\Gamma$, $\tilde v$ is a
  homogeneous degree 1 Neumann harmonic function on $\tilde\Gamma$,
  and $C$ is constant.  The last term has bounded second derivatives,
  so the sum of the other two terms must also.  Fixing $s=0$ we
  conclude that $\tilde u$ has bounded second derivatives, and hence
  is quadratic function since $\tilde\Gamma$ is tame.  Fixing $s=1$ we
  conclude that $\tilde v$ also has bounded second derivatives.  But
  the second derivatives of a homogeneous degree one function are
  homogeneous of degree $-1$, and hence are unbounded unless they are
  zero.  Therefore $\tilde v$ is a linear function, and we conclude
  that $u$ is a quadratic function.
\end{proof}

Now, we complete the proof of Proposition~\ref{prop:tame} by induction
on dimension.  Suppose that $u$ is a homogeneous degree 2 Neumann
harmonic function on $\Gamma$ with bounded second derivatives.  We
must show that $u$ is a quadratic function.

First, for $d=1$ then every Neumann harmonic function is constant, so
every homogeneous degree 2 Neumann harmonic function vanishes and
hence is a quadratic function.

Now suppose that every polyhedral cone in $\R^{p}$ is tame for $1\leq p<d$,
and let $\Gamma$ be a polyhedral cone in $\R^d$.  We observe that by Lemma \ref{lem:TCnotvertex}, for every $x_0\in \partial\Gamma\setminus\{0\}$ the tangent cone $\Gamma_{x_0}$ is a product of a cone $\tilde\Gamma$ in $(x_0)^\perp$ with $\R x_0$.  By the induction hypothesis, $\tilde\Gamma$ is tame, and hence by Lemma \ref{lem:dimredtame} we conclude that $\Gamma_{x_0}$ is tame.  That is,
$\overline{\Gamma}\setminus\{0\}$ is a tame domain.  It follows from Proposition \ref{prop:nobadreg} that $u$ is $C^2$ on $\overline{\Gamma}\setminus\{0\}$.

Since the second derivatives of $u$ are bounded, there exists a
sequence $(x_{k})_{k\ge 1}$ of points $x_k$ in $\Gamma$ and a sequence
$(e_{k})_{k\ge 1}$ of $e_k\in \S^{d-1}$ such that
\begin{displaymath}
e^{T}_{k}D^2u(x_k)e_k\to
C_2:=\sup_{(x,e)\in\Gamma\times \S^{d-1}}e^{T}D^2u(x)e\quad\text{as $k\to+\infty$}.
\end{displaymath}
The second derivatives of a homogeneous degree 2 function are homogeneous of
degree zero, so we can replace $(x_{k})_{k\ge 1}$ by $(\tilde{x}_{k})_{k\ge
  1}$ given by
$\tilde x_k = \frac{x_k}{|x_k|}\in \S^{d-1}\cap\Gamma$, and conclude
that $e^{T}_{k}D^2u(\tilde x_k)e_k\to C_2$ as $k\to+\infty$.  By compactness,
$(\tilde x_k,e_k)$ converges for a subsequence of $k$ to
$(\bar x,e)\in (\S^{d-1}\cap\overline{\Gamma})\times \S^{d-1}$.  Since
$u$ is $C^2$ at $\bar x$, we have that
$D^2u|_{\bar x}(\bar e,\bar e) = C_2$.

Now we apply Lemma \ref{lem:Hessian-SMP} with $B=\overline{\Gamma}\setminus\{0\}$, and deduce that $\Gamma = \Gamma^E\times\Gamma^\perp$, where $\Gamma^E$ is a polyhedral cone in a subspace $E$ of $\R^d$ of positive dimension $K$, and $\Gamma^\perp$ is a polyhedral cone in $E^\perp$, and we have
$$
u(x) = \Lambda|\pi_E(x)|^2 + g(\pi_{E^{\perp}}(x)).
$$
If $K=\dim E=d$ then since $u$ is harmonic we have $\Lambda=0$ and $u$ vanishes.  Otherwise we write
$$
u(x) = K\Lambda\left(\frac1K|\pi_E(x)|^2 - \frac{1}{d-K}|\pi_{E^{\perp}}(x)|^2\right) + \tilde g(\pi_{E^{\perp}}(x)).
$$
The first term is harmonic, and $u$ is harmonic, so the last term $\tilde g$ is also harmonic.  Furthermore, since $u$ is homogeneous of degree 2, so is $\tilde g$, and $\tilde g$ also satisfies zero Neumann boundary conditions on $\Gamma^\perp$ since $u$ and the first term do.  Finally, $\tilde g$ has bounded second derivatives since $u$ does.  Therefore by the induction hypothesis, $\tilde g$ is a quadratic function, and so $u$ is quadratic and $\Gamma$ is tame.  This completes the induction and the proof of Proposition \ref{prop:tame}.
\end{proof}

\section{Concave implies regular}
\label{sec: concave-implies-regular}

The results of the previous two sections allow us to complete the
proof of the main regularity result, Theorem \ref{semiconcave implies
  C2}. We begin with the following observation.


\begin{lemma}\label{lem:C11}
  Let $\Omega$ be a bounded domain in $\R^{d}$ with a continuous
  boundary $\partial\Omega$. For $\mu\in
  \R$, let $v\in H^{1}_{loc}(\Omega)$ be weak solution of
    $\Delta v+\mu=0$ on $\Omega$. 
  If $v$ is semi-concave on $\Omega$, then $v$ belongs to
  $C^{1,1}(\overline{\Omega})$.
\end{lemma}

\begin{proof}
  Note, that due to classical regularity theory of second order
  elliptic equations (cf~\cite[Corollary~8.11]{MR1814364}),
  $v\in C^{\infty}(\Omega)$. By assumption, there is constant $C\in \R$
  such that $D^2v|_x\le CI$ for every $x\in \Omega$. Given any
  $x\in\Omega$ and any unit vector $e$, choose an orthonormal basis
  $\{e_1,\cdots,e_d\}$ with $e=e_d$.  Then
  \begin{displaymath}
    D^2v|_x(e,e) = \Delta v(x) - \sum_{i=1}^{d-1}D^2v|_x(e_i,e_i) \ge \mu - C(d-1)
  \end{displaymath}
  for every $x\in \Omega$. Thus $D^2v$ is also bounded from below.
  It follows that $Dv$ is Lipschitz with bounded Lipschitz constant,
  and so extends continuously to 
  $\overline{\Omega}$
  as a Lipschitz
  function.
\end{proof}

We are now ready to prove Theorem \ref{semiconcave implies C2}.

\begin{proof}[Proof of Theorem~\ref{semiconcave implies C2}]
  We prove that $v$ is $C^2$ on a neighbourhood of any point
  $x_0\in\partial\Omega$. Choose $r>0$ sufficiently small such that
  \begin{equation}
    \label{eq:28}
    B_r(x_0)\cap\Omega = x_0+r(B_{1}(0)\cap \overline{\Gamma}_{x_0})
  \end{equation}
  and set
  \begin{displaymath}
    w(x) =
    v(x_0+rx)-Dv|_{x_0}(rx)+\frac{\mu}{2d}r^2\abs{x}^2\quad\text{for
      every $x\in B_{1}\cap\overline{\Gamma}_{x_0}$.}
  \end{displaymath}
  Then $w$ is well-defined on $B_1\cap\overline{\Gamma}_{x_0}$, with $\Delta w=0$ on $B_{1}\cap \Gamma_{x_0}$, and
  \begin{displaymath}
    D_{\nu}w|_x = rD_{\nu}v|_{x_0+rx}-rD_{\nu}v|_{x_0}+\frac{\mu}dr^2\, x\cdot\nu=0\quad\text{for $x\in
    B_{1}\cap \partial \Gamma_{x_{0}}$,}
  \end{displaymath}
  since both $x_0$ and $x_0+rx$ are in $\Sigma_i$, so  $D_{\nu_i}v|_{x_0+rx}=D_{\nu_i}v|_{x_0}=-\gamma_i$.   We also use that $x$ is normal to $\nu_i$.   
  This shows that $w$ is a weak solution
  of~\eqref{eq:weak-Neumann-B}. By hypothesis, there is a
  constant $C\in \R$ such that $D^{2}v\le C$ on $\Omega$, and so
  \begin{displaymath}
    D^{2}w|_x(e,e)=r^2D^{2}v|_{x_{0}+rx}(e,e)+ \frac{\mu}{d}r^2
    \abs{e}^{2}\le \left(r^{2}C+ \frac{\mu}{d}r^2\right)\,\abs{e}^{2}
  \end{displaymath}
  for every $e\in \R^{d}$ and $x\in B_{1}\cap \Gamma_{x_0}$, showing
  that $w$ is semi-concave on $B_{1}\cap \Gamma_{x_0}$. Thus, by
  Lemma~\ref{lem:C11}, $w$ is in $C^{1,1}(\overline{B_{1}\cap
  \Gamma_{x_0}})$. By Proposition~\ref{prop:tame}, $B=B_{1}(0)\cap
\overline{\Gamma}_{x_0}$ is tame and hence by
Theorem~\ref{prop:nobadreg}, $w\in C^{2}(B)$. Since $x_0$ is
  arbitrary, $w\in C^{2}(\overline{\Omega})$.
\end{proof}

The results of Theorem \ref{semiconcave implies C2} and
Theorem~\ref{lemma 1.2} imply the following:

\begin{corollary}\label{cor:concave-implies-quadratic}
  Let $\Omega$ be a convex polyhedral domain in $\R^{d}$ with faces
  $\Sigma_{1}, \dots, \Sigma_{m}$, and for given $\mu$,
  $\gamma_{1},\dots, \gamma_{m}\in \R$, let $v$ be a weak solution of
  problem \eqref{eq:3-generalised}. If $v$ is
  semi-concave, then $v$ is a quadratic function.
\end{corollary}

\section{Quadratic solutions and circumsolids}
\label{sec:domain}

In this section we determine precisely which are the domains on which
the solution of \eqref{eq:3} (or, more generally,
\eqref{eq:3-generalised}) is a quadratic function:

\begin{proposition}\label{prop:quad-domains}
  Let $v$ be a quadratic function on $\R^d$, and let $E_1,\cdots,E_k$
  be the eigenspaces of the Hessian of $v$ with eigenvalues
  $\lambda_1,\cdots,\lambda_k$.  Then $v$ satisfies an equation of the
  form \eqref{eq:3-generalised} on a convex polyhedral domain $\Omega$
  if and only if $\Omega = \{x\in\R^d\,\vert\ \pi_{E_i}(x)\in\Omega_i\}$,
  where $\Omega_i$ is a polyhedral domain in $E_i$ for each $i$.
  Furthermore, $v$ satisfies equation \eqref{eq:3} if and only if
  $\lambda_i<0$ and $\Omega_i$ is a circumsolid in $E_i$ with center
  at the maximum of $v|_{E_i}$ and radius equal to $-1/\lambda_i$ for
  each $i$ (see Definition \ref{def:exscribed}).
\end{proposition}
 
 \begin{proof}
   For a quadratic function, the Hessian 
   $ D^2v|_x$ is
   constant.  Accordingly we denote the Hessian by $A$ and let
   $E_1,\dots,E_k$ be the eigenspaces of $A$, so that we have
   $v(x) = \frac12\sum_{i=1}^k\lambda_i|\pi_i(x)|^2 + b\cdot x + c$,
   where $\pi_i$ is the orthogonal projection onto $E_i$, where
   $\lambda_1,\cdots,\lambda_k$ are the eigenvalues of $A$, and
   $b\in\R^d$ and $c\in\R$ are constants.
 
   First we show that $v$ satisfies \eqref{eq:3-generalised} on a
   polyhedral domain $\Omega$ if and only if $\Omega$ is a product of
   polyhedral domains $\Omega_i\subset E_i$: If $\Omega$ has this form
   then
   \begin{displaymath}
     \Omega = \bigcap_{i=1}^k\big\{x\,\vert\ \pi_i(x)\in\Omega_i\big\}
     = \bigcap_{i=1}^k\bigcap_{j=1}^{m_i}\Big\{x\,\Big\vert\
     \pi_i(x)\cdot\nu_{j}^i\leq b_j^i\Big\}
     =\bigcap_{i,j}
    \Big\{x\,\Big\vert\ x\cdot\nu_{j}^i\leq b_j^i\Big\},
\end{displaymath}
where $\Omega_i = \bigcap_{j=1}^{m_i}\{x\in E_i\,\vert\ x\cdot\nu_j^i\leq
b_j^i\}$ for each $i$.  Thus the normals to the faces of
$\Omega$ are $\nu_i^j$ for $1\leq i\leq k$ and $1\leq j\leq
m_i$, corresponding to the face $\Sigma_i^j=\overline{\Omega}\cap\{x\,\vert\
x\cdot\nu_i^j=b_i^j\}$.  The derivative of $v$ is given by
\begin{displaymath}
 Dv|_x(e) = \sum_{p=1}^k\lambda_p\pi_p(x)\cdot e+b\cdot e,
\end{displaymath}
 so on the face $\Sigma_i^j$ we have
 \begin{displaymath}
 D_{\nu_i^j}v|_x = \sum_{p=1}^k\lambda_p\pi_p(x)\cdot\nu_i^j + b\cdot
 e = \lambda_i x\cdot\nu_i^j+b\cdot e = \lambda_ib_i^j+b\cdot e,
\end{displaymath}
which is constant on the face.  Also we have $\Delta v = \sum_{i=1}^k
\dim(E_i)\lambda_i$ which is constant, and so
$v$ is a solution of an equation of the form \eqref{eq:3-generalised}
on $\Omega$.
 
 The converse statement follows from the argument of Lemma
 \ref{lem:Hessian-SMP}:  Equation \eqref{DnuDe} shows that each normal
 vector $\nu_i$ to a face of $\Omega$ is an eigenvector of $A$, and so
 lies in $E_j$ for some $j$.  This allows us to write
\allowdisplaybreaks 
 \begin{align*}
 \Omega &= \bigcap_i\Big\{x\in\R^d\,\Big\vert\ x\cdot\nu_i<b_i\Big\}\\
 &=\bigcap_{j=1}^k\bigcap_{\nu_i\in E_j}\Big\{x\in\R^d\Big\vert\ x\cdot\nu_i<b_i\Big\}\\
 &=\bigcap_{j=1}^k\bigcap_{\nu_i\in E_j}\Big\{x\in\R^d\Big\vert\ \pi_j(x)\cdot\nu_i<b_i\Big\}\\
 &=\bigcap_{j=1}^k\Big\{x\in\R^d\Big\vert\ \pi_j(x)\in\Omega_j\Big\}
 \end{align*}
 where $\Omega_j=\bigcap_{i:\ \nu_i\in E_j}\{x\in E_j\,\vert\ x\cdot\nu_i<b_i\}$.
 
 Now we specialise to the case of equation \eqref{eq:3}: First suppose
 $v$ is strictly concave, so that $\lambda_i<0$ for $i=1,\cdots,k$.
 Then we have
\begin{equation}\label{eq:concavecase}
  v = \frac12\sum_{i=1}^k\lambda_i|\pi_i(x)|^2+b\cdot x+c 
  = \frac12\sum_{i=1}^k\lambda_i\left|\pi_i(x)-\frac{1}{\lambda_i}\pi_i(b)\right|^2+\tilde c
\end{equation}
for some constant $\tilde c$.  Hence $\frac{1}{\lambda_i}\pi_i(b)$ is
the maximum point of $v$ restricted to $E_i$.  The condition that
$\Omega_i$ is a circumsolid in $E_i$ with centre at the maximum of
$v|_{E_i}$ and radius $-1/(2\lambda_i)$ is that
\begin{displaymath}
\Omega_i = \bigcap_{j=1}^{m_i}\left\{x\in E_i\,\Big\vert\
  \left(x-\frac{\pi_i(b)}{\lambda_i}\right)\cdot
  \nu_j^i<-\frac{1}{\lambda_i}\right\}.
\end{displaymath}
In this case, we have for $x$ 
in the face
$\Sigma_j^i = \{x\,\vert\ (x-\frac{\pi_i(b)}{\lambda_i})\cdot
\nu_j^i=-\frac{1}{\lambda_{i}}\}$ that
\begin{displaymath}
D_{\nu_j^i}v(x) = \sum_{p}\lambda_p\left(\pi_p(x)-\frac{\pi_p(b)}{\lambda_p}\right)\cdot \nu_j^i
 = \lambda_i\left(\pi_i(x)-\frac{\pi_i(b)}{\lambda_i}\right)\cdot \nu_j^i = \lambda_i\,\frac{-1}{\lambda_i}=-1
\end{displaymath}
as required.  Conversely, if we suppose that the boundary condition in
\eqref{eq:3} holds, then we can show that $\lambda_i<0$ for every $i$
as follows. We have
 \begin{displaymath}
   Dv|_x(e) = \sum_{p=1}^k\lambda_p\pi_p(x)\cdot e+b\cdot e.
\end{displaymath}
Integrating over $\Omega_i$ and using the divergence theorem gives
 \begin{displaymath}
 -|\partial\Omega_i| = \int_{\partial\Omega_i}\nu_j^i\cdot Dv 
 = \int_{\Omega_i}\Delta^{E_i}v = \dim(E_i)\lambda_i|\Omega_i|,
\end{displaymath}
so that $\lambda_i<0$ and $v$ is strictly concave.  Therefore $v$ has
the form \eqref{eq:concavecase}, and the boundary condition gives
 \begin{displaymath}
   -1=Dv|_x\cdot\nu_j^i = \lambda_i\left(\pi_i(x)-\frac{\pi_i(b)}{\lambda_i}\right)\cdot \nu_j^i
\end{displaymath}
so that $\Omega_i$ is a circumsolid in $E_i$ with radius
$\frac{1}{\lambda_i}$ and centre at $\frac{\pi_i(b)}{\lambda_i}$.
 \end{proof}
 
 \begin{corollary}\label{cor:quadcircum}
   For a convex polyhedral domain $\Omega$, there is a quadratic
   function $v$ solving the elliptic boundary-value
   problem~\eqref{eq:3} if and only if $\Omega$ is a product of
   circumsolids.
 \end{corollary}
 
 \begin{proof}
   Proposition \ref{prop:quad-domains} shows that if $\Omega$ has a
   quadratic solution of \eqref{eq:3} then $\Omega$ is a product of
   circumsolids.  Conversely, suppose $\Omega$ is a product of
   circumsolids. Then there is a decomposition
   $\R^d=E_1\oplus\cdots\oplus E_k$ of $\R^d$ into orthogonal
   subspaces $E_{1},\dots, E_{k}$ and
   \begin{displaymath}
     \Omega = \bigcap_{i=1}^k\Big\{x\in\R^d\,\Big\vert\;
     \pi_i(x)\in\Omega_i\Big\},\quad
     \text{where}\quad\Omega_i :=\bigcap_{j=1}^{m_i} \Big\{x\in
     E_i\;\Big\vert\;(x-p_i)\cdot\nu_j^i<R_i\Big\}
   \end{displaymath}
   for some $p_i\in E_i$ and $R_i>0$. The above calculations show that
   \begin{displaymath}
     v(x) =
     -\frac12\sum_{i=1}^k\frac{|\pi_i(x)-p_i|^2}{R_i}\quad\text{for
       every $x\in \Omega$,}
   \end{displaymath}
   is a solution of \eqref{eq:3} on $\Omega$.
 \end{proof}
 
 Summarising, let $v$ be a weak solution of~\eqref{eq:3} on a convex
 polyhedral domain $\Omega$ in $\R^{d}$. Then by Lemma~\ref{lem:C11},
 if $v$ is semi-concave then $v\in C^{1,1}(\overline{\Omega})$. Using
 that for every boundary point $x_{0}\in \partial\Omega$ and $r>0$
 small enough, $v$ can be written as $v(x_{0}+r\cdot)=w+q$ on
 $B_{1}(0)\cap\overline{\Gamma}_{x_0}$ for a quadratic function $q$
 and a weak solution $w$ of~\eqref{eq:weak-Neumann-B},
 Theorem~\ref{prop:nobadreg} and Proposition~\ref{prop:tame} states
 that $v\in C^{1,1}(\overline{\Omega})$ implies $v$ is in $C^{2}(\overline{\Omega})$
 and according to Theorem~\ref{lemma 1.2}, the latter yields that $v$ is
 quadratic. By Proposition~\ref{prop:quad-domains}, $v$ then needs to be
 concave. Combining this together with Corollary \ref{cor:quadcircum},
 we can state the following characterisation.

 \begin{corollary}
   \label{cor:concavity-characterisation}
   Let $v$ be a weak solution of~\eqref{eq:3}  on a convex polyhedral
   domain $\Omega$ in $\R^{d}$. Then the following statements are equivalent.
   \begin{enumerate}
      \item $v$ is semi-concave;  
      \item $v$ is in $C^{1,1}(\overline{\Omega})$;  
      \item $v$ is in $C^{2}(\overline{\Omega})$;  
      \item $v$ is quadratic;
      \item $v$ is concave;  
      \item $\Omega$ is a product of circumsolids.
   \end{enumerate}
 \end{corollary}

\section{Proof of the main results}
\label{sec:proof-main-results}

In this section, we complete the proofs of our main results:
Theorem~\ref{thm:main1}, Theorem~\ref{thm:main1bis}, 
and Corollary~\ref{cor:approx-omega}. 

\begin{proof}[Proof of Theorem~\ref{thm:main1}]
  Suppose $\Omega$ is polyhedral domain in $\R^{d}$ that is not a
  product of circumsolids.  We first show that for all $\alpha>0$
  small enough, the first Robin eigenfunction $u_{\alpha}$ is not
  log-concave. Set $v_{\alpha}=\log u_{\alpha}$. Then $v_{0}\equiv 0$
  and so, by Proposition~\ref{propo:properties-of-Robineigenvalues},
  $v_{\alpha}$ can be expanded as
  \begin{equation}
    \label{eq:13}
    v_{\alpha}= \alpha v+ f^\alpha,
  \end{equation}
  where $f^\alpha$ belongs to $o(\alpha)$ in
  $C^{0,\beta}(\overline{\Omega})$ for all $\alpha>0$ small enough,
  $\beta\in (0,1)$, and $v$ is a solution of the Neumann
  problem~\eqref{eq:3} for
  $\mu=\frac{d \lambda_{\alpha}}{d\alpha}_{\vert \alpha=0}$. Now, by
  Corollary~\ref{cor:concavity-characterisation}, $v$ is not concave
  on $\overline{\Omega}$.  Thus, there exist $x$,
  $y\in \overline{\Omega}$ and $t\in(0,1)$ such that
   \begin{equation}
    \varepsilon:=t\,v(x)+(1-t)\,
    v(y)-v(tx+(1-t)y)>0.
  \end{equation}
  On the other hand, for every $\delta>0$, there is an $\alpha_{0}>0$
  such that $\norm{f^\alpha}_{\infty}\le \delta\alpha$ for all
  $0<\alpha\le \alpha_{0}$. Set $\delta=\varepsilon/4$, and let
  $\alpha$ be less than the corresponding $\alpha_0$.  Then
  \allowdisplaybreaks
\begin{align*}
tv_\alpha&(x)+(1-t)v_\alpha(y)-v_\alpha(tx+(1-t)y) \\
&=\alpha \left[ t\,v(x)+(1-t)
    v(y)-v(tx+(1-t)y) \right]  \\
    & \hspace{3cm}
     + tf^\alpha(x)+(1-t)f^\alpha(y)-f^\alpha\left(tx+(1-t)y\right) \\
&\ge \alpha\varepsilon -3\delta \alpha>0,
\end{align*}
so $v_\alpha$ is not concave for any $\alpha<\alpha_0$, proving
Theorem~\ref{thm:main1}.
\end{proof}

Next we consider the convexity of superlevel sets
$\{x\big|\ u_\alpha(x)> c\}$. We first establish two preliminary
results. The first is a Lichn\'erowicz-Obata type result for the first
non-trivial Neumann eigenvalue on a convex subset of the sphere
$\S^{d-1}$, which extends partially the result of~\cite{MR1072395} by
allowing non-smooth boundary, resulting in a larger class of quality
cases.


\begin{theorem}\label{thm:lichobata}
  For $d\ge 3$, let $A$ be a  
  convex open subset of the sphere $\S^{d-1}$. Then the first nontrivial
  eigenvalue
  \begin{displaymath}
    \lambda_{1}(A)=\inf_{\varphi\in C^{\infty}(\overline{A}) :
      \int_{A}\varphi dV_{g}=0}\frac{\int_{A}\abs{D\varphi}^{2}\,
      \textrm{d}V_{g}}{\int_{A}\abs{\varphi}^{2} dV_{g}}
  \end{displaymath}
  of the Neumann Laplacian on $A$ satisfies $\lambda_{1}\ge
  d-1$. Moreover, $\lambda_{1}(A)= d-1$ if and only if 
  the cone $\Gamma=\{x=rz\in \R^{d}\,\vert\,z\in A\}$ in $\R^{d}$ has
  a linear factor, so that (after an orthogonal transformation)
  $\Gamma=\tilde{\Gamma}\times \R$ for some convex cone
  $\tilde{\Gamma}$ in $\R^{d-1}$. In this case, the corresponding
  eigenfunction is the restriction to $\S^{d-1}$ of the linear function
  $L(x,y)=y$ for $(x,y)\in\R^{d-1}\times\R$.
 \end{theorem}
 
 \begin{proof}
 First suppose that $A$ has smooth boundary.  Then for any $u\in H^1(A)$ and $f\in H^3(A)$ with $D_\nu f=0$ on $\partial A$, the following Reilly-type formula holds:
 \begin{equation}\label{eq:Reilly}
 \begin{split}
 \int_A\left(\Delta f-(d-1)u\right)^2 &- \int_A\|\nabla^2f -ug\|^2 - (d-2)\int_A\|\nabla f+\nabla u\|^2 - \int_{\partial A} h(\bar\nabla f,\bar\nabla f)\\
 &=(d-2)\left[(d-1)\int_A u^2 - \int_A\|\nabla u\|^2\right] \end{split}
 \end{equation}
 where $\nabla$ is the covariant derivative on $S^{d-1}$, $h$ is the second fundamental form of $\partial A$, and $\bar\nabla f$ is the gradient vector of the restriction of $f$ to $\partial A$.  This is proved by integration by parts and application of the curvature identity (the proof due to the first author for the situation without boundary is described in \cite[Theorem B.18]{CNL}).
 
 In particular, given $u\in H^1(A)$ with $\int_Au=0$, let $f$ be a solution of the problem
 \begin{equation}\label{eq:NP}
\begin{cases} \Delta f = (d-1)u&\text{on\ }A;\\
 D_\nu f=0&\text{on\ }\partial A.\end{cases}
 \end{equation}
 With this choice the first term on the left vanishes, and the remaining terms are non-positive, so the right-hand side is non-positive, proving the Poincar\'e inequality
 $$
 \int_A \|\nabla u\|^2 \geq (d-1)\int_A u^2
 $$
 for all $u\in H^1(A)$ with $\int_Au=0$, implying that $\lambda_1(A)\geq d-1$.
 
 Now consider the general case, where the boundary of $A$ may not be smooth.  Suppose that $\{A_n\}$ is a sequence of convex domains in $S^n$ with smooth boundary, which converge in Hausdorff distance to $A$ (these can be constructed by smoothing level sets of the distance to $\partial A$, for example).  Let $\{u_n\}$ be the corresponding sequence of first eigenfunctions, normalised to $\int_{A_n}u_n^2=1$.  The solution of \eqref{eq:NP} is then given by 
 $f_n = -\frac{d-1}{\lambda_1(A_n)}u_n$.  As $n\to\infty$ we have $\lambda_1(A_n)\to \lambda_1(A)$, so $\lambda_1(A)\geq d-1$.  
 
 Suppose that equality holds.  Then we can find a subsequence along which $u_n$ converges weakly in $H^1$ to the first eigenfunction $u$ on $A$, and the interior regularity estimates imply that $u_n$ converges to $u$ in $C^\infty(B)$ for any compact subset $B$ of $A$.  The right-hand side of \eqref{eq:Reilly} is equal to $(d-1)-\lambda_1(A_n)$, which converges to zero.  The first term on the left is equal to zero for every $n$, and the last term on the left is non-positive by the convexity of $A_n$, so on any compact subset $B$ we have
\begin{align*}
\int_{B}\|\nabla^2f_n-(d-1)u_n\|^2 &+ (d-2)\int_{B}\|\nabla f_n+\nabla u_n\|^2\\
&\leq \int_{A_n}\|\nabla^2f_n-(d-1)u_n\|^2 + (d-2)\int_{A_n}\|\nabla f_n+\nabla u_n\|^2\\
&\leq (d-2)(\lambda_1(A_n)-(d-1)).
\end{align*}
 Since $f_n$ converges to $-u$ on $B$, the left-hand side converges to $\int_{B}\|\nabla^2u+ug\|^2$, while the right-hand side converges to zero.  Therefore we have $\nabla^2u+ug=0$ at every point of $B$, and hence at every point of $A$ since $B$ is an arbitrary compact subset of $A$.
 
 It follows that $u$ is the restriction of a linear function on $\R^d$ to $\S^{d-1}$:  Define $e(z):= u(z)z+\nabla_i u(z)g^{ij}\partial_jz\in\R^d$.  Then we have
 $$
 \partial_k e = \partial_ku z + u\partial_kz + \nabla_k\nabla_iu g^{ij}\partial_jz +\nabla_iu g^{ij}(-g_{ki}z) = (\nabla^2u+ug)_{ki}g^{ij}\partial_jz = 0,
 $$
 so that $e$ is constant on $A$.  Finally, we have $u(z) = e(z)\cdot z$, which is a linear function.   The claimed structure of $A$ now follows from the Neumann condition $D_\nu u=0$.
 \end{proof}

This result has an immediate consequence, which is important in our
proof of Theorem~\ref{thm:main1bis}.

\begin{lemma}
  \label{lem:existence-of-homog-1-function}
  Let $\Gamma$ be a polyhedral convex cone in $\R^{d}$ with vertex at the origin.
 Then there is a harmonic function $\hat{w}$ on
  $\Gamma$ which is homogenous of degree one and satisfies
  $D_{\nu}\hat{w}=-1$ on $\partial\Gamma$.
\end{lemma}

\begin{proof}
  Set $A:=\Gamma\cap \S^{d-1}$. First consider the case when $\Gamma$
  does not admit a linear factor. Then by Theorem~\ref{thm:lichobata},
  $d-1$ is in the resolvent set $\rho(-\Delta^{\S^{d-1}}_{\vert A})$
  of the operator $-\Delta^{\S^{d-1}}_{\vert A}$ equipped with
  homogeneous Neumann boundary conditions and realised in $L^{2}(A)$.
  Therefore, there exists a unique weak solution $\tilde{\varphi}$ of
\begin{equation}
  \label{eq:eigenvalue-on-Sd-1}
  \begin{cases}
    \Delta^{\S^{d-1}}\tilde{\varphi}+(d-1)\tilde{\varphi}&=\phantom{-}0 \text{ on $A$,}\\
   \hfill D_{\nu}\tilde{\varphi}&=-1  \text{ on $\partial A$.}
  \end{cases}
\end{equation}

It follows that the function
  \begin{equation}\label{eq:hatw}
    \hat{w}(rz):=r\,\tilde{\varphi}(z)\quad\text{for $r\in [0,1]$, $z\in A$}
  \end{equation}
  is harmonic on $\Gamma$, homogeneous of degree one, and satisfies
  $D_{\nu}\hat{w}=-1$ on $\partial\Gamma$.
  
  Now suppose $\Gamma$ has a linear factor, so that
  there is a $k\in \{1,\dots,d-1\}$ such that
  $A=(\R^{k}\oplus \tilde{\Gamma})\cap \S^{d-1}$ for a polyhedral cone
  $\tilde{\Gamma}$ in $\R^{d-k}$ with no linear factors. In particular, if $k=d-1$ then
  $\tilde{\Gamma}=(0,+\infty)$, and then $\tilde{\varphi}(z)=z$ is a
  solution of~\eqref{eq:eigenvalue-on-Sd-1} on 
  $\tilde{\Gamma}$ and so,
  \begin{displaymath}
    \hat{w}(x_{1},\dots,x_{d-1},z):=z\quad\text{for
      every $(x_{1},\dots,x_{d-1},z)\in \Gamma$}
  \end{displaymath}
  is harmonic on $\Gamma$, homogeneous of degree one, and satisfies $D_{\nu}\hat{w}=-1$ on
  $\partial\Gamma$.  Otherwise we have $1\leq k\leq d-2$ and
  $\tilde{\Gamma}$ is a convex polyhedral cone with has no linear factor. Then by the first case, there is a
  harmonic function $\tilde{w}$ on $\tilde{\Gamma}$ which is homogeneous of degree one and satisfies
  $D_{\nu}\tilde{w}=-1$ on $\partial\tilde{\Gamma}$.  Then the function
  \begin{displaymath}
    \hat{w}(x_{1},\dots,x_{k},z):=\tilde{w}(z)\quad\text{for
      every $(x_{1},\dots,x_{k},z)\in \R^{d-k}\times\tilde\Gamma=\Gamma$}
  \end{displaymath}
  is a harmonic on $\Gamma$, homogeneous of degree one, and satisfies
  $D_{\nu}\hat{w}=-1$ on $\partial\Gamma$.  We note that the solution
  space is in general of dimension $k$, since we can add an arbitrary
  linear function on the linear factors.
\end{proof}

Further, in dimension $d\ge 3$, we will use the following 
characterisation of polyhedra with boundary points with inconsistent normals. 
We
omit the proof of this result.

\begin{proposition}
  \label{propo:irregular-bdry}
  Let $\Omega$ be a convex polyhedral domain in $\R^{d}$, $d\ge 2$,
  with outer unit face normals $\nu_{1},\dots,\nu_{m}$. For each point
  $x\in \overline{\Omega}$, let $\mathcal{I}(x)$ be the index
  set~\eqref{eq:9} of faces touching $x$. Then the
  following statements are equivalent.
  \begin{enumerate}
     \item $x$ has inconsistent normals:  That is, there is no $\gamma\in \R^{d}$ satisfying
       \begin{displaymath}
         \nu_{i}\cdot \gamma =-1\quad\text{for every $i\in \mathcal{I}(x)$.}
       \end{displaymath}  
     \item If $\hat w$ is a function on $\Gamma_{x}$ which is harmonic and homogeneous of degree one and satisfies
       \begin{displaymath}
         D_{\nu_{i}}\hat{w}=-1\quad\text{ on $\Sigma_{i}$ for all
           $i\in \mathcal{I}(x)$.}
     \end{displaymath}
     then $\hat w$ is not a linear function.
     \item The
       tangent cone $\Gamma_{x}$ to $\Omega$ is not a circumsolid.
  \end{enumerate}
\end{proposition}

We now proceed to the proof of Theorem \ref{thm:main1bis}:

\begin{proof}[Proof of Theorem~\ref{thm:main1bis}]  
  Under the assumptions of Theorem \ref{thm:main1bis}, we first prove
  that the function $v$ satisfying the Neumann problem ~\eqref{eq:3}
  has some non-convex superlevel sets.
    
  Let $x_{0}\in \partial\Omega$ and $\Gamma=\Gamma_{x_0}$, and choose
  $r>0$ small enough so that~\eqref{eq:28} holds. We define
  \begin{displaymath}
        \tilde{v}(x) := v(x_0+x)+\frac{\mu}{2d}\abs{x}^2\quad\text{for
     $x\in B_{r}(0)\cap \Gamma$.}
  \end{displaymath}
  Then $\tilde{v}$ is harmonic on $B_{r}\cap \Gamma$ and
  satisfies $D_{\nu}\tilde{v}=-1$ on $B_{r}\cap\partial
  \Gamma_{x_0}$. By Lemma~\ref{lem:existence-of-homog-1-function},
  there is a harmonic function $\hat{w}$ on
  $\Gamma$ which is homogenous of degree one 
  and satisfies
  $D_{\nu}\hat{w}=-1$ on $\partial\Gamma$. Then the
  function
  \begin{displaymath}
    w(x):=\tilde{v}(x)-\hat{w}(x)\quad\text{for
      every $x\in B_{r}(0)\cap \Gamma$}
  \end{displaymath}
  is a weak solution of the Neumann
  problem~\eqref{eq:weak-harmonic-Neumann} on
  $B_r(0)\cap\Gamma$. 
  Proposition~\ref{lem:series-expansion} applied to a suitable dilation of $w$ gives the series
  expansion~\eqref{eq:hhdecomp}. Therefore, $v$ can be written as
  \begin{equation}
    \label{eq:29}
    v(x_{0}+x)=-\frac{\mu}{2d}\abs{x}^2+ \hat{w}(x) +
 \sum_{i=0}^\infty f_i\, \psi_{i}(x)\quad\text{for
      every $x\in B_{r}\cap\Gamma$,}
  \end{equation} 
  where $\psi_{i}$ is the harmonic function on $\Gamma$ given by
  \begin{equation} \label{eqn-for-psi}
    \psi_{i}(x):=s^{\beta_i}\varphi_i(z) \quad\text{for
      every $x=sz$ with $s>0$ and $z\in A$.}
  \end{equation}
  Here $A=\Gamma\cap \S^{d-1}$, $\{\varphi_{i}\}_{i=0}^{\infty}$ is an
  orthonormal basis of $L^{2}(A)$ consisting of eigenfunctions $\varphi_{i}$ of
  the  Neumann-Laplacian $\Delta^{\S^{d-1}}$ on $A$, and
  $\beta_{i}$ are given by~\eqref{eq:beta-i}.   Further, $\beta_0=0$ (corresponding to $\lambda_0=0$), and the remaining
    $\beta_i$ are estimated by Theorem~\ref{thm:lichobata}
    and~\eqref{eq:beta-i}, so that $\beta_i\ge 1$ for $i\ge1$.
  Moreover, there is no loss of generality in assuming that each
  $\beta_{i}>1$, since
  $\tilde{w}(x):=\sum_{i : \beta_{i}=1}f_{i}\psi_{i}(x)$ is harmonic
  on $\Gamma$, of homogeneous degree one, and satisfies
  $D_{\nu}\tilde{w}=0$ on $\partial\Gamma$, and so $\tilde{w}$ can be
  included in $\hat{w}$.   Summarising, we can write
  \begin{equation}
    \label{eq:29bis}
    v(x_{0}+x)=v(x_0)-\frac{\mu}{2d}\abs{x}^2 +\hat w(x)+
 \sum_{i=1}^\infty f_i\, \psi_{i}(x)\quad\text{for
      every $x\in B_{r}\cap\Gamma$,}
  \end{equation}
  where the non-vanishing terms in the sum all have exponent $\beta_i>1$.
\bigskip 

  Before continuing the proof of Theorem~\ref{thm:main1bis}, we observe that the proof of
  Theorem \ref{prop:nobadreg} applies almost without change to prove the following 
  generalisation:

  \begin{theorem}\label{thm:nobadreg-bis}
    Let $\Omega$ be a polyhedral domain in $\R^{d}$, and $B$ a relatively open subset of $\Omega$.   Let $w\in H^1(B)$ be a weak solution of problem \eqref{eq:weak-Neumann-B}.  For each
    $x_{0}\in B\cap\partial\Omega$, choose
    $r(x_0)>0$ small enough so that~\eqref{eq:28} holds, so that  by Proposition \ref{lem:series-expansion} $w$ is given by the expansion
    \begin{displaymath}
      w(x_0+x)=
      \sum_{i=0}^\infty f_i(x_0)\, \psi_{i}^{x_0}(x) \quad\text{for
      every $x\in B_{r(x_0)}(0)\cap\Gamma_{x_0}$,}
    \end{displaymath}
    where $\psi^{x_0}_{i}$ is the harmonic function on $\Gamma_{x_0}$ given
    by $\psi^{x_0}(x) = |x|^{\beta_i(x_0)}\varphi_i^{x_0}\left(\frac{x}{|x|}\right)$, $A_{x_0}=\Gamma_{x_0}\cap \S^{d-1}$, $\{\varphi_{i}^{x_0}\}_{i=0}^{\infty}$ is an
    orthonormal basis of $L^{2}(A_{x_0})$ consisting of eigenfunctions $\varphi_{i}^{x_0}$ of
    the  Neumann-Laplacian $\Delta^{\S^{d-1}}$ on $A_{x_0}$, and
    $\beta_{i}(x_0)$ are given by~\eqref{eq:beta-i}. If $f_i(x_0)\neq 0$ only for those $i$ with $\beta_{i}(x_0)\geq 2$ for every $x_0\in B\cap\partial\Omega$,
    then $w\in C^{2}(B)$.
  \end{theorem}

  \noindent\emph{Continuation of the Proof of
    Theorem~\ref{thm:main1bis}.}  
   
  \textbf{The case of inconsistent normals:} In the case where
  $\Omega$ has a boundary point $x_0$ where the normal vectors are
  inconsistent, we have by Proposition \ref{propo:irregular-bdry} that
  $\hat w$ is not a linear function.  It follows that $\hat w$ does
  not have convex superlevel sets: Choosing any point $z\in A$ where
  $\hat w(z)\neq 0$ and $D^2\hat w|_z\neq 0$, we have that $z$ is a
  null eigenvector of $D^2\hat w$ (since $\hat w$ is homogeneous of
  degree one), and that the trace of $D^2\hat w|_z$ on the orthogonal
  subspace $(\R z)^{\perp}$ is zero (since $\hat w$ is harmonic).  It follows that
  $D^2\hat w|_z$ has an eigenvector $\xi\in (\R z)^{\perp}$ with
  positive eigenvalue, so that $D^2\hat w|_z(\xi,\xi)>0$.
  
  Now let $\eta = \xi - \frac{D\hat w|_z(\xi)}{\hat w(z)}z$.  Then we have
  \begin{displaymath}
    D\hat w|_z(\eta) = D\hat w|_z(\xi) - \frac{D\hat w|_z(\xi)}{\hat w(\xi)}D\hat w|_z(z) =0,
  \end{displaymath}
  since $D\hat w|_z(z)=\hat w(z)$ by the homogeneity of $\hat w$.  Also, we have
  \begin{displaymath}
    D^2\hat w|_z(\eta,\eta) = D^2\hat w|_z(\xi,\xi)>0,
  \end{displaymath}
  since $z$ is a null eigenvector of $D^2\hat w|_z$.  It follows that
  the superlevel set $S=\{x\,\vert\ \hat w(x)>\hat w(z)\}$ is not convex
  near $z$, since for small $s\neq 0$ we have
  $\hat w(z\pm s\eta)>\hat w(z)$ and hence $z\pm s\eta\in S$, but
  $z\notin S$.  Since $\hat w$ is homogeneous, the superlevel sets
  $S_\lambda = \{x\,\vert\ \hat w(x)>\lambda\hat w(z)\}$ are also non-convex
  near $\lambda z$, for any $\lambda>0$.
  
  Now we conclude that $v$ also has some non-convex superlevel sets:
  By the non-convexity and openness of $S$, there exist points $x_1$
  and $x_2$ in $\Gamma$ such that $x_1,x_2\in S$ but
  $\frac{x_1+x_2}{2}\notin \overline{S}$.  It follows that there
  exists $\varepsilon>0$ such that $\hat w(x_i)>\hat w(z)+\varepsilon$
  for $i=1,2$, but
  $\hat w\left(\frac{x_1+x_2}{2}\right)<\hat w(z)-\varepsilon$.  Now
  we use the expression \eqref{eq:29} to write
 \begin{align*}
 v(x_0+\lambda x_j) &= v(x_0)+\lambda \hat w(x_j)-\frac{\mu}{2d}
        \lambda^2|x_j|^2+\sum_{i>1}f_i\lambda^{\beta^i}\psi_i(x_j)\\
 &=v(x_0)+\lambda\left(\hat w(x_j) - \frac{\mu\lambda}{2d}|x_j|^2
   +\sum_{i>1}\lambda^{\beta_i-1}\psi_i(x_j)\right)\\
 &>v(x_0)+\lambda\left(\hat w(x_j)-\varepsilon\right)\\
 &>v(x_0)+\lambda\hat w(z)
 \end{align*}
 for $j=1,2$, for $\lambda>0$ sufficiently small.  Here we used the
 fact that the sum $\sum_{i>1}\lambda^{\beta_i-1}\psi_i(x_j)$
 converges to zero as $\lambda$ approaches zero, which follows as in
 the proof of Lemma \ref{lem:twicediff}.  Similarly, we have
 \begin{displaymath}
   v\left(x_0+\lambda\frac{x_1+x_2}{2}\right)<v(x_0)+\lambda \hat w(z)
\end{displaymath}
for $\lambda>0$ sufficiently small. This proves that the superlevel set
$\{x\,\vert\, v(x)>v(x_0)+\lambda\hat w(z)\}$ is not convex.

\textbf{The case of consistent normals:} Now we consider the case
where the normals are consistent at every point.  By Proposition
\ref{propo:irregular-bdry} this implies that for every $x_0$, the
function $\hat w$ on $\Gamma_{x_0}$ provided by Lemma
\ref{lem:existence-of-homog-1-function} is linear.  If for every $x_0$
the non-zero terms in the expansion of the Neumann harmonic function
$w(x) =\sum_{i=1}^\infty f_i\psi_i(x)$ had exponent $\beta_i\geq 2$
for every $x_0$, then by Proposition \ref{thm:nobadreg-bis} $w$ is
$C^2$ near $x_0$ and hence \eqref{eq:29bis} implies that $v$ is also
$C^2$ near $x_0$.

However, if we assume that $\Omega$ is not a product of circumsolids,
then by Corollary \ref{cor:concavity-characterisation}, we have that
$v$ is not $C^{2}$ and so there must be some $x_0\in\partial\Omega$
such that the first nontrivial term in the sum in \eqref{eq:29bis} has
exponent $\beta_i$ between $1$ and $2$: Precisely, we can assume (by
choosing a new basis for the corresponding eigenspace if necessary)
that
  \begin{equation}
    \label{eq:29modified}
    v(x_{0}+x)=v(x_0)-\frac{\mu}{2d}\abs{x}^2 + f_1\psi_1(x)+
    \sum_{i>1}^\infty f_i\, \psi_{i}(x)+ \hat{w}(x)\quad\text{for
      every $x\in B_{r}\cap\Gamma$,}
  \end{equation}
  where $f_1>0$, $1<\beta_1<2$, and $\beta_i>\beta_1$ for $i>1$.
  Since $\psi_{i}$ is homogeneous of order $\beta_{i}$ and
  $\hat{w}(x)=x\cdot\gamma$, we have
  \begin{align}
    \label{eq:32bis}
     Dv|_{x_{0}+\lambda x}(\xi) &=\gamma\cdot\xi-\frac{\mu\lambda}{d}\, x\cdot \xi + 
      \sum_{i\geq 1}^\infty f_i\,  \lambda^{\beta_i-1}
                        D\psi_{i}\big|_{x}(\xi)\\
    \label{eq:32}
    D^{2}v|_{x_{0}+\lambda x}(\xi,\eta) &= -\frac{\mu}{d}\xi\cdot\eta 
                                          + f_1\lambda^{\beta_1-2}D^2\psi_1|_x(\xi,\eta)+
    \sum_{i>1}^\infty f_i\,
    \lambda^{\beta_i-2}D^{2}\psi_i\big|_{x}(\xi,\eta)
      \end{align}
  for every $x\in B_{r}(0)\cap\Gamma$, $\lambda\in(0,1)$, and
  $\xi,\eta\in\R^d$. 

  To show that $v$ has a non-convex superlevel set, it suffices to
  show that there exists $x$ with $x_0+x\in\Omega$, and $\xi\in \R^d$,
  such that
  \begin{equation}
    \label{eq:30}
    Dv|_{x_0+x}(\xi)=0\quad\text{and}\quad D^{2}v|_{x_0+x}(\xi,\xi)>0.  
  \end{equation}
  We note that as $\lambda$ approaches zero, the right-hand side of
  \eqref{eq:32bis} is dominated by the first term since the remaining
  terms are homogeneous of positive degree in $\lambda$, while the
  right-hand side of \eqref{eq:32} is dominated by the first
  non-trivial term in the sum since this is homogeneous of degree
  $\beta_1-2<0$ in $\lambda$.\bigskip

This motivates the following lemma:

\begin{lemma}\label{lem:section-concave}
  Suppose that the restriction of $\psi_1$ to the hyperplanar section
  $L:= \{x\in\Gamma\,\vert\ \gamma\cdot x=|\gamma|\}$ of
  $\Gamma=\Gamma_{x_0}$ is not concave.  Then $v$ has a non-convex
  superlevel set.
\end{lemma}

\begin{proof}
  Since the restriction of $\psi_1$ to $L$ is not concave, there
  exists $x\in L$ and $\xi_0\perp\gamma$ such that
  $D^2\psi_1|_x(\xi_0,\xi_0)>0$.  The expression \eqref{eq:32bis} then
implies
  \begin{displaymath}
    Dv|_{x_0+\lambda x}(\xi_{0}) =
    O(\lambda^{\beta_1-1}),\quad\text{and}\quad 
    Dv|_{x_0+\lambda x}(x) = |\gamma|+O(\lambda^{\beta_1-1}),
\end{displaymath}
for $\lambda\to 0+$, from which it follows that
$Dv|_{x_0+\lambda x}(\xi_{0}+c(\lambda)x)=0$ for some
$c(\lambda)=O(\lambda^{\beta_1-1})$ as $\lambda\to0+$. Then we have by
\eqref{eq:32} that
\begin{displaymath}
  D^2v|_{x_0+\lambda x}(\xi_{0}+c(\lambda)x,\xi_{0}+c(\lambda)x) 
  = \lambda^{\beta_1-2}\left(f_1D^2\psi_1|_x(\xi_{0},\xi_{0})+O(\lambda^\sigma)\right),
\end{displaymath}
where $\sigma = \min\{\beta_1-1,2-\beta_2,\beta_2-\beta_1\}$.  In
particular, since $D^2\psi_1|_x(\xi_{0},\xi_{0})>0$, we have that
$D^2v|_{x_0+\lambda x}(\xi_{0}+c(\lambda)x,\xi_{0}+c(\lambda)x)>0$ for
$\lambda>0$ sufficiently small, proving that $v$ has a non-convex
superlevel set.
\end{proof}

\begin{remark}\label{rem:not-able-to-proof}
  We are unable to establish the hypothesis of Lemma
  \ref{lem:section-concave} for dimensions $d\ge 3$, but note here that this
  would be sufficient to prove that $v$ has a non-convex superlevel
  set whenever $\Omega$ is not a product of circumsolids,
  substantially strengthening the result of Theorem~\ref{thm:main1bis}.
\end{remark}

\textbf{The case $d=2$:} We can establish the hypothesis of Lemma
\ref{lem:section-concave} in the case $d=2$, as follows: In this case
the tangent cone $\Gamma_{x_0}$ at any boundary point
$x_0\in\partial\Omega$ is a sector with opening angle
$\theta_0\le\pi$.  The case $\theta_0=\pi$ cannot arise, since in that
case the homogeneous Neumann harmonic functions on the half-plane
$\Gamma_{x_0}$ are spherical harmonics with integer degree of
homogeneity, so one cannot have $\beta_1\in(1,2)$. Therefore
$\theta_0\in(0,\pi)$. 

Let $\gamma$ be the inward-pointing bisector of this sector of length
$1/\sin\left(\theta_0/2\right)$.  Then we have
$\nu_i\cdot\gamma=-1$ for $i=1,2$, where $\nu_1$ and $\nu_2$ are the
outer unit normal vectors to the two faces of $\Omega$ which meet at
$x_0$.  The homogenous degree one harmonic function of
Lemma~\ref{lem:existence-of-homog-1-function} is then given by
$\hat{w}(x)=\gamma\cdot x$.  In particular $\hat w$ is linear, so we
are in the situation where all boundary points have consistent
normals.
  
The corresponding eigenfunctions are given by 
 
\begin{displaymath}
 \psi_i(r\,(\cos\theta)\, e_1+r\,(\sin\theta)\, e_2) =  
    \begin{cases}
      r^{\frac{i\pi}{\theta_0}}\cos\left(\frac{i\pi}{\theta_0}\theta\right),&i\ \textrm{even};\\
      r^{\frac{i\pi}{\theta_0}}\sin\left(\frac{i\pi}{\theta_0}\theta\right),&i\ \textrm{odd},
    \end{cases}
    \quad\text{for $\theta\in \left(-\tfrac{\theta_{0}}{2}, \tfrac{\theta_{0}}{2}\right)$,}
\end{displaymath}
with degree of homogeneity $\beta_i = \frac{i\pi}{\theta_0}$, for
non-negative integer $i$. Here, $e_{1}=\frac{\gamma}{|\gamma|}$, and
$e_2$ is a unit vector orthogonal to $\gamma$.
 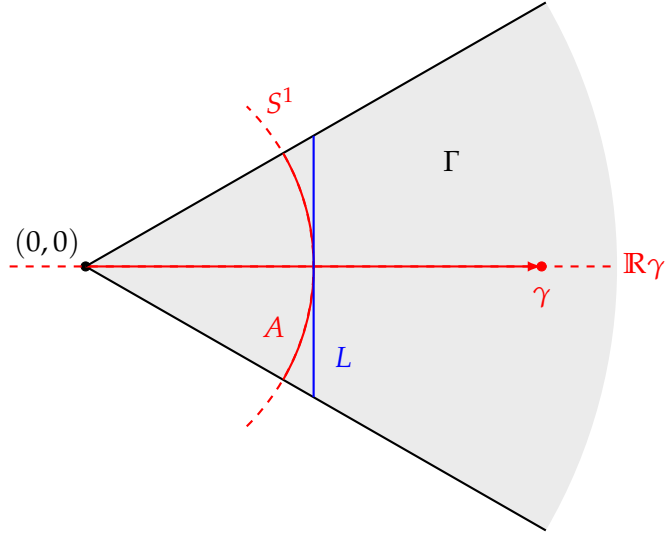
\begin{figure}[!htb]\label{fig:d=2}
     \minipage{9cm}
     \mbox{}\hspace{1cm}
     \begin{tikzpicture}
        \fill [grey!20!white] (0,0) -- (6.055,-3.5) arc (-30:30:7cm) -- (0,0);
       \node [left,black] at (5,1.4) {$\Gamma$};
       \node [left,black] at (0,0.3) {$(0,0)$};
       \draw [black,fill] (0,0) circle (0.06cm);
       \draw [thick,black] (0,0) -- (6.055,-3.5);
       \draw [thick,red,->] (0,0) -- (6,0);
       \draw [red,fill] (6,0) circle (0.06cm); 
       \node [left,red,below] at (6,-0.2) {$\gamma$};
       \draw [thick,red,dashed] (-1,0) -- (7,0);
       \node [left,red] at (7.7,0) {$\R\gamma$};
       \draw [thick,black] (0,0) -- (6.055,3.5);
       \draw [thick,red] (2.6,-1.5) arc (-30:30:3cm);
       \draw[thick, red, dashed] (2.12,-2.12) arc (-45:45:3);
       \node [right,red] at (2.3,2.2) {$S^1$};
        \node [left,red] at (2.7,-0.8) {$A$};
        \draw [thick, blue] (3,-1.732) -- (3,1.732); 
        \node [right,blue] at (3.2,-1.2) {$L$};
     \end{tikzpicture}
     \caption{The case $d=2$.}
     \label{fig:d2}
     \endminipage
   \end{figure}  

The only possibilities which can give rise to $1<\beta_{i}<2$ are where
$\theta_0\in(\pi/2,\pi)$ and $i=1$.  In this case $\psi_1$ is odd in
$\theta$, and hence is an odd function when restricted to the line $L$
(see Figure \ref{fig:d=2}).  Since an odd concave function is
necessarily a multiple of the identity function, the only possibility
in which $\psi_1$ has a concave restriction to $L$ is when
\begin{displaymath}
  \psi_1(e_1+ye_2) = cy,
\end{displaymath}
which implies by homogeneity that
\begin{displaymath}
\psi_1(xe_1+ye_2) = cx^{\beta_1-1}y.
\end{displaymath}
However a direct computation shows that this is harmonic only in the
cases $\beta_1=1$ or $\beta_1=2$, which are impossible. This proves
Lemma \ref{lem:section-concave} for the case $d=2$, so we have
established that $v$ has a non-convex superlevel set whenever
$\Omega$ is not a product of circumsolids.  We note that for $d=2$
this applies except when $\Omega$ is either a circumsolid or a
rectangle.

Now we complete the proof of Theorem \ref{thm:main1bis}, by proving
that the Robin eigenfunction $u_\alpha$ also has some non-convex
superlevel sets for sufficiently small $\alpha>0$:

By Proposition~\ref{propo:properties-of-Robineigenvalues}, for all
sufficiently small $\alpha>0$, the first Robin eigenfunction
$u_{\alpha}$ is given by
  \begin{equation}
    \label{eq:37}
    u_{\alpha}=\mathds{1}+\alpha v+f^{\alpha}
  \end{equation}
  where $f^{\alpha}$ is  $o(\alpha)$ in
  $C^{0,\beta}(\overline{\Omega})$ for some $\beta\in (0,1)$.

  We have proved that $v$ has some non-convex superlevel sets, which
  means that there exist points $x_1$ and $x_2$ in $\Omega$, a number
  $c\in\R$, and $\varepsilon>0$ such that $v(x_i)>c+\varepsilon$ for
  $i=1,2$, but $v\left(\frac{x_1+x_2}{2}\right)<c-\varepsilon$.  But
  then we have by \eqref{eq:37} for $\alpha$ sufficiently small that
 \begin{displaymath}
   u_\alpha(x_i) = 1+\alpha v(x_i)+o(\alpha)>1+\alpha c+\alpha \varepsilon+o(\alpha)>1+\alpha c
\end{displaymath}
 for $i=1,2$, while
 \begin{displaymath}
   u_\alpha\left(\frac{x_1+x_2}{2}\right) = 1+\alpha
   v\left(\frac{x_1+x_2}{2}\right)
   +o(\alpha)<1+\alpha c-\alpha \varepsilon+o(\alpha)<1+\alpha c.
\end{displaymath}
It follows that the superlevel set $\{x\,\vert\, u_\alpha(x)>1+\alpha c\}$
is not convex for sufficiently small $\alpha>0$.
\end{proof}

It remains to give the proof of Corollary~\ref{cor:approx-omega}.

\begin{proof}[Proof of Corollary~\ref{cor:approx-omega}]
  It suffices to show the following: If $\Omega$ is a convex domain
  for which the Robin ground state $u_\alpha(\Omega)$ is
  not log-convex (or has a non-convex superlevel set) for some
  $\alpha$, and $\{\Omega_n\}$ is a sequence of convex domains which
  approach $\Omega$ in Hausdorff distance, then the Robin
  eigenfunction $u_{\alpha,n}$ of $\Omega_n$ is not log-concave
  (respectively, has a non-convex superlevel set) for sufficiently
  large $n$.

  We apply Proposition \ref{propo:domain-perturbation}, which applies
  since the volume and perimeter of convex sets are continuous with
  respect to Hausdorff distance.  In particular, by \eqref{eq:16} the
  eigenfunctions $u_{n,\alpha}$ converge uniformly to $u_\alpha$ on
  any subset which is contained in $\overline{\Omega}_n$ for all large
  $n$.

  Under the assumption that $u_\alpha$ is not log-concave on $\Omega$,
  there exist points $x_1$ and $x_2$ in $\Omega$ such that
  $\frac12(\log u_\alpha(x_1)+\log u_\alpha(x_2))>\log
  u_\alpha\left(\frac{x_1+x_2}{2}\right)$, or equivalently
  $u_\alpha(x_1)u_\alpha(x_2)>u_\alpha\left(\frac{x_1+x_2}{2}\right)^2$.
  For sufficiently large $n$ the points $x_1$, $x_2$ and
  $\frac{x_1+x_2}{2}$ are all contained in $\Omega_n$, and hence we
  have
\begin{displaymath}
  u_{\alpha,n}(x_1)u_{\alpha,n}(x_2)-u_{\alpha,n}\left(\frac{x_1+x_2}{2}\right)^2
  \to 
  u_\alpha(x_1)u_\alpha(x_2)-u_\alpha\left(\frac{x_1+x_2}{2}\right)^2>0
\end{displaymath}
as $n\to\infty$, and hence the left-hand side is positive for
sufficiently large $n$, proving that $u_{\alpha,n}$ is not log-concave
for $n$ large.

Similarly, under the assumption that $u_\alpha$ has a non-convex
superlevel set, there exist points $x_1,x_2$ in $\Omega$ and $c\in\R$
such that $u_\alpha(x_i)>c$ for $i=1,2$, while
$u_\alpha\left(\frac{x_1+x_2}{2}\right)<c$.  As before the convergence
of $u_{\alpha,n}$ to $u_\alpha$ at the points $x_1$, $x_1$ and
$\frac{x_1+x_2}{2}$ guarantees that $u_{\alpha,n}(x_i)>c$ for $i=1,2$
and $u_{\alpha,n}\left(\frac{x_1+x_2}{2}\right)<c$ for $n$
sufficiently large, proving that $u_{\alpha,n}$ has a non-convex
superlevel set.
\end{proof}

\section{Final Discussions and conjectures}

We conclude this paper by formulating some interesting
observations and conjectures.\medskip 

We recall that the Dirichlet eigenvalue problem corresponds to the limiting
case $\alpha\to+\infty$ in which it is well-known (cf~\cite{MR0450480})
that the first eigenfunction is log-concave. Thus, our first
conjecture is naturally:
    
\begin{center}
  \begin{minipage}[c]{0.95\linewidth}
    {\bfseries $\mbox{}\quad$ 1. Conjecture. }\emph{For a given bounded convex domain
      $\Omega$, there is an $\alpha_{0}>0$ such that for all $\alpha\ge
      \alpha_{0}$, the first Robin
      eigenfunction $u_{\alpha}$ is log-concave.}\\[-7pt]
  \end{minipage}
\end{center}

Furthermore, it would be interesting to know whether the threshold
$\alpha_{0}$ depends on the dimension $d\ge 2$ and whether it can be
independent of the domain $\Omega$.\medskip

Let $\Omega$ be a convex polyhdral domain that is
not the product of circumsolids. In order to prove in dimensions
$d\ge 3$ that the first Robin eigenfunction $u_{\alpha}$ has
non-convex superlevel sets without imposing the stronger hypothesis
\emph{$\Omega$ has inconsistent normals at some boundary point}, our proof of
Theorem~\ref{thm:main1bis} shows that one needs to study the second
case when the harmonic function $\hat{w}$ given by
Lemma~\ref{lem:existence-of-homog-1-function} is
\emph{linear}. The linear case in dimension $d=2$ is much simpler to
treat than the $(d-1)$-dimensional hyperplane
   \begin{math}
     \mathcal{H}:=\{x\in \R^{d}\,\vert\, x\cdot\gamma=|\gamma|\}
   \end{math} 
reduces to a line segment $L$. Nevertheless, we are convinced that the
following conjecture holds. \\[-7pt]

\begin{center}
  \begin{minipage}[c]{0.95\linewidth}
    {\bfseries $\mbox{}\quad$ 2. Conjecture. }\emph{If $\Omega$ is a convex
      polyhedron in $\R^{d}$ for $d\ge 3$ which is not a product of
      circumsolids, then for sufficiently small $\alpha>0$, the first
      Robin eigenfunction $u_{\alpha}$ has non-convex superlevel
      sets.}\\[-7pt]
  \end{minipage}
\end{center}

Our argument shows that it would be sufficient to establish Lemma
\ref{lem:section-concave} whenever $\Gamma$ is a polyhedral convex
cone which is a circumsolid about the point $\gamma$, and $\psi_1$ is
a homogeneous harmonic function with Neumann boundary conditions on
$\Gamma$ with degree of homogeneity between $1$ and $2$ (see Remark~\ref{rem:not-able-to-proof}).

Our initial motivation for the work undertaken in this paper was to establish
the {\bfseries fundamental gap conjecture for
Robin eigenvalues}: \\[-7pt]

\begin{center}
  \begin{minipage}[c]{0.95\linewidth}
    \emph{Let $\Omega$ be a bounded convex domain in $\R^{d}$ of
      diameter $D$, $V$ be a weakly convex potential, and for
      $\alpha>0$, let $\lambda_{i}(\alpha)$ be the Robin eigenvalues
      on the interval $(-\tfrac{D}{2},\tfrac{D}{2})$. Then for
      $\alpha>0$, the Robin eigenvalues $\lambda_{i}^{\! V}(\alpha)$
      of the Schr\"odinger operator $-\Delta+V$ satisfy
      \begin{displaymath}
        \lambda_{1}^{\! V}(\alpha)-\lambda_{0}^{\! V}(\alpha)
        \ge \lambda_{1}(\alpha)-\lambda_{0}(\alpha).
      \end{displaymath}}\\[-7pt]
  \end{minipage}
\end{center}

In the Dirichlet case 
this conjecture was first
observed by van den Berg~\cite{MR711491} and then later independently
suggested by Ashbaugh and Benguria~\cite{MR942630}, and
Yau~\cite{MR865650}. The complete proof of the fundamental gap
conjecture in this case was given
in~\cite{fundamental}. Theorem~\ref{Robin gap} is a first attempt to
prove the fundamental gap conjecture for Robin eigenvalues, but
provides non-optimal lower bounds. But due to our main
Theorem~\ref{thm:main1}, it is clear that this conjecture can only be
proved by methods avoiding the log-concavity of the first Robin
eigenfunction. To the best of our knowledge, only Lavine's work~\cite{MR1185270}
provides a proof of the fundamental gap conjecture which
does not use the log-concavity of the first eigenfunction. That paper
concerns the Dirichlet and Neumann case on a bounded
interval. With this in mind, we conclude
with the following question: \\[-7pt]

\begin{center}
  \begin{minipage}[c]{0.95\linewidth}
    {\bfseries Open problem.\; } \emph{How can one prove the fundamental
      gap conjecture for Robin eigenvalues without using the
      log-concavity of the first eigenfunction?}
  \end{minipage}
\end{center}

%
%

\bibliographystyle{plain}
\bibliography{citations}

\def\cprime{$'$}
\begin{thebibliography}{10}

\bibitem{fundamental}
Ben Andrews and Julie Clutterbuck.
\newblock Proof of the fundamental gap conjecture.
\newblock {\em J. Amer. Math. Soc.}, 24(3):899--916, 2011.

\bibitem{AM}
Tom~M. Apostol and Mamikon~A. Mnatsakanian.
\newblock Solids circumscribing spheres.
\newblock {\em Amer. Math. Monthly}, 113(6):521--540, 2006.

\bibitem{MR942630}
Mark~S. Ashbaugh and Rafael Benguria.
\newblock Optimal lower bound for the gap between the first two eigenvalues of
  one-dimensional {S}chr\"odinger operators with symmetric single-well
  potentials.
\newblock {\em Proc. Amer. Math. Soc.}, 105(2):419--424, 1989.

\bibitem{MR0450480}
Herm~Jan Brascamp and Elliott~H. Lieb.
\newblock On extensions of the {B}runn-{M}inkowski and {P}r\'ekopa-{L}eindler
  theorems, including inequalities for log concave functions, and with an
  application to the diffusion equation.
\newblock {\em J. Functional Analysis}, 22(4):366--389, 1976.

\bibitem{MR2759829}
Haim Brezis.
\newblock {\em Functional analysis, {S}obolev spaces and partial differential
  equations}.
\newblock Universitext. Springer, New York, 2011.

\bibitem{MR3566212}
Dorin Bucur, Alessandro Giacomini, and Paola Trebeschi.
\newblock The {R}obin-{L}aplacian problem on varying domains.
\newblock {\em Calc. Var. Partial Differential Equations}, 55(6):Paper No. 133,
  29, 2016.

\bibitem{MR768584}
Isaac Chavel.
\newblock {\em Eigenvalues in {R}iemannian geometry}, volume 115 of {\em Pure
  and Applied Mathematics}.
\newblock Academic Press, Inc., Orlando, FL, 1984.

\bibitem{MR639355}
Shiu~Yuen Cheng and Peter Li.
\newblock Heat kernel estimates and lower bound of eigenvalues.
\newblock {\em Comment. Math. Helv.}, 56(3):327--338, 1981.

\bibitem{CNL}
Bennett Chow, Peng Lu, and Lei Ni.
\newblock {\em Hamilton's {R}icci flow}, volume~77 of {\em Graduate Studies in
  Mathematics}.
\newblock American Mathematical Society, Providence, RI; Science Press Beijing,
  New York, 2006.

\bibitem{MR1072395}
Jos\'e~F. Escobar.
\newblock Uniqueness theorems on conformal deformation of metrics, {S}obolev
  inequalities, and an eigenvalue estimate.
\newblock {\em Comm. Pure Appl. Math.}, 43(7):857--883, 1990.

\bibitem{MR3238844}
Alexey Filinovskiy.
\newblock On the eigenvalues of a {R}obin problem with a large parameter.
\newblock {\em Math. Bohem.}, 139(2):341--352, 2014.

\bibitem{MR1814364}
David Gilbarg and Neil~S. Trudinger.
\newblock {\em Elliptic partial differential equations of second order}.
\newblock Classics in Mathematics. Springer-Verlag, Berlin, 2001.
\newblock Reprint of the 1998 edition.

\bibitem{MR1335452}
Tosio Kato.
\newblock {\em Perturbation theory for linear operators}.
\newblock Classics in Mathematics. Springer-Verlag, Berlin, 1995.
\newblock Reprint of the 1980 edition.

\bibitem{KennedyPhD2010}
J.~B. Kennedy.
\newblock {\em On the Isoperimetric Problem for the Laplacian with Robin and
  Wentzell Boundary Conditions}.
\newblock PhD thesis, School of Mathematics and Statistics, University of
  Sydney, NSW 2006, Australia, 2010.

\bibitem{MR1185270}
Richard Lavine.
\newblock The eigenvalue gap for one-dimensional convex potentials.
\newblock {\em Proc. Amer. Math. Soc.}, 121(3):815--821, 1994.

\bibitem{MR2527916}
Giovanni Leoni.
\newblock {\em A first course in {S}obolev spaces}, volume 105 of {\em Graduate
  Studies in Mathematics}.
\newblock American Mathematical Society, Providence, RI, 2009.

\bibitem{MR2812574}
Robin Nittka.
\newblock Regularity of solutions of linear second order elliptic and parabolic
  boundary value problems on {L}ipschitz domains.
\newblock {\em J. Differential Equations}, 251(4-5):860--880, 2011.

\bibitem{MR0117419}
L.~E. Payne and H.~F. Weinberger.
\newblock An optimal {P}oincar\'e inequality for convex domains.
\newblock {\em Arch. Rational Mech. Anal.}, 5:286--292 (1960), 1960.

\bibitem{SWYY}
I.~M. Singer, Bun Wong, Shing-Tung Yau, and Stephen S.-T. Yau.
\newblock An estimate of the gap of the first two eigenvalues in the
  {S}chr\"odinger operator.
\newblock {\em Ann. Scuola Norm. Sup. Pisa Cl. Sci. (4)}, 12(2):319--333, 1985.

\bibitem{smits1996spectral}
Robert~G Smits.
\newblock Spectral gaps and rates to equilibrium for diffusions in convex
  domains.
\newblock {\em The Michigan Mathematical Journal}, 43(1):141--157, 1996.

\bibitem{MR711491}
M.~van~den Berg.
\newblock On condensation in the free-boson gas and the spectrum of the
  {L}aplacian.
\newblock {\em J. Statist. Phys.}, 31(3):623--637, 1983.

\bibitem{MR865650}
Shing-Tung Yau.
\newblock {\em Nonlinear analysis in geometry}, volume~33 of {\em Monographies
  de L'Enseignement Math\'ematique [Monographs of L'Enseignement
  Math\'ematique]}.
\newblock L'Enseignement Math\'ematique, Geneva, 1986.
\newblock S\'erie des Conf\'erences de l'Union Math\'ematique Internationale
  [Lecture Series of the International Mathematics Union], 8.

\bibitem{YuZhong}
Qi~Huang Yu and Jia~Qing Zhong.
\newblock Lower bounds of the gap between the first and second eigenvalues of
  the {S}chr\"odinger operator.
\newblock {\em Trans. Amer. Math. Soc.}, 294(1):341--349, 1986.

\bibitem{ZhongYang}
Jia~Qing Zhong and Hong~Cang Yang.
\newblock On the estimate of the first eigenvalue of a compact {R}iemannian
  manifold.
\newblock {\em Sci. Sinica Ser. A}, 27(12):1265--1273, 1984.

\end{thebibliography}

\end{document}